\documentclass{article}
\usepackage[latin1]{inputenc}
\usepackage[T1]{fontenc}
\usepackage{amsmath,amssymb}
\usepackage{amsfonts,amsthm}
\usepackage{bbm}
\usepackage{mathrsfs}
\usepackage{geometry}
\usepackage{color}
   \definecolor{cites}{rgb}{0.75 , 0.00 , 0.00}  
   \definecolor{urls} {rgb}{0.00 , 0.00 , 1.00}  
   \definecolor{links}{rgb}{0.00 , 0.00 , 0.5}   
  \definecolor{gray}{rgb}{0.5,0.5,.5}
\usepackage{graphicx}
\usepackage{cancel}
\usepackage[plainpages=false,colorlinks=true, citecolor=cites, urlcolor=urls, linkcolor=links, breaklinks=true]{hyperref}

\usepackage{soul}
\usepackage{comment}
\usepackage{subcaption}

\definecolor{amcol}{rgb}{0.8,0,0}

\definecolor{escol}{rgb}{0,0,0.8}
\definecolor{estcol}{rgb}{0,0.6,0}

\definecolor{cwcol}{rgb}{0.5,0,0.5}
\definecolor{cwstcol}{rgb}{0,0.6,0.6}

\newcommand{\C}{\mathbb{C}}

\newcommand{\N}{\mathbb{N}}
\newcommand{\R}{\mathbb{R}}
\newcommand{\T}{\mathbb{T}}
\newcommand{\Z}{\mathbb{Z}}

\newcommand{\rd}{\mathrm{d}}
\newcommand{\pdiff}[2]{\frac{\partial #1}{\partial #2}}
\newcommand{\half}{\frac{1}{2}}
\newcommand{\tand}{\text{ and }}
\newcommand{\nxy}{|x-y|}
\newcommand{\LtG}{{L^2(\Gamma)}}

\newcommand{\Tc}{\mathcal{T}}

\newcommand{\cD}{\mathscr{D}}
\newcommand{\cA}{\mathcal{A}}
\newcommand{\cS}{\mathscr{S}}
\newcommand{\cG}{\mathcal{G}}
\newcommand{\cF}{\mathcal{F}}
\newcommand{\cC}{\mathcal{C}}
\newcommand{\cL}{\mathcal{L}}
\newcommand{\cR}{\mathcal{R}}

\newcommand{\fF}{\mathfrak{F}}
\newcommand{\fG}{\mathfrak{G}}
\newcommand{\fI}{\mathfrak{I}}

\newcommand{\fK}{\mathfrak{K}}

\newcommand\toH{
   \unitlength0.1ex
   \begin{picture}(30,15)
   \put(13,16){\makebox(0,0)[]{\tiny\rm H}}
   \put(15,5){\makebox(0,0)[]{$\to$}}
   \end{picture}
}

\DeclareMathOperator{\spec}{\sigma}
\DeclareMathOperator{\ess}{ess}

\DeclareMathOperator{\supp}{supp}
\DeclareMathOperator{\arsinh}{arsinh}
\DeclareMathOperator*{\esssup}{ess \, sup}
\DeclareMathOperator*{\conv}{conv}
\DeclareMathOperator*{\dist}{dist}

\DeclareMathOperator{\Real}{Re}
\renewcommand{\Re}{\Real}
\DeclareMathOperator{\Imag}{Im}
\renewcommand{\Im}{\Imag}

\newcommand{\from}{\colon}

\providecommand{\scpr}[2]{\left\langle #1, #2 \right\rangle}
\renewcommand{\sp}{\scpr}
\providecommand{\abs}[1]{\left\lvert#1\right\rvert}
\providecommand{\norm}[1]{\left\lVert#1\right\rVert}
\providecommand{\set}[1]{\left\{ #1\right\}}

\newtheorem{thm}{Theorem}[section]
\newtheorem*{thm*}{Theorem}

\newtheorem{lem}[thm]{Lemma}
\newtheorem{prop}[thm]{Proposition}
\newtheorem{cor}[thm]{Corollary}
\newtheorem*{cor*}{Corollary}

\theoremstyle{definition}
\newtheorem{defn}[thm]{Definition}
\newtheorem{conject}[thm]{Conjecture}
\newtheorem{quest}[thm]{Question}

\theoremstyle{remark}
\newtheorem{rem}[thm]{Remark}

\numberwithin{equation}{section}

\renewcommand{\epsilon}{\varepsilon}

\allowdisplaybreaks

\definecolor{scnew}{rgb}{0,0.6,0}

\begin{document}

\title{On the spectrum of the double-layer operator on locally-dilation-invariant Lipschitz domains}\author{Simon N. Chandler-Wilde\thanks{Department of Mathematics and Statistics, University of Reading, Whiteknights, PO Box 220, Reading RG6 6AX, UK, {\tt s.n.chandler-wilde@reading.ac.uk, j.a.virtanen@reading.ac.uk}}, Raffael Hagger\thanks{Mathematisches Seminar, Christian-Albrechts-Universit\"at zu Kiel, Heinrich-Hecht-Platz 6, 24118 Kiel, Germany,
{\tt hagger@math.uni-kiel.de}}, Karl-Mikael Perfekt\thanks{Department of Mathematical Sciences, Norwegian University of Science and Technology (NTNU),
NO-7491 Trondheim, Norway, {\tt karl-mikael.perfekt@ntnu.no}}, Jani A. Virtanen\footnotemark[1]}

\maketitle

\begin{abstract}
We say that $\Gamma$, the boundary of a bounded Lipschitz domain, is locally dilation invariant if, at each $x\in \Gamma$, $\Gamma$ is either locally $C^1$ or locally coincides (in some coordinate system centred at $x$) with a Lipschitz graph $\Gamma_x$ such that $\Gamma_x=\alpha_x\Gamma_x$, for some $\alpha_x\in (0,1)$. In this paper we study, for such $\Gamma$,
the  essential spectrum of $D_\Gamma$, the double-layer (or Neumann-Poincar\'e) operator of potential theory, on $L^2(\Gamma)$.
We show, via localisation and Floquet-Bloch-type arguments, that this essential spectrum
is the union of the spectra of related continuous families of operators $K_t$, for $t\in [-\pi,\pi]$; moreover, each $K_t$ is compact if $\Gamma$ is $C^1$ except at finitely many points. For the 2D case where, additionally, $\Gamma$ is piecewise analytic, we construct convergent sequences of approximations to the essential spectrum of $D_\Gamma$; each approximation is the union of the eigenvalues of finitely many finite matrices arising from Nystr\"om-method approximations to the operators $K_t$.  Through error estimates with explicit constants, we also construct functionals that determine whether any particular locally-dilation-invariant piecewise-analytic $\Gamma$ satisfies the well-known
spectral radius conjecture, that the essential spectral radius of $D_\Gamma$ on $L^2(\Gamma)$ is $<1/2$ for all Lipschitz $\Gamma$.  We illustrate this theory with examples; for each we show that the essential spectral radius {\em is} $<1/2$, providing additional support for the conjecture. We also, via new results on the invariance of the essential spectral radius under  locally-conformal $C^{1,\beta}$ diffeomorphisms, show that the spectral radius conjecture holds for all Lipschitz curvilinear polyhedra.
\end{abstract}

\section{Introduction} \label{sec:intro}

Given a bounded Lipschitz domain\footnote{For us, as, e.g., in \cite{MitreaD:97}, ``domain'' will just mean ``open set''; a domain need not be connected.} $\Omega_- \subset \R^d, d \geq 2$, with boundary $\Gamma$ and outward-pointing unit normal vector $n$, the interior and exterior Dirichlet and Neumann problems for Laplace's equation (posed in $\Omega_-$ and in $\Omega_+:=\R^d\setminus \overline{\Omega_-}$, respectively), can be reformulated as boundary integral equations involving the operators
\begin{equation}\label{eq:secondkindBIEs}
D_\Gamma\pm \textstyle{\half} I \quad\tand\quad D'_\Gamma \pm \textstyle{\half} I
\end{equation}
(see, e.g., \cite{Ver84}, \cite[\S5.9, \S5.15.1]{Me:18}), where the \emph{double-layer (or Neumann or Neumann-Poincar\'e) operator} $D_{\Gamma}$ and the \emph{adjoint double-layer operator} $D_{\Gamma}'$ are defined by
\begin{equation}\label{eq:DD'}
D_{\Gamma} \phi(x) = \int_\Gamma \pdiff{\Phi(x,y)}{n(y)} \phi(y)\, \rd s(y) \quad \tand\quad
D_{\Gamma}' \phi(x) = \int_\Gamma \pdiff{\Phi(x,y)}{n(x)} \phi(y)\, \rd s(y),
\end{equation}
for  $\phi\in \LtG$ and (almost all) $x \in \Gamma$, with the integrals understood, in general, as Cauchy principal values. Here $\Phi(x,y)$ is the fundamental solution for Laplace's equation, defined by\footnote{Our sign convention and normalisation are those of many authors (e.g.\ \cite{St:08,Kr:14}), but other authors (e.g.\ \cite{Ken}), use a fundamental solution that is the negative of ours.}
\begin{equation}\label{eq:fund}
\Phi(x,y):= \frac{1}{2\pi} \log \left(\frac{1}{\nxy}\right), \;\; d= 2, \;\;\quad := \frac{1}{(d-2)c_d}\frac{1}{\nxy^{d-2}}, \;\; d \geq 3,
\end{equation}
where $c_d$ is the surface measure of the unit sphere in $\R^d$. Explicitly,
\begin{equation}\label{eq:DD'2}
D_{\Gamma} \phi(x) = \frac{1}{c_d}\int_\Gamma\frac{(x-y) \cdot n(y)}{|x-y|^d} \phi(y)\, \rd s(y) \quad \tand\quad
D_{\Gamma}' \phi(x) = \frac{1}{c_d}\int_\Gamma\frac{(y-x) \cdot n(x)}{|x-y|^d} \phi(y)\, \rd s(y),
\end{equation}
for  $\phi\in \LtG$ and (almost all) $x \in \Gamma$.

Complementing \eqref{eq:secondkindBIEs},
\begin{equation} \label{eq:trans}
 D_{\Gamma}'-\lambda I,
\end{equation}
with $\lambda\in \C\setminus \{-\half,\half\}$, arises as the operator in the boundary integral equation reformulation of transmission problems in electrostatics, where the Laplace equation $\Delta u = 0$ holds in $\Omega_-$ and $\Omega_+$ and the trace of $u$ or its normal derivative jumps across $\Gamma$ (see, e.g., \cite[\S5.12]{Me:18}). In this context the case $|\lambda|\geq 1/2$, especially with $\lambda$ real, is classically of interest (e.g., \cite[\S5.12]{Me:18}); more recently the case where $\lambda$ is complex with $|\lambda|< 1/2$ has been studied intensively as a model of quasi-static electromagnetic plasmonic problems (e.g., \cite{AmDeMi:16,HePe:17,Schnitzer:20,LeCoPer:22}).

Motivated by these physical applications, and by questions in harmonic analysis, there has been long-standing interest in the computation of the spectrum and essential spectrum\footnote{Given a Banach space $Y$ and a bounded linear operator $T\from Y\to Y$ we denote the spectrum of $T$, the set of $\lambda\in \C$ for which $T-\lambda I$ is not invertible, by $\sigma(T;Y)$, and the essential spectrum, the set of $\lambda$ for which $T-\lambda T$ is not Fredholm, by $\sigma_{\ess}(T;Y)$, abbreviating these by $\sigma(T)$ and $\sigma_{\ess}(T)$ where the Banach space $Y$ is clear from the context.} of $D_\Gamma$ as an operator on a variety of function spaces, especially for non-smooth domains (e.g., \cite{Kral65,FaSaSe:92,Ken,Mi:99,ChLe08,QiNi:12,AmDeMi:16,HePe:17,Pe:19}). The largest part of this literature is concerned specifically with the 2D/3D cases where $\Gamma$ is a (curvilinear) polygon (e.g., \cite{Sh:69,Ch:84,Sh:91,Mi:02,MiOtTu:17}) or polyhedron (e.g., \cite{Ra:92,El:92,Mi:99,GM13,HP13,MiOtTu:17,HePe:17,LeCoPer:22}). In this paper we will study and compute the essential spectrum of $D_\Gamma$ as an operator on $L^2(\Gamma)$ for a substantially larger class of boundaries, namely for the case where the boundary $\Gamma$ is {\em locally dilation invariant} in the sense of Definition \ref{defn:ldi} below. In 2D (3D) this class includes polygons (polyhedra) but it also admits much wilder boundary behaviour  (e.g., Figure \ref{fig:exLDI}) as we discuss next in \S\ref{sec:src_main}.

\subsection{The spectral radius conjecture and the main question we address} \label{sec:src_main}
Given a bounded linear operator $T\from Y\to Y$ on a Banach space $Y$, we define
its {\em spectral radius}, $\rho(T;Y)$, and its {\em essential spectral radius}, $\rho_{\ess}(T;Y)$, by
\begin{equation} \label{eq:srdef}
\rho(T;Y) := \sup_{\lambda \in \sigma(T;Y)} |\lambda| \quad \mbox{and} \quad \rho_{\ess}(T;Y) := \sup_{\lambda \in \sigma_{\ess}(T;Y)} |\lambda|,
\end{equation}
abbreviating $\rho(T;Y)$ and $\rho_{\ess}(T;Y)$ by $\rho(T)$ and $\rho_{\ess}(T)$, respectively, when the space $Y$ is clear from the context. The analysis and computation we will carry out are motivated by the so-called {\em spectral radius conjecture}.

This conjecture, in the explicit 1994 formulation of Kenig \cite{Ken}, is as follows, where $L^2_0(\Gamma):= \{\phi\in L^2(\Gamma):\int_\Gamma\phi \,\rd s=0\}$.

\begin{conject} \label{con:kenig} If $\Gamma$ is the boundary of a bounded Lipschitz domain $\Omega_-$  and is connected, the spectral radius of $D_{\Gamma}'$ on $L^2_0(\Gamma)$ is $<\frac{1}{2}$, i.e.\ $\rho(D'_\Gamma;L^2_0(\Gamma)) < \half$.
\end{conject}
\noindent In Section~\ref{sec:specrad} we will
discuss the following alternative formulation
of the conjecture which makes sense regardless of the connectivity of $\Gamma$, and show  its equivalence with Conjecture \ref{con:kenig}.

\begin{conject} \label{con:kenigmod} If $\Gamma$ is the boundary of a bounded Lipschitz domain $\Omega_-$, the essential spectral radius of $D_{\Gamma}$ on $L^2(\Gamma)$ is $<\frac{1}{2}$, i.e.\ $\rho_{\ess}(D_\Gamma;L^2(\Gamma)) < \half$.
\end{conject}

The spectral radius conjecture is very well studied, owing to its intrinsic interest in harmonic analysis, its possible relevance for computation\footnote{Notably, if the spectral radius conjecture holds then the Neumann series for $(D'_\Gamma\pm \half I)\phi=g$, equivalently the Neumann iteration $\pm \half \phi^{(n)} =g- D'_\Gamma\phi^{(n-1)}$, $n\in \N$, converges in $L^2_0(\Gamma)$.}, and its immediate role within electrostatics, as well as in interpretations of electrodynamical problems \cite{HP13}. It
originated in the setting of continuous functions in $\overline{\Omega_{-}}$, dating all the way back to C.~Neumann in the late 1800s, who treated convex domains, and to Radon \cite{Rad19}, who famously analyzed curves of bounded rotation. In a tour de force, Kr\'al \cite{Kral65} completely characterized when the essential norm of $D_\Gamma$
is $< 1/2$. This result was extended into higher dimensions by Burago and Maz'ya \cite{BG67} and Netuka \cite{Net74}. For polyhedra in 3D, the essential norm can be $ > 1/2$. Nevertheless, Rathsfeld \cite{Ra:92} (and see \cite{Ra:95}) and Grachev and Maz'ya \cite{GM13} independently
proved the spectral radius conjecture in the continuous setting holds for general polyhedra. These results were extended to locally conformal deformations of polyhedra by Medkov\'a \cite{Med97}. Even when specialised to the continuous setting, the history is vast, and we refer to \cite{We:09} for an in-depth survey.

As modern harmonic analysis developed, the natural setting for the double-layer potential shifted toward $L^2(\Gamma)$. We make particular mention of the demonstrations of $L^2$-boundedness of the Cauchy integral due to Calder\'on \cite{Cald77} and Coifman, Mcintosh, and Meyer \cite{CMM82}, and Verchota's \cite{Ver84} application of these results to study invertibility of $D'_\Gamma\pm \half I$ on $L^2(\Gamma)$ and $L^2_0(\Gamma)$ when $\Gamma$ is connected. Since then, a flurry of activity and findings in this area have provided support for Kenig's conjecture, though a complete proof has proved elusive.

Indeed, to the best of our knowledge, Conjecture \ref{con:kenig} has been established (only) in the following cases: a) $\Omega_-$ is convex \cite{FaSaSe:92} (and see \cite{ChLe08} for extensions to locally convex domains); b) $\Omega_-$ has small Lipschitz character\footnote{See, e.g., \cite[Definition 3.1]{ChaSpe:21} for the definition of the Lipschitz character of a Lipschitz domain.} \cite{Mi:99}, a case which includes all $C^1$ domains \cite{FaJoRi:78}; c) $\varepsilon$-regular Semmes--Kenig--Toro domains $\Omega_-$ for sufficiently small $\varepsilon > 0$ \cite{HMT10}, including in particular all domains whose gradient has vanishing mean oscillation \cite{Hof94}; d) $\Omega_-$ is a polygon or curvilinear polygon in 2D \cite{Sh:69,Sh:91}, or a Lipschitz polyhedron in 3D \cite{El:92}.

\begin{figure}[ht]
\centering
\begin{subfigure}[t]{.47\textwidth}
  \centering
\includegraphics[scale=0.5, trim = 0.5cm 1.9cm 0cm 1.5cm, clip]{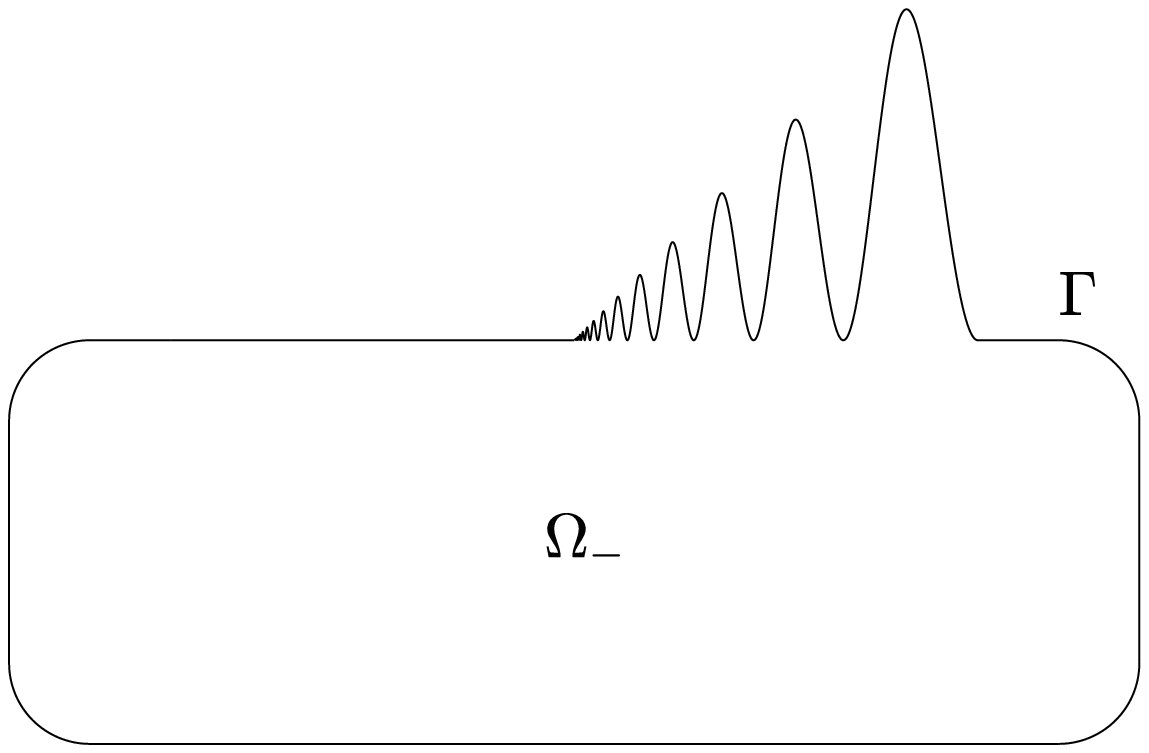}
\subcaption{}
\end{subfigure}
\begin{subfigure}[t]{.47\textwidth}
  \centering
\includegraphics[scale=0.5, trim = 0.5cm 1.9cm 0cm 1.5cm, clip]{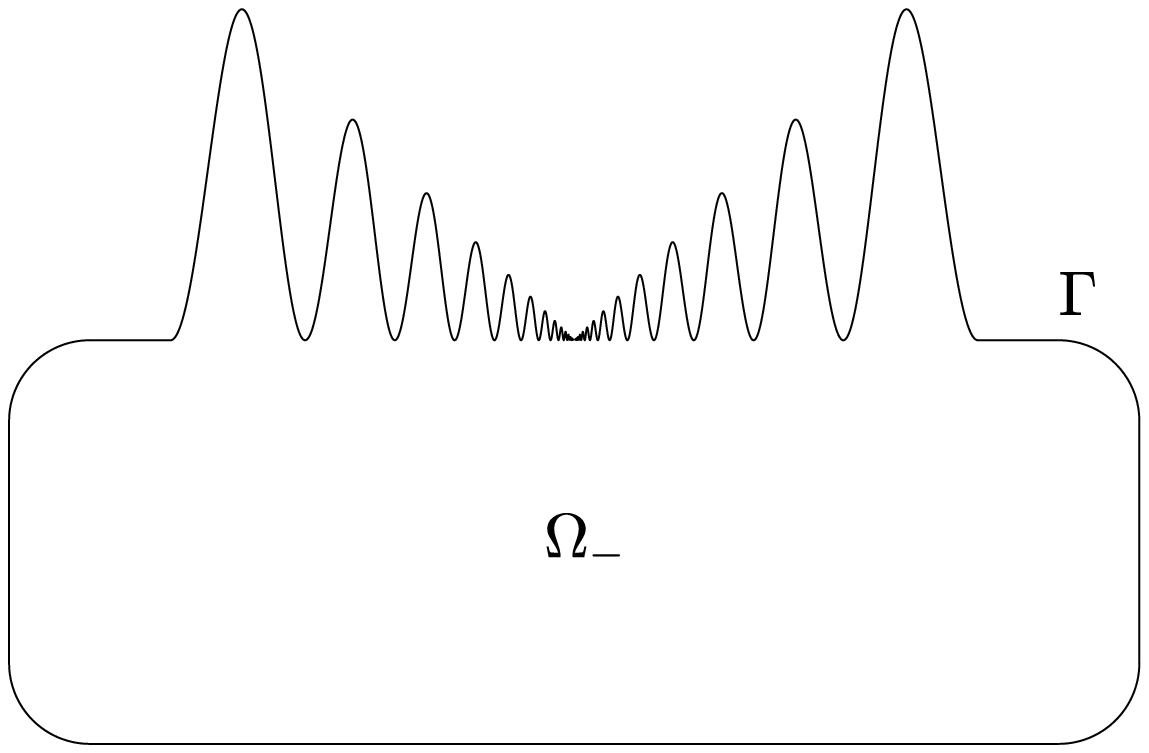}
\subcaption{}
\end{subfigure}
\caption{Examples of bounded Lipschitz domains $\Omega_-$ with boundaries $\Gamma\in \mathscr{D}_\cA \subset \mathscr{D}$ that are piecewise analytic locally dilation invariant.} \label{fig:exLDI}
\end{figure}

Note that polygonal and polyhedral boundaries $\Gamma$ are locally invariant under all dilations: at every point $x\in \Gamma$, $\Gamma$ locally coincides with a graph $\Gamma_x$ such that (in some local coordinate system centred at $x$) $\Gamma_x = \alpha \Gamma_x$ {\em for all} $\alpha>0$. In this paper we will investigate domains where, locally, the dilation invariance $\Gamma_x = \alpha_x \Gamma_x$ only holds {\em for one} $\alpha_x \in (0,1)$  (in which case we say that $\Gamma$ is {\em locally dilation invariant at $x$}).
Precisely our focus will be on the following class of boundaries.
\begin{defn}[\bf{\em Locally dilation invariant}] \label{defn:ldi}
Given $\Gamma$, the boundary of a Lipschitz domain $\Omega_-$, we say that $\Gamma \in \mathscr{D}$, the set of {\em locally-dilation-invariant} boundaries if, at every $x \in \Gamma$, $\Gamma$ is either locally $C^1$ or locally coincides (in some coordinate system centred at $x$)  with a Lipschitz graph $\Gamma_x$ which is dilation invariant with respect to some $\alpha_x \in (0,1)$, i.e.\ $\Gamma_x = \alpha_x \Gamma_x$.
\end{defn}
Note that, already in 2D, $\mathscr{D}$ is a hugely larger class of domains than that of the curvilinear polygons.
Indeed, for a Jordan curve $\Gamma \subset \mathbb{R}^2$ and $x \in \Gamma$, let
\[
v_\Gamma(x) = \int_{0}^{2\pi} |\{x + re^{i\theta} \in \Gamma \, : \, r > 0 \}| \, \frac{d\theta}{2\pi},
\]
where $\abs{\cdot}$ in this equation denotes the counting measure of a set. Then $D_\Gamma \phi$ is uniformly continuous on $\Omega_-$ for every $\phi \in C(\Gamma)$ if and only if
\begin{equation} \label{eq:contsetbdd}
\sup_{x\in \Gamma} v_\Gamma(x) < \infty,
\end{equation}
see \cite{BGS66, Kral62, Kral64}.  If $\Gamma$ is Lipschitz and locally dilation invariant at $x$, then it is clear that $v_\Gamma(x) = \infty$, unless $\Gamma$ coincides with two line segments around $x$. It follows that the only curves $\Gamma \in \mathscr{D}$ satisfying \eqref{eq:contsetbdd} are curvilinear polygons. That is, except for curvilinear polygons, our curves exhibit such wild boundary behaviour that it is not  possible to consider $D_\Gamma$ in the setting of continuous boundary data.

In the context of studying the double-layer and related operators, the class of  domains $\cD$ seems to have been first considered in \cite{ChaSpe:21}, where the essential numerical range\footnote{Recall that, for a bounded linear operator $T\from H \to H$ on a Hilbert space $H$, the numerical range of $T$ is $W(T) := \set{\sp{T\phi}{\phi} : \norm{\phi} = 1}$ and its essential numerical range is $W_{\ess}(T) := \bigcap\limits_{K \text{ compact}} \overline{W(T+K)}$.}  of $D_\Gamma$ was studied and, in 2D and 3D, examples of boundaries $\Gamma\in \cD$ with large Lipschitz character were constructed such that
$D_\Gamma \colon L^2(\Gamma) \to L^2(\Gamma)$ has arbitrarily large {\em essential numerical radius},
$$
w_{\ess}(D_\Gamma) := \sup\limits_{z \in W_{\ess}(D_\Gamma)} |z|,
$$
and so also arbitrarily large {\em essential norm},
$$
\|D_\Gamma\|_{\ess} := \inf_{K \mbox{ \small compact}}\|D_\Gamma-K\|,
$$ since, for any bounded operator $T$ on a Hilbert space
\cite[\S1.3]{GuRa:97},
\begin{equation} \label{eq:NormBounds}
\textstyle{\half} \|T\|_{\ess} \leq w_{\ess}(T) \leq \|T\|_{\ess}.
\end{equation}

 The 2D examples in \cite{ChaSpe:21} (see, e.g., \cite[Fig.~3]{ChaSpe:21} and cf.~Figure \ref{fig:exLDI}) with arbitrarily large $w_{\ess}(D_\Gamma)$ are necessarily examples for which \eqref{eq:contsetbdd} fails to hold, since for curvilinear polygons on $L^2(\Gamma)$ it is well-known that
$$
\|D_\Gamma\|_{\ess}= w_{\ess}(D_\Gamma)=\rho(D_\Gamma;L^2(\Gamma))<\textstyle{\half}
$$
(see \cite{Sh:69,Ch:84,Sh:91,Mi:02} for $\|D_\Gamma\|_{\ess}$ and $\rho(D_\Gamma;L^2(\Gamma))$; equality for $w_{\ess}(D_\Gamma)$ follows by \eqref{eq:NormBounds} and since $W_{\ess}(D_\Gamma)\supset \sigma_{\ess}(D_\Gamma;L^2(\Gamma))$ \cite{bonsall1973numerical}).

A natural question, prompted by the examples from \cite{ChaSpe:21} with $w_{\ess}(D_\Gamma)$ arbitrarily large and the spectral radius conjecture, is the following:
\begin{quotation}
\noindent Given that there exist $\Gamma\in \cD$ with $w_{\ess}(D_\Gamma)\gg \half$, in particular such examples in 2D, is there a $\Gamma\in \cD$, in particular an example in 2D, with $\rho_{\ess}(D_\Gamma;L^2(\Gamma))\geq \half$?
\end{quotation}
Of course, a positive answer would provide a counterexample to Conjecture \ref{con:kenigmod} and hence to the original spectral radius conjecture, Conjecture \ref{con:kenig}. The aim of this paper is to address this question, through mathematical analysis and 
computational methods supported by numerical analysis error estimates where constants are made explicit. These will enable us to estimate $\rho_{\ess}(D_\Gamma;L^2(\Gamma))$ for a large class of $\Gamma\in \cD$
sufficiently accurately to determine whether or not $\rho_{\ess}(D_\Gamma;L^2(\Gamma))\geq \half$.

\subsection{Our main results and their significance} \label{sec:main_sig}
The first step in our analysis is Theorem \ref{thm:localization} in \S\ref{sec:localwithout}, the localisation result (cf.~\cite{El:92,Mi:99,LeCoPer:22,ChaSpe:21}) that, for $\Gamma\in \cD$ (and for any dimension $d\geq 2$), there exists a finite set $F\subset \Gamma$ such that
\begin{equation} \label{eq:local}
\sigma_{\ess}(D_\Gamma;L^2(\Gamma)) = \bigcup_{x\in F} \sigma_{\ess}(D_{\Gamma_x};L^2({\Gamma_x})).
\end{equation}
In the above formula $\Gamma_x$ is defined as above if $\Gamma$ is locally dilation invariant at $x$, while if $\Gamma$ is locally $C^1$ at $x$ then $\Gamma_x$ is the graph of a $C^1$ compactly supported function so that \cite{FaJoRi:78} $D_{\Gamma_x}$ is compact and $\sigma_{\ess}(D_{\Gamma_x};L^2({\Gamma_x}))=\{0\}$.

The localisation \eqref{eq:local} reduces the computation of $\sigma_{\ess}(D_\Gamma;L^2(\Gamma))$ to that of $\sigma_{\ess}(D_{\Gamma_x};L^2({\Gamma_x}))$ for finitely many $x$ for which $\alpha_x\Gamma_x = \Gamma_x$, for some $\alpha_x\in (0,1)$. Computation of $\sigma_{\ess}(D_{\Gamma_x};L^2({\Gamma_x}))$ for such $x$ is our focus in \S\ref{sec:dilation_invariant} where we study and compute spectral properties of $D_\Gamma$ in the case that
 $\Gamma \subset \mathbb{R}^d$ is a {\em dilation invariant} Lipschitz graph, meaning that there exists an $\alpha \in (0,1)$ such that $\alpha \Gamma = \Gamma$ (an example is Figure \ref{fig:ex2}). If $\Gamma_0\subset \Gamma\setminus\{0\}$ is a particular relatively closed and bounded Lipschitz subgraph of $\Gamma$ such that
\begin{equation} \label{eq:discr}
\Gamma = \bigcup_{j \in \mathbb{Z}} \alpha^j \Gamma_0,
\end{equation}
and such that $(\alpha^j\Gamma_0)\cap \Gamma_0$ has zero surface measure for $j\neq 0$, we show that $D_\Gamma$ can be written as a discrete $\ell_1$ convolution operator whose entries are bounded linear operators on $L^2(\Gamma_0)$ related to the discretization \eqref{eq:discr} of $\Gamma$. This allows us to decompose $D_\Gamma$, by a Floquet-Bloch transform, into a continuous family $(K_t)_{t \in [-\pi, \pi]}$ of operators $K_t \colon L^2(\Gamma_0) \to L^2(\Gamma_0)$, and, by standard results for such convolutions (see \cite[Thm.~2.3.25]{RaRoSi04}), to characterise the essential spectrum of $D_\Gamma$ as (Theorem \ref{cor:norm_and_spec})
\begin{equation} \label{eq:specKt}
\sigma_{\ess}(D_\Gamma;L^2(\Gamma)) = \sigma(D_\Gamma;L^2(\Gamma)) = \bigcup_{t\in [-\pi,\pi]} \sigma(K_t;L^2(\Gamma_0)).
\end{equation}

This characterisation is particularly useful when, apart from a singularity at 0, $\Gamma$ is $C^1$ (see, e.g., Figure \ref{fig:ex2}), for then (Corollary \ref{cor:K_t_compact}) each $K_t$ is compact, and so has a more easily computed discrete spectrum. The characterisation \eqref{eq:specKt} holds for every Lipschitz dilation invariant graph in any dimension; in \S\ref{sec:two-sided} we specialise to the 2D case where $\Gamma$, except at $0$, is the graph of a real analytic function (see, e.g., Figure \ref{fig:ex2}); we denote by $\cA$ this subset of the 2D dilation invariant graphs. We show, for each $t\in [-\pi,\pi]$, that $K_t$ is unitarily equivalent to $\tilde K_t$, a $2\times 2$ matrix of integral operators on $L^2(0,1)$ that have real analytic kernels. Further, the range of each of these integral operators is a space of $1$-quasi-periodic real analytic functions. As a consequence,
$$
\sigma(K_t;L^2(\Gamma_0)) = \sigma(\tilde K_t;(C[0,1])^2),
$$
and this latter spectrum can be computed by approximating $\tilde K_t$ by a $2N\times 2N$ matrix $\tilde A_{t,N}$ obtained by  a simple midpoint-rule based Nystr\"om discretization. As is well-known (see \cite{Kr:14,TrWe14}, Theorem \ref{thm:Kress}), the midpoint rule is exponentially convergent for periodic analytic functions, so that it follows from Nystr\"om-method spectral estimates for integral operators with continuous kernels \cite{At:75} that, for each $t$, the eigenvalues of $A_{t,N}$ converge at an exponential rate to those of $\tilde K_t$ as $N\to\infty$.

This leads (see Theorem \ref{thm:Hausdorff_convergence_2}) to a Nystr\"om approximation, $\sigma^N(D_\Gamma)$, for $\sigma(D_\Gamma)=\sigma_{\ess}(D_\Gamma)$, which is $\{0\}$ plus the union of the eigenvalues of finitely many $2N\times 2N$ matrices. Our first main result is to show, as Theorem \ref{thm:Hausdorff_convergence_2}, that $\sigma^N(D_\Gamma)\to \sigma(D_\Gamma)$ in the Hausdorff metric as $N\to \infty$; that this convergence is achievable is somewhat surprising given that $D_\Gamma$ is neither compact nor self-adjoint. Our second, and more substantial result (Theorem \ref{thm:spectral_radius2}) is that we develop a fully discrete algorithm to test whether,
as an operator on $L^2(\Gamma)$, $\rho_{\ess}(D_\Gamma)<\half$. Precisely, we construct (see Remark \ref{rem:final}), for each $c>0$, a nonlinear functional $S_{c}:\cA\times \N\times \N\times \N \to \R$  with the properties that: a) the functional can be computed in finitely many arithmetic operations and finitely many evaluations of elementary functions, given finitely many sampled values of $f$ and its first and second derivatives for real arguments, plus bounds on the analytic continuation of $f$ to a neighbourhood of the real line that depends on $c$; b) $\rho_{\ess}(D_\Gamma)<\half$ if, for some $c>0$ and $m,M,N\in \N$, it holds that:
\begin{enumerate}
\item[i)] $\rho(A^M_{t_k,N})< \half$, for $k=1,\ldots,m$, where $t_k:= (k-1/2)\pi/m$ and $A^M_{t_k,N}$ is a specific approximation to $A_{t_k,N}$ depending on the parameter $M$;
\item[ii)] $S_{c}(\Gamma,m,M,N) < 0$.
\end{enumerate}
Conversely, if $\rho_{\ess}(D_\Gamma)<\half$, then, for all sufficiently small $c>0$, i) and ii) hold for all sufficiently large $N,M,m\in \N$.

In \S\ref{sec:num_ex} we bring these results together to address our question at the end of \S\ref{sec:src_main}.
We restrict attention to the following 2D class of domains; this class includes polygons and piecewise-analytic curvilinear polygons, but also wilder boundary behaviour as illustrated in Figure \ref{fig:exLDI}.

\begin{defn}[\bf{\em Piecewise analytic locally dilation invariant}] \label{defn:ldiA}
In the case $d=2$ we say that $\Gamma \in \mathscr{D}_\cA$, the set of {\em piecewise-analytic locally-dilation-invariant} boundaries, if $\Gamma\in \mathscr{D}$ and $\Gamma$ is locally analytic (i.e., is locally the graph of a real-analytic function) at all but finitely many $x\in \Gamma$.
\end{defn}

\noindent Given $\Gamma\in \mathscr{D}_\cA$, if $F\subset \Gamma$ is the finite set of points at which $\Gamma$ is not locally analytic and is locally dilation invariant, it is easy to see that $\Gamma_x\in \cA$ for $x\in F$ and that \eqref{eq:local} holds for this set $F$. Thus (see Theorem \ref{thm:synth})
\begin{equation} \label{eq:spradius}
\Sigma^N(D_\Gamma) := \bigcup_{x\in F} \sigma^N(D_{\Gamma_x}) \to \sigma_{\ess}(D_\Gamma)
\end{equation}
in the Hausdorff metric as $N\to \infty$, and note that $\Sigma^N(D_\Gamma)$ is the union of the eigenvalues of finitely many $2N\times 2N$ matrices. Further, as we discuss in \S\ref{sec:num_ex}, it follows that $\rho_{\ess}(D_\Gamma)<\half$ if these eigenvalues lie within the disc of radius $\half$, i.e.\ if
\begin{equation} \label{eq:RN}
R_N(D_\Gamma):=\max \{|z|: z\in \Sigma^N(D_\Gamma)\} < \textstyle{\half},
\end{equation}
and if also, for some $c>0$,
\begin{equation} \label{eq:scrS}
\mathscr{S}_{c}(\Gamma, N) := \max_{x\in F} S_{c}(\Gamma_x,N,N,N) < 0,
\end{equation}
where $\mathscr{S}_{c}(\Gamma, N)$ can be computed in finitely many arithmetic operations plus finitely many evaluations of elementary functions, given inputs describing each $\Gamma_x$ as discussed above. Conversely, if $\rho_{\ess}(D_\Gamma)<\half$ then, for all sufficiently small $c>0$ and all sufficiently large $N$, $R_N(D_\Gamma)<\half$ and $\mathscr{S}_{c}(\Gamma, N)<0$. Thus, given inputs describing $\Gamma$, our fully discrete algorithms enable us to test, for individual $\Gamma\in \mathscr{D}_\cA$, the validity of Conjecture \ref{con:kenigmod}, i.e.\ whether or not $\rho_{\ess}(D_\Gamma)<\half$, through computation of the eigenvalues of finitely many finite matrices, plus finitely many additional arithmetic operations.

In \S\ref{sec:LD} we prove the localisation result \eqref{eq:local}. We also, in the spirit of Medkov\'a's study in the continuous setting \cite{Med90,Med97}, consider the stability of Conjecture~\ref{con:kenigmod} under locally conformal deformations. We prove that the spectral radius conjecture is independent of such deformations, under the assumption that the absolute value of the kernel of the double-layer potential also defines a bounded operator on $L^2(\Gamma)$. While this additional hypothesis may fail to include the wildest of boundaries, it applies to many domains from $\mathscr{D}$. For example, as we discuss in \S\ref{sec:deform}, it is satisfied by any dilation invariant Lipschitz graph $\Gamma$ whose generating set $\Gamma_0$ is polygonal or polyhedral, in particular it holds if $\Gamma$ is a polyhedron and, in 2D, if $\Gamma\in \mathscr{D}_\cA$. As one consequence (Corollary \ref{cor:CP}), Conjecture~\ref{con:kenigmod} holds for Lipschitz curvilinear polyhedra, because \cite{El:92} it holds for polyhedra; as another (Corollary \ref{cor:final}), if it holds for $\Gamma\in \mathscr{D}_\cA$, then it holds for any locally conformal deformation of $\Gamma$.

To illustrate the above results, and use them to test Conjecture \ref{con:kenigmod} and address our main question from \S\ref{sec:src_main}, we include a range of numerical examples in \S\ref{sec:num_ex}, for piecewise-analytic Lipschitz graphs $\Gamma\in \cA$, and for piecewise-analytic locally-dilation-invariant $\Gamma\in \cD_\cA$ that are the boundaries of bounded Lipschitz domains. For each example we plot an approximation ($\sigma^N(D_\Gamma)$ or $\Sigma^N(D_\Gamma)$, as appropriate) to $\sigma_{\ess}(D_\Gamma)$. Moreover, we employ the algorithms described above to provide convincing numerical evidence that $\rho_{\ess}(D_\Gamma;L^2(\Gamma))<\half$ in every case we examine, including cases where $w_{\ess}(D_\Gamma)$ is significantly $>\half$. These results are evidence that the spectral radius conjecture holds for the class of 2D domains  $\cD_\cA$; we emphasise again that this conjecture has not been studied previously for boundaries in this class, except for the special case of curvilinear polygons.

Let us briefly summarise the remainder of the paper. In \S\ref{sec:specrad} we prove the equivalence of Conjectures \ref{con:kenig} and \ref{con:kenigmod}. In \S\ref{sec:Nystrom}, as a key step to our main results, we derive bounds for the spectral radii of general compact operators in \S\ref{sec:cc} (e.g., Corollary \ref{cor:GspecRMfinal}), specialising to the case of integral operators approximated by the Nystr\"om method in \S\ref{sec:Ny} (e.g., Theorem \ref{thm:GspecRMfinal}). In \S\ref{sec:dilation_invariant} we prove the results noted above on the essential spectrum and essential spectral radius of $D_\Gamma$ in the case when $\Gamma$ is a dilation-invariant Lipschitz graph, with particular focus (\S\ref{sec:oneside}--\S\ref{sec:num_range}) on the 2D piecewise-analytic case. In \S\ref{sec:LD} we prove our localisation results. In \S\ref{sec:num_ex} we bring the earlier results together, in particular to study the spectral radius conjecture for $\Gamma\in \cD_\cA$, and we illustrate our theory by numerical examples.

\section{Formulations of the spectral radius conjecture} \label{sec:specrad}

Our results are related to the spectral radius conjecture of Kenig \cite[Problem 3.2.12]{Ken}, that $\rho(D_{\Gamma}';L^2_0(\Gamma))<\half$
if $\Gamma$ is the boundary of a bounded Lipschitz domain. Conjecture \ref{con:kenig}, stated in the introduction, is a version of this conjecture that avoids difficulties when $\Gamma$ is not connected\footnote{By results of D.~Mitrea \cite[Theorem 4.1]{MitreaD:97}, if $\Omega_-$ or $\Omega_+$ are not connected, $\{\frac{1}{2},-\frac{1}{2}\}\cap \sigma(D_{\Gamma}';L^2_0(\Gamma))\neq \emptyset$. Indeed, $D_{\Gamma}' \pm \frac{1}{2}I$ are Fredholm of index zero on $L^2(\Gamma)$ by part (2) of \cite[Theorem 4.1]{MitreaD:97}, but by (4) of the same theorem, and since $D_{\Gamma}(1)=-\half$ so that $D_{\Gamma}'(L^2_0(\Gamma))\subset L^2_0(\Gamma)$ (see \eqref{eq:mapping} below), the codimension of $D_{\Gamma}' - \frac{1}{2}I$ on $L^2_0(\Gamma)$ is $\geq 1$ if $\Omega_+$ is not connected; that of  $D_{\Gamma}' + \frac{1}{2}I$ is $\geq 1$ if $\Omega_-$ is not connected.}. In this section we show the equivalence between Conjecture \ref{con:kenig} and Conjecture \ref{con:kenigmod}, also stated in the introduction. Conjecture \ref{con:kenigmod} concerns the essential spectrum rather than the spectrum and makes sense whatever the connectedness of $\Gamma$, $\Omega_-$, and $\Omega_+$.

Of course, because $D_{\Gamma}'$ is the adjoint of $D_{\Gamma}$ (as an operator on $L^2(\Gamma)$), and so shares the same essential spectrum, a statement equivalent to Conjecture \ref{con:kenigmod} is that $\rho_{\ess}(D_{\Gamma}';L^2(\Gamma))<\frac{1}{2}$. Further, as bounded Lipschitz domains have only finitely many boundary components, it  is clear
that Conjecture \ref{con:kenigmod} is true if it is true whenever $\Gamma$ is connected\footnote{Note that if $\Gamma_1$ and $\Gamma_2$ are separate components of $\Gamma$, so that there is a positive distance between $\Gamma_1$ and $\Gamma_2$, the double-layer operator from $L^2(\Gamma_1)$ to $L^2(\Gamma_2)$ has a kernel that is bounded, so is a Hilbert-Schmidt operator and hence compact.}.  Thus the equivalence of Conjectures \ref{con:kenig} and \ref{con:kenigmod}, i.e.\ that Conjecture \ref{con:kenig} holds (as claimed, whenever $\Gamma$ is connected) if and only if Conjecture \ref{con:kenigmod} holds (as claimed, whatever the topology of $\Gamma$), is implied by the following lemma.

\begin{lem} \label{lem:equiv} Assume that $\Gamma$ is the boundary of a bounded Lipschitz domain $\Omega_-$ and is connected. Then $\rho(D_{\Gamma}';L^2_0(\Gamma))<\half$ if and only if $\rho_{\ess}(D_{\Gamma}';L^2(\Gamma))<\half$.
\end{lem}

The main step in the proof of this lemma is the following theorem stated in \cite{FaSaSe:92}\footnote{In \cite[Theorem 1.1]{FaSaSe:92} slightly more is claimed, that the eigenvalues of $D_{\Gamma}'$ lie in $[-\half,\half)$. This is more than is claimed in \cite[Chapter IX, \S11]{Kellogg29}, and in fact $\half$ is an eigenvalue if $\Omega_+$ is not connected; see \cite[Theorem 4.1]{MitreaD:97}.}. As noted in \cite{FaSaSe:92}, the proof of this result is that in Kellogg's classical book for the case of smooth boundaries (see \cite[Chapter IX, \S11]{Kellogg29}), which carries over to the Lipschitz case.

\begin{thm}[Theorem 1.1 in  \cite{FaSaSe:92}] \label{thm:Kell} If $\Gamma$ is the boundary of a bounded Lipschitz domain $\Omega_-$, the eigenvalues of $D_{\Gamma}'$, as an operator on $L^2(\Gamma)$, are real and lie in $[-\half,\half]$.
\end{thm}

\begin{proof}[Proof of Lemma \ref{lem:equiv}]
Let $\phi\in L^2(\Gamma)$ and let $\langle\cdot,\cdot\rangle$ denote the inner product on $L^2(\Gamma)$. We first observe that, since $D_{\Gamma}(1)=-\half$, we have
\begin{equation} \label{eq:mapping}
\langle D_{\Gamma}'\phi,1\rangle = \langle\phi,D_{\Gamma}(1)\rangle = -\tfrac{1}{2} \langle\phi,1\rangle,
\end{equation}
so $D_{\Gamma}'(L^2_0(\Gamma)) \subseteq L^2_0(\Gamma)$.

Now assume that $\lambda\in\C$ and $D_{\Gamma}' - \lambda I$ is invertible as an operator on $L^2_0(\Gamma)$. Let $P$ be orthogonal projection from $L^2(\Gamma)$ onto the constants, so that $Q:=I-P$ is projection onto $L^2_0(\Gamma)$. Then $(D_{\Gamma}' - \lambda)Q+P$ is invertible as an operator on $L^2(\Gamma)$. It follows that
\begin{equation} \label{eq:PQ}
D_{\Gamma}' - \lambda I = ((D_{\Gamma}' - \lambda)Q+P) + ((D_{\Gamma}' - \lambda)P-P)
\end{equation}
is Fredholm as an operator on $L^2(\Gamma)$. This implies that
$\sigma_{\ess}(D_{\Gamma}';L^2(\Gamma)) \subset \sigma(D_{\Gamma}';L^2_0(\Gamma))$,
which settles one direction.

Conversely, assume that $\rho_{\ess}(D_{\Gamma}';L^2(\Gamma))<\frac{1}{2}$.
Let $\lambda \in \C$ with $|\lambda| \geq \half$. Then  $D_{\Gamma}' - \lambda I$ is Fredholm of index 0 on $L^2(\Gamma)$, so (see \eqref{eq:PQ}) $(D_{\Gamma}' - \lambda )Q +P$ is Fredholm of index zero on $L^2(\Gamma)$, so that $D_{\Gamma}' - \lambda I$ is Fredholm of index zero on $L^2_0(\Gamma)$. Thus $\lambda \in \sigma(D_{\Gamma}';L^2_0(\Gamma))$ if and only if $\lambda$ is an eigenvalue of $D_{\Gamma}'$. Hence, if $\lambda \notin [-\half,\half]$, $D_{\Gamma}' - \lambda I$ is invertible on both $L^2(\Gamma)$ and $L^2_0(\Gamma)$ by Theorem \ref{thm:Kell} (as $L^2_0(\Gamma) \subseteq L^2(\Gamma)$, every eigenvalue on $L^2_0(\Gamma)$ is also an eigenvalue on $L^2(\Gamma)$). But also, if $\lambda = \pm\half$, since $\Gamma$ is connected, Verchota's results \cite{Ver84} show that $D_{\Gamma}' \pm \half I$ is invertible on $L^2_0(\Gamma)$. Thus $D_{\Gamma}' - \lambda I$ is invertible on $L^2_0(\Gamma)$ for $\abs{\lambda} \geq \half$, so that, since the spectrum is closed, $\rho(D_{\Gamma}';L^2_0(\Gamma))<\half$.
\end{proof}

\section{Approximation of the spectral radius for compact operators} \label{sec:Nystrom}

In this section we recall in \S\ref{sec:cc} results from operator approximation theory in Banach spaces related to
the spectra of compact operators, and derive what appear to be new general criteria for $\rho(T)<\rho_0$ when $T$ is compact and $\rho_0>0$ (Lemma \ref{lem:GspecR}, Corollary \ref{cor:GspecRMfinal}). This leads, in \S\ref{sec:Ny}, to results relating to the approximation of integral operators with continuous kernels by the Nystr\"om method that will be key for the arguments in \S\ref{sec:oneside} and \S\ref{sec:two-sided}.  Notably, Theorem \ref{thm:GspecRMfinal} provides criteria for $\rho(T)<\rho_0$ when $T$ is an integral operator with a continuous kernel that requires the computation only of the spectral radius of a finite matrix plus the norms of finitely many finite matrix resolvents.

\subsection{Operator approximation results} \label{sec:cc}

We recall first two standard results on the approximation of operators in $\cL(Y)$, the space of bounded linear operators on a Banach space $Y$. The first is the basic perturbation estimate that, if $S,T\in \cL(Y)$ and $S$ is invertible, then $T$ is invertible if $\norm{S-T}\norm{S^{-1}} < 1$, with
\begin{equation} \label{eq:PerBas}
\norm{T^{-1}}^{-1} \geq \norm{S^{-1}}^{-1} - \norm{S-T}.
\end{equation}
If $S\in \cL(Y)$ is invertible, then $\|S^{-1}\|^{-1}=\inf_{\|\phi\| = 1} \|S\phi\|$, this sometimes called the {\em lower norm} of $S$ (see, e.g., \cite[Lemma 2.35]{Li:06}). The second estimate is as follows:

\begin{lem}[Theorem 4.7.7 of \cite{Ha:95}] \label{thm:Hackbusch}
Let $Y$ be a Banach space, $S \in \mathcal{L}(Y)$ and $\lambda \notin \spec(S) \cup \{0\}$. If $T \in \mathcal{L}(Y)$ is a compact operator that satisfies
\[\norm{(T-S)T} < \abs{\lambda}\norm{(S-\lambda I)^{-1}}^{-1},\]
then $T - \lambda I$ is invertible.
\end{lem}

The following result
is a consequence of the above estimates and the maximum principle applied to the resolvent.
Here $\T=\{z:|z|=1\}$ is the unit circle in the complex plane, so that $\rho_0\T$ is the circle of radius $\rho_0$.

\begin{lem} \label{lem:GspecR} Let $Y$ be a Banach space, $S,\widehat S,T \in \mathcal{L}(Y)$, $\rho_0>0$, $F\subset \rho_0\T$, and suppose that: $T$ is compact; $\rho(\widehat S)<\rho_0$; for every $\lambda\in \rho_0\T$ there exists $\mu\in F$ such that
\begin{equation} \label{eq:keybound}
\norm{(T-S)T} < \rho_0\left( \norm{(\widehat S-\mu I)^{-1}}^{-1} - \norm{S-\widehat S} - |\lambda-\mu|\right).
\end{equation}
Then $\rho(T)<\rho_0$.
\end{lem}

\noindent The idea is to choose $F$ to be a finite set\footnote{We will also apply this lemma later in the case that $F=\rho_0\T$, when \eqref{eq:keybound} reduces to the condition that, for every $\lambda\in \rho_0\T$, $\|(T-S)T\| < \rho_0( \|(\widehat S-\lambda I)^{-1}\|^{-1} - \|S-\widehat S\|)$.}, $S$ a finite rank approximation to $T$, and $\widehat S$ a numerical approximation to $S$, in which case one can show $\rho(T)<\rho_0$ by computing $\rho(\widehat S)$ and $\|(\widehat S-\mu I)^{-1}\|$ for finitely many $\mu$. Taking $F=\rho_0\T_n$, where $\T_n$ is the $n$th roots of unity, we obtain:
\begin{cor} \label{cor:uni_space}
Let $Y$ be a Banach space, $S,\widehat S,T \in \mathcal{L}(Y)$, $\rho_0>0$, $n\in \N$, and suppose that $T$ is compact, $\rho(\widehat S)<\rho_0$, and
\begin{equation} \label{eq:keybound2}
\norm{(T-S)T} < \rho_0\left( \|(\widehat S-\mu I)^{-1}\|^{-1} - \|S-\widehat S\| - 2\rho_0\sin\left(\frac{\pi}{2n}\right)\right), \qquad \mu\in \rho_0 \T_n.
\end{equation}
Then $\rho(T)<\rho_0$.
\end{cor}

\begin{proof}[Proof of Lemma \ref{lem:GspecR}] Suppose that the conditions of the lemma are satisfied. Then $\rho(\widehat S)< \rho_0$ and, for every $\lambda \in \rho_0\T$, there exists $\mu\in F$ such that \eqref{eq:keybound} holds.
It follows from \eqref{eq:PerBas} that
\begin{equation} \label{eq:2ndbound}
\norm{(T-S)T} < \rho_0\left( \|(\widehat S-\lambda I)^{-1}\|^{-1} - \|S-\widehat S\|\right), \qquad \lambda \in \rho_0\T.
\end{equation}
The resolvent map $\lambda \mapsto (\widehat S-\lambda I)^{-1}$ is analytic on $|\lambda|>\rho(\widehat S)$, which set contains all $\lambda$ with $|\lambda|\geq \rho_0$. Thus, by the maximum principle, $\|(\widehat S-\lambda I)^{-1}\|$ attains its maximum in $|\lambda| \geq \rho_0$ on $\rho_0\T$. Thus \eqref{eq:2ndbound} in fact holds for all $\lambda$ with $|\lambda|\geq \rho_0$, so that $S-\lambda I$ is invertible for all such $\lambda$ and, by \eqref{eq:PerBas},
\begin{equation} \label{eq:3ndbound}
\norm{(T-S)T} < \rho_0 \norm{(S-\lambda I)^{-1}}^{-1}, \qquad \mbox{if } |\lambda| \geq \rho_0.
\end{equation}
Since $T$ is compact, the result follows from Lemma \ref{thm:Hackbusch}.
\end{proof}

When Lemma \ref{lem:GspecR} is used for computation with $F$ finite, it is desirable to minimise the cardinality of $F$ since $\|(\widehat S-\mu I)^{-1}\|$ has to be computed for every $\mu\in F$. One can choose $F=\rho_0\T_n$, with points uniformly distributed on $\rho_0\T$, as in Corollary \ref{cor:uni_space}, but $n$ needs to be at least large enough so that $\|(\widehat S-\mu I)^{-1}\|^{-1} > 2\rho_0\sin(\pi/(2n))$, for every $\mu\in F$. In many applications, including in \S\ref{sec:oneside} and \S\ref{sec:two-sided}, $\|(\widehat S-\mu I)^{-1}\|$ varies significantly as $\mu$ moves around $\rho_0\T$ and it is more efficient to vary the spacing of the points in $F$ approximately in proportion to $\|(\widehat S-\mu I)^{-1}\|^{-1}$. The adaptive algorithm described in the following lemma, which we will see implemented in Figure \ref{fig:plot2}(b) below, approximately achieves this.

\begin{lem} \label{lem:R*def} Let $Y$ be a Banach space, $\widehat S\in \mathcal{L}(Y)$, and $\rho_0>0$. Suppose that $\rho(\widehat S)< \rho_0$, and recursively define $\mu_\ell$, for $\ell\in \N$, by $\mu_{1} := \rho_0$, and by
\begin{equation} \label{eq:mu_ell1}
\nu_{\ell} := \norm{(\widehat S - \mu_{\ell} I)^{-1}}^{-1} \quad \text{and} \quad \mu_{\ell+1} := \mu_{\ell}e^{i\frac{\nu_{\ell}}{2\rho_0}}, \quad \mbox{for } \ell\in \N.
\end{equation}
Further, set $n_* \in \N$ to be the smallest integer such that $\sum_{\ell = 1}^{n_*} \nu_{\ell}\geq 4\pi \rho_0$, and set $F:= \{\mu_1,\mu_2,\ldots,\mu_{n_*+1}\}$. Then, for every $\lambda\in \rho_0\T$ there exists $\mu\in F$ such that
\begin{equation} \label{eq:R*def}
\norm{(\widehat S - \lambda I)^{-1}}^{-1} \geq \norm{(\widehat S - \mu I)^{-1}}^{-1}-|\lambda-\mu| \geq R_* := \min_{\ell=1,\ldots,n_*} \left(\frac{\nu_\ell}{4}+\frac{\nu_{\ell+1}}{2}\right).
\end{equation}
\end{lem}
\begin{proof} Arguing as in the proof of Lemma \ref{lem:GspecR}, $\|(\widehat S -\lambda I)^{-1}\|^{-1}$ is bounded below by some $c>0$ on $\rho_0\T$, so that $n_*$ is well-defined with $n_*<1+4\pi \rho_0/c$. For $\ell=1,\ldots,n_*+1$, $\mu_\ell = \rho_0e^{i\theta_\ell}$, with $\theta_1=0$, $\theta_{n_*+1}\geq 2\pi$, and $\theta_{\ell+1}-\theta_\ell = \nu_\ell/(2\rho_0)$. Thus, if $\lambda\in \rho_0\T$, then
$$
\lambda = \mu_\ell \exp(is\nu_\ell/(2\rho_0)) = \mu_{\ell+1}\exp(i(s-1)\nu_\ell/(2\rho_0)),
$$
for some $\ell\in\{1,\ldots,n_*\}$ and some $s\in [0,1]$. Since $|e^{it}-1|=|\int_0^t e^{iu}\, \rd u|\leq |t|$, for $t\in \R$, it follows, using \eqref{eq:PerBas}, that
\begin{eqnarray*}
\|(\widehat S - \lambda I)^{-1}\|^{-1} &\geq & \nu_\ell - |\lambda - \mu_\ell| \geq \nu_\ell - \frac{s}{2}\nu_\ell = \left(1-\frac{s}{2}\right)\nu_\ell \quad \mbox{and}\\
\|(\widehat S - \lambda I)^{-1}\|^{-1} &\geq & \nu_{\ell+1} - |\lambda - \mu_{\ell+1}| \geq \nu_{\ell+1} - \frac{1-s}{2}\nu_\ell.
\end{eqnarray*}
In particular, $\nu_{\ell+1} \geq \nu_\ell/2$ and $\nu_\ell \geq  \nu_{\ell+1} - \nu_\ell/2$, so that $s^*:= 3/2-\nu_{\ell+1}/\nu_\ell \in [0,1]$. But, for $0\leq s\leq s*$, $(1-s/2)\nu_\ell \geq (1-s^*/2)\nu_\ell =\nu_\ell/4+\nu_{\ell+1}/2$, while, for $s^*\leq s\leq 1$, $\nu_{\ell+1} - (1-s)\nu_\ell/2 \geq \nu_{\ell+1} - (1-s^*)\nu_\ell/2 = \nu_\ell/4+\nu_{\ell+1}/2$, and the bound \eqref{eq:R*def} follows.
\end{proof}

The following corollary is immediate from the above lemma and Lemma \ref{lem:GspecR}.
\begin{cor}
\label{cor:GspecRMfinal}
Let $Y$ be a Banach space, $S,\widehat S,T \in \mathcal{L}(Y)$, $\rho_0>0$, and suppose that $T$ is compact, $\rho(\widehat S)<\rho_0$, and
\begin{equation} \label{eq:keyb}
\norm{(T-S)T} < \rho_0\left( R_* - \|S-\widehat S\|\right),
\end{equation}
where $R_*$ is as defined in Lemma \ref{lem:R*def}. Then $\rho(T)<\rho_0$.
\end{cor}

We will apply the above results
in the case when $S=T_N$, where $(T_N)_{N\in \N}$ is a {\em collectively compact}\footnote{Recall, e.g., \cite{An:71}, \cite[\S10.3]{Kr:14}, that a set $S\subset \cL(Y)$ is said to be {\em collectively compact} if $\{T\phi:T\in S, \phi\in Y, \|\phi\|\leq 1\}$ is relatively compact.} sequence of operators converging strongly to $T$ (we write $T_N\to T$ for strong convergence). A standard, simple but important result (e.g., \cite[Cor.~1.9]{An:71}, \cite[Thm.~10.10]{Kr:14}) is that
\begin{equation} \label{eq:pointwise}
(T_N)_{N\in \N} \mbox{ collectively compact}, \quad T_N\to T \quad \Rightarrow \quad \|(T-T_N)T\|\to 0.
\end{equation}
A consequence of \eqref{eq:pointwise} is Theorem \ref{thm:convergence_estimate} below, which follows from \cite[Thm~4.8]{An:71} and \cite[Thm.~4.16]{An:71} (or see \cite{At:75}). This gives conditions on operators $T_N$ and $T$ that ensure convergence of $\sigma(T_N)$ to $\sigma(T)$ in the standard Hausdorff metric $d_H(\cdot,\cdot)$ (see, e.g., \cite[\S3.1.2]{HaRoSi01})
on the set $\C^C$ of compact subsets of $\C$.
Given a sequence $(A_N)_{N\in \N}\subset \C^C$ and $A\in \C^C$ we will write $A_N\toH A$ if $d_H(A,B)\to 0$, i.e.\ if $A_N$ converges to $A$ in the Hausdorff metric. We recall (e.g., \cite[Prop.~3.6]{HaRoSi01}) that $A_N\toH A$ if and only if $(A_N)_{N\in\N}$ is uniformly bounded and $\liminf A_N = \limsup A_N = A$, where $\liminf A_N$ is the set of limits of sequences $(z_N)$ such that $z_N\in A_N$ for each $N$, while $\limsup A_N$ is the set of partial limits of such sequences.
\begin{thm}[\cite{An:71}]\label{thm:convergence_estimate}
Let $Y$ be a Banach space and $T\in \cL(Y)$ a compact operator. Further, assume that $(T_N)_{N \in \N}\subset \cL(Y)$ is collectively compact and converging strongly to $T$. Then $\sigma(T_N)\toH \sigma(T)$.
\end{thm}

\subsection{The Nystr\"om method} \label{sec:Ny}

In this section, with a view to applications in \S\ref{sec:oneside} and \S\ref{sec:two-sided}, we apply the results of \S\ref{sec:cc} to the case where $Y=C(X)$, for some compact $X\subset \R^{d-1}$, and where the operators are integral operators that we approximate by the Nystr\"om method. So suppose $d\geq 2$, let $X \subset \R^{d-1}$ be a compact set of positive ($(d-1)$-dimensional) Lebesgue measure, and $K\in \cL(C(X))$ an integral operator with continuous kernel $K(\cdot,\cdot)$, so that $K$ is compact. Thus, for $\phi\in C(X)$ and $x\in X$,
\begin{equation} \label{eq:Kdef}
K\phi(x) = \int_X K(x,y)\phi(y)\, dy = J(K(x,\cdot)\phi), \quad \mbox{where}\quad J(\psi) := \int_X \psi(y) \, dy, \quad \mbox{for } \psi\in C(X).
\end{equation}
In the Nystr\"om method we approximate $K$ by replacing the integration functional $J:C(X)\to \C$ by a sequence of numerical quadrature rules.  For each $N \in \N$ we choose points $x_{q,N}\in X$ and weights\footnote{We assume, for simplicity, that the weights $\omega_{q,N}$ are positive, but the theory below applies, with minor changes, to the case of general real or complex weights.}
$\omega_{q,N}\geq 0$, for $q=1,\ldots, N$, and define $J_N:C(X)\to \C$ by
\begin{equation} \label{JNdef}
J_N(\psi) := \sum_{q=1}^N \omega_{q,N}\psi(x_{q,N}), \qquad \psi\in C(X),
\end{equation}
 and a Nystr\"om approximation $K_N\in \cL(C(X))$ to $K$ by
\begin{equation} \label{eq:Nyst}
K_N \phi(x) := J_N(K(x,\cdot)\phi)) = \sum_{q=1}^N \omega_{q,N}K(x,x_{q,N}) \phi(x_{q,N}), \qquad x \in X, \;\; \phi\in C(X).
\end{equation}
We will assume that the sequence of quadrature rules is convergent, by which we mean that
\begin{equation} \label{eq:JNtoJ}
J_N\to J, \quad \mbox{i.e.} \quad J_N\psi \to J\psi, \quad \mbox{for all } \psi\in C(X).
\end{equation}
This implies (e.g., \cite[Prop.~2.1, 2.2]{An:71}) that $K_N \to K$ and that the sequence $(K_N)_{N\in \N}\subset \cL(C(X))$ is collectively compact, so that \eqref{eq:pointwise} holds and Theorem \ref{thm:convergence_estimate} is applicable. We will also assume that, for each $N$, $J_N \phi = J \phi$ if $\phi\in C(X)$ is constant, i.e. that
\begin{equation} \label{eq:Xsum}
\sum_{q=1}^N \omega_{q,N} = |X|,
\end{equation}
where $|X|$ denotes the Lebesgue measure of $X$.

Define $\widehat K_N:\C^N\to C(X)$ by
\[\widehat K_N v(x) := \sum_{q=1}^N \omega_{q,N}K(x,x_{q,N})v_q,\]
for $x \in X$ and $v=(v_1,...,v_N)\in \C^N$.
Moreover, let $P_N \from C(X)\to \C^N$ be defined by
$P_N\phi(q) := \phi(x_{q,N})$, for $q=1,...,N$, $\phi\in C(X)$, and
 define $A_N \from \C^N \to \C^N$ by $A_N := P_N\hat{K}_N$,
so that the matrix entries of $A_N$ are given by
\begin{equation} \label{eq:A_{t,N}}
A_N(p,q) = \omega_{q,N}K(x_{p,N},x_{q,N}), \qquad p,q=1,\ldots,N.
\end{equation}
In the following we will use $\|\cdot\|_\infty$ to denote all of: i) the standard supremum norm on $C(X)$; ii) the standard infinity norm on $\C^N$; iii) the induced operator norm of an operator on $C(X)$; iv) the infinity norm of a square matrix.
\begin{lem} \label{lem:reduction_to_matrices}
The following inequalities and equalities hold for all $N \in \N$:
\begin{itemize}
	\item[(i)] \[\max_{1\leq p\leq N} \sum_{q=1}^N \omega_{q,N}|K(x_{p,N},x_{q,N})| = \norm{A_N}_\infty \leq \norm{K_N}_\infty = \max_{x\in X} \sum_{q=1}^N \omega_{q,N}|K(x,x_{q,N})|;\]
	\item[(ii)] $\spec(K_N) = \{0\} \cup \spec(A_N)$;
	\item[(iii)] for $\lambda \in \C \setminus \spec(K_N)$,
\begin{align*}
\max\left(|\lambda|^{-1},\norm{(A_N - \lambda I)^{-1}}_\infty\right) \leq \norm{(K_N - \lambda I)^{-1}}_\infty
&\leq |\lambda|^{-1}\left(1+\norm{K_N}_\infty \norm{(A_N - \lambda I)^{-1}}_\infty\right).
\end{align*}
\end{itemize}
\end{lem}
\begin{proof}

The first equality in (i) is the standard explicit formula for the infinity norm of a matrix. The last equality is proved similarly, and then (i) is clear (or see \cite[Lemma 4.7.17]{Ha:95}). That the spectra of $K_N$ and $A_N$ coincide on $\C\setminus \{0\}$ is standard (e.g., \cite[Lemma 4.7.18]{Ha:95}), and $0\in \spec(K_N)$ since $K_N$ is compact and $C(X)$ is infinite-dimensional, so (ii) holds. Part (iii) is a combination of
\cite[Lemma 4.7.18]{Ha:95} and \cite[Proposition 2.3]{An:71}, plus the facts that $0\in \sigma(K_N)$ and (e.g., \cite[Th.~1.2.10]{BDav})  $\|(K_N - \lambda I)^{-1}\|\geq (\dist(\lambda,\sigma(K_N)))^{-1}$.
\end{proof}

\noindent Note that, by Lemma \ref{lem:reduction_to_matrices}(i) and \eqref{eq:Xsum},
\begin{equation} \label{eq:KNbound}
\|A_N\|_\infty \leq \|K_N\|_\infty \leq K_{\rm max} |X|, \quad \mbox{where} \quad K_{\rm max}:= \max_{x,y\in X}|K(x,y)|.
\end{equation}

Commonly, for computational efficiency or otherwise, and this is the case in \S\ref{sec:oneside} and \S\ref{sec:two-sided},  we approximate $K(\cdot,\cdot)$ by another continuous kernel $K^{\dag}(\cdot,\cdot)$ in \eqref{eq:Nyst} and \eqref{eq:A_{t,N}}. Let $K_N^{\dag}\in \cL(C(X))$ denote the operator defined by the right hand side of \eqref{eq:Nyst} with $K$ replaced by $K^{\dag}$, and similarly $A_N^{\dag}$ denote the matrix \eqref{eq:A_{t,N}} with $K$ replaced by $K^{\dag}$. Lemma \ref{lem:reduction_to_matrices}(i) and \eqref{eq:Xsum} imply that
\begin{equation} \label{eq:NormEst}
\|A_N-A_N^{\dag}\|_\infty \leq \|K_N-K_N^{\dag}\|_\infty \leq e^{\dag} |X|, \quad \mbox{where} \quad e^{\dag}:= \max_{x,y\in X}|K(x,y)-K^{\dag}(x,y)|.
\end{equation}
The following theorem (cf.~Corollary \ref{cor:GspecRMfinal}) follows in large part from Lemma \ref{lem:GspecR},  Lemma \ref{lem:R*def}, and Lemma \ref{lem:reduction_to_matrices}.

\begin{thm} \label{thm:GspecRMfinal} Suppose that $\rho_0 > 0$, $N\in \N$, and $\rho(A_N^{\dag})< \rho_0$, and recursively define $\mu_\ell$, for $\ell\in \N$, by $\mu_{1} := \rho_0$, and by
\begin{equation} \label{eq:mu_ell}
\nu_{\ell} := \norm{(A_{N}^{\dag} - \mu_{\ell} I)^{-1}}_{\infty}^{-1} \quad \text{and} \quad \mu_{\ell+1} := \mu_{\ell}e^{i\frac{\nu_{\ell}}{2\rho_0}}, \quad \mbox{for } \ell\in \N.
\end{equation}
Further, let $n_N \in \N$ be the smallest integer such that $\sum_{\ell = 1}^{n_N} \nu_{\ell}\geq 4\pi \rho_0$, and let
\begin{equation} \label{eq:RNdef}
R_N := \min_{\ell=1,\ldots,n_N} \left(\frac{\nu_\ell}{4}+\frac{\nu_{\ell+1}}{2}\right).
\end{equation}
If
\begin{equation} \label{eq:keyboundfinal}
\norm{(K-K_N)K}_\infty < \rho_0\left(\rho_0\left(1+\|K^{\dag}_N\|_\infty R_N^{-1}\right)^{-1} - \|K_N-K_N^{\dag}\|_\infty\right),
\end{equation}
or $\|A_N-A_N^{\dag}\|_\infty < R_N$ and
\begin{equation} \label{eq:keyboundfinalAlt}
\norm{(K-K_N)K}_\infty < \rho_0^2\left(1+\|K_N\|_\infty (R_N-\|A_N-A_N^{\dag}\|_\infty)^{-1}\right)^{-1},
\end{equation}
then $\rho(K)<\rho_0$. Conversely, if $\rho(K)<\rho_0$, provided $e^{\dag}$ defined by \eqref{eq:NormEst} is sufficiently small, \eqref{eq:keyboundfinal}, \eqref{eq:keyboundfinalAlt}, and $\rho(A_N^{\dag})<\rho_0$  hold for all sufficiently large $N$.
\end{thm}

\begin{proof} If $\rho_0>0$ and $\rho(A_N^{\dag})< \rho_0$ then, by Lemma \ref{lem:R*def},
$\|(A_N^{\dag}-\lambda I)^{-1}\|_\infty\leq R_N^{-1}$ for $\lambda \in \rho_0\T$, so that $\|(K_N^{\dag}-\lambda I)^{-1}\|_\infty\leq r^{-1}_0(1+\|K_N^{\dag}\|_\infty R_N^{-1})$, by Lemma \ref{lem:reduction_to_matrices}(iii).
Thus  $\rho(K)<\rho_0$ if \eqref{eq:keyboundfinal} holds, by Lemma \ref{lem:GspecR} applied with $Y=C(X)$, $T=K$, $S=K_N$, $\widehat S = K_N^{\dag}$, and $F=\rho_0 \T$. Further, for $\lambda\in \rho_0 \T$, the first of the above bounds and \eqref{eq:PerBas} implies that $R_N-\|A_N-A_N^{\dag}\|_\infty \leq \|(A_N-\lambda I)^{-1}\|^{-1}_\infty$, so that, by Lemma \ref{lem:reduction_to_matrices}(iii), $\|(K_N-\lambda I)^{-1}\|_\infty\leq r^{-1}_0(1+\|K_N\|_\infty (R_N-\|A_N-A_N^{\dag}\|_\infty)^{-1})$.
Thus $\rho(K)<\rho_0$ if \eqref{eq:keyboundfinalAlt} holds, by
Lemma \ref{lem:GspecR} applied with $Y=C(X)$, $T=K$, $\widehat S =S=K_N$, and $F=\rho_0 \T$.

To see the converse, note that, by Theorem \ref{thm:convergence_estimate}, $\sigma(K_N)\toH \sigma(K)$. Indeed \cite[Thm.~4.7]{An:71}, there exists $N_0\in \N$ such that $\|(K_N-\lambda I)^{-1}\|_\infty$ is bounded uniformly in $\lambda$ and $N$ for $N\geq N_0$ and $|\lambda|\geq \rho_0$, which implies, by \eqref{eq:PerBas} and \eqref{eq:NormEst}, that the same holds for $\|(K^{\dag}_N-\lambda I)^{-1}\|_\infty$ if $e^{\dag}$ is sufficiently small. This implies, by Lemma \ref{lem:reduction_to_matrices}(iii), that the same holds for $\|(A^{\dag}_N-\lambda I)^{-1}\|_\infty$.  Noting Lemma \ref{lem:reduction_to_matrices}(ii), it follows that there exists $c>0$ such that, for all sufficiently large $N$, $\rho(A_N^{\dag})=\rho(K_N^{\dag})< \rho_0$ and $R_N\geq c$. Thus, and applying \eqref{eq:KNbound} and \eqref{eq:NormEst}, we see that, for some $c^*>0$, the right hand sides of \eqref{eq:keyboundfinal} and \eqref{eq:keyboundfinalAlt} are $\geq c^*$ for all sufficiently large $N$ if $e^{\dag}$ is sufficiently small. But also $\|(K-K_N)K\|\to 0$ by \eqref{eq:pointwise}, so that \eqref{eq:keyboundfinal} holds for all sufficiently large $N$.
\end{proof}

\begin{rem}[\em \bf Computational cost as $N$ increases] \label{rem:nNbounded}
The argument in the above proof makes clear that, if $\rho(K) < \rho_0$ and $e^{\dag}$ is sufficiently small, then, for some $N_0\in \N$, $\|(A_N^{\dag}-\lambda I)^{-1}\|_\infty$ is bounded uniformly in $\lambda$ and $N$ for $\lambda\in \rho_0 \T$ and $N\geq N_0$. This in turn implies that, for some $n_{\max}\in \N$, $n_N\leq n_{\max}$ for $N\geq N_0$. Thus the computational cost of evaluation of $R_N$ given by \eqref{eq:RNdef} is $O(N^3)$, the cost of inverting $n_N\leq n_{\max}$ order $N$ matrices by classical direct methods.
\end{rem}

\begin{rem}[\em\bf Comparison of \eqref{eq:keyboundfinal} and \eqref{eq:keyboundfinalAlt}] \label{rem:comparison}
Let $\mathrm{RHS}_1$ and $\mathrm{RHS}_2$ denote the right hand sides of \eqref{eq:keyboundfinal} and \eqref{eq:keyboundfinalAlt}, respectively. If $K_N^{\dag}=K_N$ (so $A_N^{\dag}=A_N$),  then $\mathrm{RHS}_2=\mathrm{RHS}_1$. If $K^{\dag}_N\neq K_N$ with $\|A_N-A_N^{\dag}\|_\infty <R_N$, then, where $\mathcal{D} :=
\rho_0^{-2}(R_N+\|K_N\|_\infty - \|A_N-A_N^\dag\|_\infty)(\mathrm{RHS}_2-\mathrm{RHS}_1)
$,
\begin{eqnarray*}
\mathcal{D} &=& \rho_0^{-1}\|K_N-K_N^{\dag}\|_\infty (\|K_N\|_\infty-\rho_0 + R_N-\|A_N-A_N^\dag\|_\infty) \, + \\
& & \; \frac{\|K_N^\dag\|_\infty(\|K_N-K_N^\dag\|_\infty-\|A_N-A_N^\dag\|_\infty)+R_N(\|K_N-K_N^\dag\|_\infty + \|K_N^\dag\|_\infty-\|K_N\|_\infty)}{R_N+\|K_N^\dag\|_\infty},
\end{eqnarray*}
so that $\mathrm{RHS}_2-\mathrm{RHS}_1> \rho_0\|K_N-K_N^{\dag}\|_\infty (\|K_N\|_\infty-\rho_0)/(R_N+\|K_N\|_\infty - \|A_N-A_N^{\dag}\|_\infty)$, recalling \eqref{eq:NormEst}. Thus \eqref{eq:keyboundfinal} implies \eqref{eq:keyboundfinalAlt} if $\|A_N-A_N^{\dag}\|_\infty <R_N$ and $\rho_0\leq \|K_N\|_\infty$. Note that \cite[Thm.~2.13]{An:71} $\|K_N\|_\infty\to \|K\|$ as $N\to \infty$ and it is $\rho_0<\|K\|_\infty$ for which Theorem \ref{thm:GspecRMfinal} is arguably of most interest, as $\rho(K) \leq \|K\|_\infty$, and we may be able to estimate $\|K\|_\infty$ sharply by other methods.
\end{rem}

\begin{rem}[\em \bf The matrix case] \label{rem:matrix} In \S\ref{sec:two-sided} we will apply the above results, in particular Theorem \ref{thm:GspecRMfinal}, in a case where $K$ is a $2\times 2$ matrix of integral operators on $C(X)$ with continuous kernels, and $K_N$ is its Nystr\"om method approximation defined by approximating each integral operator in the $2\times 2$ matrix as in \eqref{eq:Nyst}. The matrix $A_N$ is then a $2N\times 2N$ matrix consisting of four $N\times N$ blocks each defined as in \eqref{eq:A_{t,N}}. Parts (ii) and (iii) of Lemma \ref{lem:reduction_to_matrices} apply in this case, as does (i) in a straightforwardly modified form, in particular we still have that $\|A_N\|_\infty \leq \|K_N\|_\infty$. (Here $\|A_N\|_\infty$ is the usual infinity norm of the matrix $A_N$ and $\|K_N\|_\infty$ is the norm of $K_N$ as an operator on $(C(X))^2$, which we equip with the norm defined by $\|(\phi_1,\phi_2)\|_\infty := \max\{\|\phi_1\|_\infty,\|\phi_2\|_\infty\}$, for $(\phi_1,\phi_2)\in C(X)^2$). Thus Theorem \ref{thm:GspecRMfinal}, which depends on Lemma \ref{lem:reduction_to_matrices}(ii) and (iii) and the general Banach space results of \S\ref{sec:cc}, still applies. One way to see the validity of Lemma \ref{lem:reduction_to_matrices}(ii) and (iii) in this matrix case is to argue as follows. Choose $x^*\in \R^{d-1}$ such that $X^\prime := X+x^*$ does not intersect $X$. $K$ is equivalent, through an obvious isometric isomorphism, to a matrix operator $K^\prime$ on $C(X)\times C(X^\prime)$, which is in turn equivalent, through another obvious isometric isomorphism, to a single integral operator $\widetilde K$ on $C(X\cup X^\prime)$. Lemma  \ref{lem:reduction_to_matrices} applies to $\widetilde K$ and its Nystr\"om method approximation, so that (ii) and (iii) of this lemma (and (i) in modified form) apply to $K$.
\end{rem}

\section{The Double-Layer Operator on Dilation invariant graphs} \label{sec:dilation_invariant}

Let $A \subseteq \R^{d-1}$ be an open cone, $f \from A \to \R$ a Lipschitz continuous function, and consider the graph $\Gamma = \set{(x,f(x)) : x \in A}\subset \R^d$ in the case that  $\alpha\Gamma = \Gamma$, for some $\alpha \in (0,1)$; that is, in the case that
\begin{equation} \label{eq:dileq}
f(\alpha x) = \alpha f(x), \quad x\in A.
\end{equation}
We will term such graphs \emph{dilation invariant}.

In this section, the largest of the paper, we study the spectrum and essential spectrum of the double-layer (DL) operator $D_{\Gamma} \from L^2(\Gamma) \to L^2(\Gamma)$ given by \eqref{eq:DD'} on such graphs. (The case $A=\R^{d-1}$ is of particular interest for later applications.) In \S\ref{sec:FB} we show, for  general dimension $d\geq2$, that (as operators on $L^2(\Gamma)$) $\sigma(D_\Gamma)=\sigma_{\ess}(D_\Gamma)$ (and, similarly, that $W(D_\Gamma)=W_{\ess}(D_\Gamma)$) and, by Floquet-Bloch-transform arguments, that
$\sigma(D_\Gamma)$ is the union of the spectra of a family of operators $K_t:L^2(\Gamma_0)\to L^2(\Gamma_0)$, for $t\in [-\pi,\pi]$, where $\Gamma_0$ is a particular relatively closed and bounded subset of $\Gamma$. Moreover, helpful for the later computation of $\sigma(K_t)$, each $K_t$ is compact, and so has a discrete spectrum, if $f\in C^1(A\setminus\{0\})$ (Corollary \ref{cor:K_t_compact}).

In the remaining subsections, \S\ref{sec:oneside}--\S\ref{sec:num_range}, we focus on the 2D case, considering the Nystr\"om approximation of spectral properties of $K_t$, combining the general results of \S\ref{sec:Ny} with explicit estimates for the particular operators $K_t$.  The case $A=\R$, with $f$ real-analytic on $\R\setminus\{0\}$, is treated in \S\ref{sec:two-sided}. It is this case that is relevant, via, e.g., \eqref{eq:spradius}, to the computation of $\sigma(D_\Gamma)$ and the spectral radius conjecture when $\Gamma$ is the boundary of a bounded Lipschitz domain.  But this case is rather complex; the operator $K_t$ is studied by reducing it to a $2\times2$ operator matrix, corresponding to the split of $\R\setminus \{0\}$ into the two half-axes $(-\infty,0)$ and $(0,\infty)$. To get the main ideas across, and prove many of the results we need in a simpler setting, we first study, in \S\ref{sec:oneside}, the easier case $A=(0,\infty)$ with $f$ real-analytic. Subsections \S\ref{sec:oneside} and \S\ref{sec:two-sided} are concerned with computation of the spectrum and spectral radius of $D_\Gamma$. In \S\ref{sec:num_range}, related to the question at the end of \S\ref{sec:src_main}, we also compute lower bounds for $W_{\ess}(D_\Gamma)$, under the same assumptions on $\Gamma$ as in \S\ref{sec:oneside} and \S\ref{sec:two-sided}.

\subsection{Floquet-Bloch transform results} \label{sec:FB}
Let $V_{\alpha} \from L^2(\Gamma) \to L^2(\Gamma)$ be dilation by $\alpha$, that is,
\begin{equation} \label{eq:dilation}
V_{\alpha}\phi(x) = \alpha^{(d-1)/2}\phi(\alpha x), \quad x \in \Gamma.
\end{equation}
$V_{\alpha}$ is unitary and commutes with $D_{\Gamma}$: noting that $n(y)=n(\alpha y)$, for $y\in \Gamma$, we see that, for all $\phi\in L^2(\Gamma)$ and almost every $x \in \Gamma$,
\begin{align*}
D_{\Gamma}V_{\alpha}\phi(x) &= \frac{1}{c_d}\int_{\Gamma} \frac{(x-y) \cdot n(y)}{|x-y|^d} \alpha^{(d-1)/2}\phi(\alpha y) \, \mathrm{d}s(y)\\
&= \frac{\alpha^{(d-1)/2}}{c_d}\int_{\Gamma} \frac{(x-\alpha^{-1}y) \cdot n(\alpha^{-1}y)}{|x-\alpha^{-1}y|^d} \phi(y) \, \alpha^{-(d-1)}\mathrm{d}s(y)\\
&= \frac{\alpha^{(d-1)/2}}{c_d}\int_{\Gamma} \frac{(\alpha x-y) \cdot n(y)}{|\alpha x-y|^d} \phi(y) \, \mathrm{d}s(y)\\
&= V_{\alpha}D_{\Gamma}\phi(x).
\end{align*}

This already implies that the spectrum and the essential spectrum of $D_{\Gamma}$ coincide, as implied by the following simple proposition (cf.~\cite[Lemma 2.7]{ChaSpe:21}).

\begin{prop} \label{prop:spec_equals_ess_spec}
Let $H$ be a Hilbert space and $T\in \cL(H)$. If $T$ commutes with a sequence of unitary operators $(U_j)_{j \in \N}$ that converges weakly to $0$, then $\sigma(T) = \sigma_{\ess}(T)$.
\end{prop}

\begin{proof}
Assume there exists  $\phi \in H \setminus \set{0}$ such that $T\phi = 0$. As $T$ and $U_j$ commute, also $TU_j\phi = 0$ for all $j \in \N$. In particular, $\set{U_j\phi : j \in \N} \subseteq \ker(T)$. As the operators $U_j$ are unitary and $U_j \to 0$ weakly, the sequence $(U_j\phi)_{j \in \N}$ cannot have a convergent subsequence. Hence, $\ker(T)$ is either trivial or infinite-dimensional. Similarly, $\ker(T')$ is either trivial or infinite-dimensional. Thus, if $T$ is Fredholm, it is invertible. Considering $T - \lambda I$ instead of $T$ yields the result.
\end{proof}

\begin{cor} \label{cor:ess_ne}
Let $\Gamma \subset \R^d$
be a dilation invariant graph. Then $\sigma(D_{\Gamma}) = \sigma_{\ess}(D_{\Gamma})$, $\overline{W(D_{\Gamma})} = W_{\ess}(D_{\Gamma})$, and $\norm{D_{\Gamma}}_{\ess} = \norm{D_{\Gamma}}$.
\end{cor}

\begin{proof}
Assume that $\phi,\psi \in L^2(\Gamma)$ have compact support and $0 \notin \supp\phi \cup \supp\psi$. Then
\[\sp{V_{\alpha}^j\phi}{\psi} = \int_{\Gamma} \alpha^{j(d-1)/2}\phi(\alpha^j x)\overline{\psi(x)} \, \mathrm{d}s(x).\]
If $\abs{j}$ is sufficiently large, the integrand vanishes. Because compactly supported functions are dense in $L^2(\Gamma)$, it follows that $V_{\alpha}^j\to 0$ weakly as $\abs{j} \to \infty$. The equality of spectrum and essential spectrum follows from Proposition \ref{prop:spec_equals_ess_spec}. The results for the numerical range and norm follow from \cite[Lemma 2.7]{ChaSpe:21}.
\end{proof}

To make use of standard Floquet-Bloch/Fourier transform results, it is convenient to view $D_\Gamma$ as a discrete convolution operator. For $j\in \Z$ let $\Gamma_j := \set{(\tilde x,f(\tilde x)) \in \Gamma : |\tilde x| \in [\alpha^{j+1},\alpha^j]}$. We can identify $L^2(\Gamma_j)$ with a closed subspace of $L^2(\Gamma)$ by extending by $0$. Let $P_j \from L^2(\Gamma) \to L^2(\Gamma_j)$ denote orthogonal projection. Clearly,
\[P_j\phi(x) = \begin{cases} \phi(x) & \text{if } x \in \Gamma_j, \\ 0 & \text{otherwise,} \end{cases}\]
and we note that $P_jV_{\alpha} = V_{\alpha}P_{j+1}$, so that also $P_jV_\alpha^k = V_\alpha^k P_{j+k}$, $k\in \Z$.
Let
\[D_j := V_{\alpha}^jP_jD_{\Gamma}|_{L^2(\Gamma_0)} : L^2(\Gamma_0) \to L^2(\Gamma_0),\]
so that
\begin{equation} \label{eq:DjDef}
D_j\phi(x) = \alpha^{j(d-1)/2} D_{\Gamma}\phi(\alpha^jx),  \quad x\in \Gamma_0, \;\; \phi\in L^2(\Gamma_0).
\end{equation}
Let $\cG:L^2(\Gamma)\to \ell^2(\Z, L^2(\Gamma_0))$ be the unitary operator\footnote{This is a discretization operator in the sense, e.g., of \cite[\S1.2.3]{Li:06}.} $\phi\mapsto (V_\alpha^jP_j\phi)_{j\in \Z}$, and let $\widetilde D_\Gamma := \cG D_\Gamma \cG^{-1}$. It is a straightforward calculation to see that the action of $\widetilde D_\Gamma$ is that of a discrete convolution: for $\psi=(\psi_n)_{n\in \Z}\in \ell^2(\Z, L^2(\Gamma_0))$, $\widetilde D_\Gamma\psi= ((\widetilde D_\Gamma\psi)_m)_{m\in \Z}$ where
\begin{equation} \label{eq:conv}
(\widetilde D_\Gamma\psi)_m = \sum_{n\in \Z} D_{m-n} \psi_n, \quad m\in \Z.
\end{equation}
The series in the above definition converges absolutely; indeed $\widetilde D_\Gamma$ is an operator in the so-called {\em Wiener algebra}, i.e.\ $\sum_{j\in \Z} \|D_j\|<\infty$ (e.g., \cite[Defn.~1.43]{Li:06}), by the following estimate.

\begin{prop} \label{prop:convergence}
The operators $D_j$ are Hilbert-Schmidt for $|j|\geq 2$ and satisfy the Hilbert-Schmidt norm estimate
\begin{equation} \label{eq:HS_estimates}
\norm{D_j}_{\mathrm{HS}} \leq \frac{1}{c_d} \frac{|\Gamma_0|}{(\alpha - \alpha^{|j|})^{d-1}}\,\alpha^{|j|(d-1)/2}, \quad |j| \geq 2,
\end{equation}
where $|\Gamma_0|$ denotes the surface measure
of $\Gamma_0$.
\end{prop}

\begin{proof}
For $j \geq 2$ we have
\[\int_{\Gamma_j} \int_{\Gamma_0} \abs{\frac{(x-y) \cdot n(y)}{|x-y|^d}}^2 \, \mathrm{d}s(y) \, \mathrm{d}s(x) \leq \int_{\Gamma_j} \int_{\Gamma_0} \frac{1}{|x-y|^{2d-2}} \, \mathrm{d}s(y) \, \mathrm{d}s(x)\]
and $|x-y| \geq \alpha - \alpha^j$. Hence
\[\int_{\Gamma_j} \int_{\Gamma_0} \abs{\frac{(x-y) \cdot n(y)}{|x-y|^d}}^2 \, \mathrm{d}s(y) \, \mathrm{d}s(x) \leq \frac{|\Gamma_j||\Gamma_0|}{(\alpha - \alpha^j)^{2d-2}} = \frac{|\Gamma_0|^2\alpha^{j(d-1)}}{(\alpha - \alpha^j)^{2d-2}}.\]
Similarly, for $j \leq -2$,
\[\int_{\Gamma_j} \int_{\Gamma_0} \abs{\frac{(x-y) \cdot n(y)}{|x-y|^d}}^2 \, \mathrm{d}s(y) \, \mathrm{d}s(x) \leq \frac{|\Gamma_0|^2\alpha^{j(d-1)}}{(\alpha^{j+1} - 1)^{2d-2}} = \frac{|\Gamma_0|^2\alpha^{-j(d-1)}}{(\alpha - \alpha^{-j})^{2d-2}}.\]
Thus, recalling the standard characterisation of the Hilbert-Schmidt norm of integral operators
(e.g., \cite[Ex.~11.11]{Jo:82}),
$P_jD_{\Gamma}|_{L^2(\Gamma_0)}$ is Hilbert-Schmidt, with Hilbert-Schmidt norm bounded by the right hand side of \eqref{eq:HS_estimates}. As unitary operators preserve Hilbert-Schmidt norms, it follows that $D_j$ is Hilbert-Schmidt and that \eqref{eq:HS_estimates} holds.
\end{proof}

Given a Hilbert space $H$ let $\cF:\ell^2(\Z,H)\to L^2([-\pi,\pi],H)$ denote the operator, often termed in the general Hilbert space case, e.g., \cite{ArMaRo22}, a {\em Floquet-Bloch transform}, that constructs a Fourier series from coefficients, given by
$$
(\cF \psi)(t) = (2\pi)^{-1/2} \sum_{j\in \Z} e^{i jt} \psi_n, \quad t\in [-\pi,\pi],
$$
for $\psi=(\psi_j)_{j\in \Z}\in  \ell^2(\Z,H)$. It is standard that this is a unitary operator (e.g., \cite[Proof of Thm.~4.4.9]{BDav})
that diagonalises discrete convolutions (see, e.g., \cite[Thm.~4.49]{BDav} for the case when $H$ is finite-dimensional, \cite[Thm.~2.3.25]{RaRoSi04} for the general case).
In the case $H=L^2(\Gamma_0)$,
defining $\widehat D:L^2([-\pi,\pi],H)\to L^2([-\pi,\pi],H)$ by $\widehat D := \cF \widetilde D_\Gamma \cF^{-1}$, straightforward calculations yield that
\begin{equation} \label{eq:unitary}
\langle\widehat D \phi,\psi\rangle_{L^2([-\pi,\pi],L^2(\Gamma_0))} = \int_{-\pi}^\pi \langle K_t\phi(t),\psi(t)\rangle \, \rd t, \quad \phi,\psi\in L^2([-\pi,\pi],L^2(\Gamma_0)),
\end{equation}
where
\begin{equation} \label{eq:K_t}
K_t := \sum\limits_{j = -\infty}^{\infty} e^{ijt}D_j, \quad t\in \R.
\end{equation}
(The bounds of Proposition \ref{prop:norm_estimate} imply that $K_t$ is well-defined by \eqref{eq:K_t} and depends continuously on $t$; indeed the mapping $t\mapsto K_t$ is $C^\infty$.)
The following characterisation follows immediately from \eqref{eq:unitary}, the continuity of $t\mapsto K_t$, and Corollary \ref{cor:ess_ne}; note that $\conv$ denotes the closed convex hull.

\begin{thm} \label{cor:norm_and_spec}
We have
\begin{align*}
\|D_\Gamma\|_{\ess} &= \|D_{\Gamma}\| = \max\limits_{t \in [-\pi,\pi]} \norm{K_t},\\
W_{\ess}(D_\Gamma) &= \overline{W(D_{\Gamma})} = \conv\left(\bigcup\limits_{t \in [-\pi,\pi]} W(K_t)\right),\\
\sigma_{\ess}(D_\Gamma)&=\spec(D_{\Gamma}) = \bigcup\limits_{t \in [-\pi,\pi]} \spec(K_t).
\end{align*}
\end{thm}
\begin{proof} That the essential spectrum, numerical range, and norm coincide with their non-essential counterparts is Corollary \ref{cor:ess_ne}. Since $D_\Gamma$ and $\widehat D$ are unitarily equivalent, they have the same spectrum, numerical range, and norm. The result thus follows from \eqref{eq:unitary}; see \cite[Thm.~2.3.25]{RaRoSi04} for the case of the spectrum; the argument for the norm and numerical range are similar.
\end{proof}

\begin{rem}[{\em \bf Symmetry of $K_t$}] \label{rem:sym}
Where $K_t(\cdot,\cdot)$ denotes the kernel of $K_t$, $K_{-t}(\cdot,\cdot)=\overline{K_t(\cdot,\cdot)}$, for $t\in [-\pi,\pi]$. Thus $\|K_{-t}\|=\|K_{t}\|$, $W(K_{-t}) = \{\bar z:z\in W(K_t)\}$, and $\sigma(K_{-t})=\{\bar z:z\in \sigma(K_t)\}$, for $t\in [-\pi,\pi]$, so that,
where $w(K_t) := \sup_{z\in W(K_t)}|z|$ is the {\em numerical radius} of $K_t$,
$$
\|D_\Gamma\|_{\ess} = \max_{t\in [0,\pi]}\|K_t\|,\quad w_{\ess}(D_\Gamma) = \max_{t\in [0,\pi]} w(K_t), \quad \rho_{\ess}(D_\Gamma) = \max_{t\in [0,\pi]} \rho(K_t).
$$
\end{rem}

Our focus in the next subsections will be 2D cases where $f$ is analytic on $A\setminus\{0\}$. In such cases, indeed whenever $\Gamma$ is locally $C^1$ away from $0$, $K_t$ is compact for all $t$ as a consequence of standard results on the double-layer operator on $C^1$ domains \cite{FaJoRi:78} and Proposition \ref{prop:norm_estimate}.

\begin{cor} \label{cor:K_t_compact}
Suppose that $f\in C^1(A\setminus\{0\})$ so that $\Gamma$ is locally $C^1$ at each $x\in \Gamma\setminus\{0\}$. Then $K_t$ is compact for every $t \in \R$.
\end{cor}

\begin{proof}
Let $Q = P_{-1} + P_0 + P_1$. Then $QD_{\Gamma}|_{L^2(\Gamma_{-1} \cup \Gamma_0 \cup \Gamma_1)}$ is the DL operator on $\Gamma_{-1} \cup \Gamma_0 \cup \Gamma_1$. It thus follows from \cite[Theorem 1.2]{FaJoRi:78} that $QD_{\Gamma}|_{L^2(\Gamma_{-1} \cup \Gamma_0 \cup \Gamma_1)}$ is compact. As
\[D_j = V_{\alpha}^jP_jD_\Gamma|_{L^2(\Gamma_0)} = V_{\alpha}^jP_jQD_\Gamma|_{L^2(\Gamma_0)},\]
$D_j$ is compact for $j = -1,0,1$. The compactness of $K_t$ thus follows from Proposition \ref{prop:convergence}.
\end{proof}

\subsection{The 2D case: One-sided infinite graphs} \label{sec:oneside}

\begin{figure}[ht]
\centering
\includegraphics[scale=0.6, trim = 0.5cm 0.6cm 0cm 0.6cm, clip]{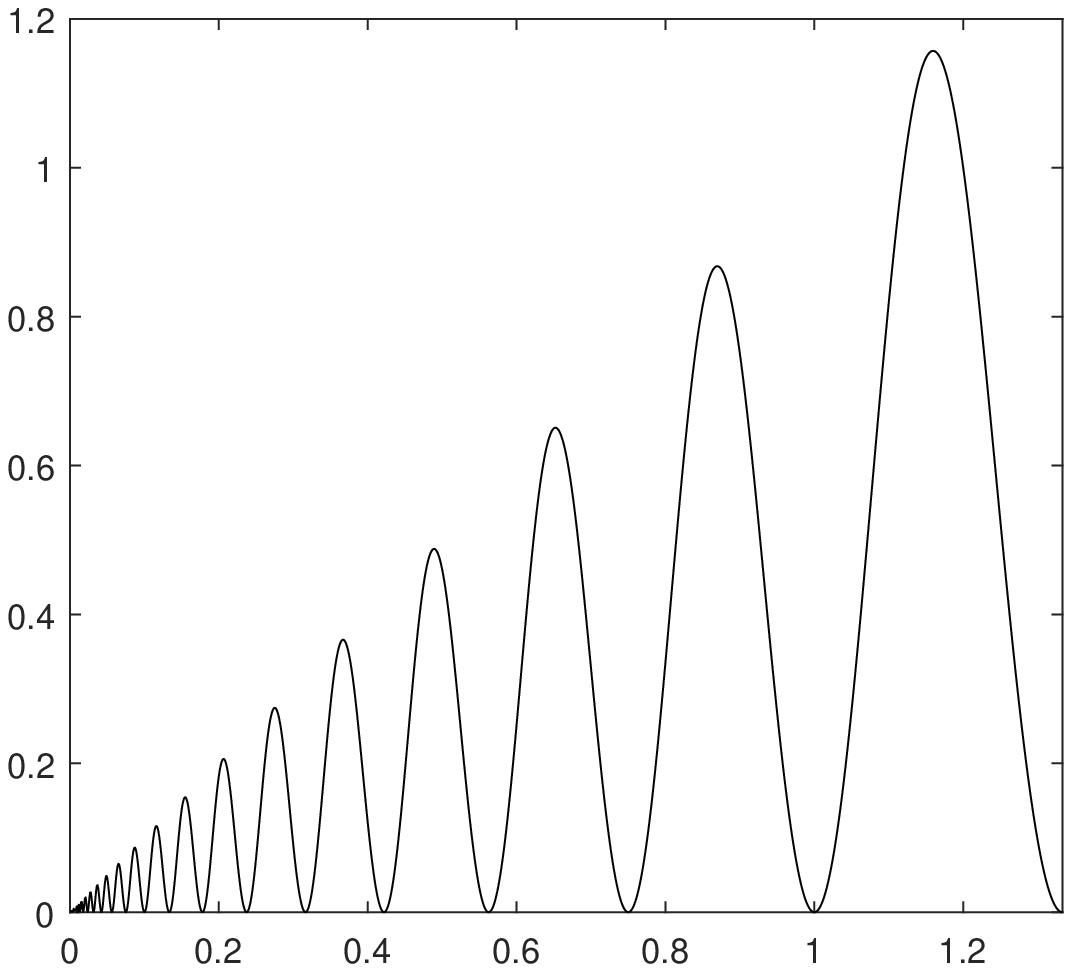}
\caption{Graph of $f \from \R_+ \to \R$, $f(x) := x\sin^2(\pi\log_{\alpha}(x))$, for $\alpha = \frac{3}{4}$} \label{fig:ex1}
\end{figure}
We continue to assume that $\Gamma$ is a  dilation invariant graph, as defined at the start of \S\ref{sec:dilation_invariant}, but specialise now to the case where $d=2$ and the cone is the half-axis $A=\R_+ := (0,\infty)$. Thus $\Gamma = \set{(x,f(x)): x \in \R_+}$ and, for some $\alpha \in (0,1)$,  $f(\alpha x) = \alpha f(x)$, $x\in \R_+$. Our starting point is Corollary \ref{cor:norm_and_spec} which expresses $\sigma(D_\Gamma)$ as the union of the spectra of the operators $K_t$, $t\in [-\pi,\pi]$. Recall that these operators are compact if $f\in C^1(\R_+)$.
Our goal is to apply the Nystr\"om method and the results of \S\ref{sec:Ny} to compute spectral properties of $K_t$ and hence of $D_\Gamma$ in the case that $f\in \cA(\R_+)$, the space of functions $\R_+\to \R$ that are real analytic (a prototypical example is Figure \ref{fig:ex1}). Our standing assumption through this subsection is that
\begin{equation} \label{eq:standing}
\Gamma = \set{(x,f(x)): x \in \R_+} \;\mbox{ where }\; f\in \cA(\R_+) \; \mbox{ and, for some $\alpha \in (0,1)$, } \; f(\alpha x) = \alpha f(x), \quad x\in \R_+.
\end{equation}

Notably, via approximations of the spectrum of $K_t$ for each $t$, we will obtain (see Theorem \ref{thm:Hausdorff_convergence}) a Nystr\"om approximation $\sigma^N(D_\Gamma)$, for $\sigma(D_\Gamma)=\sigma_{\ess}(D_\Gamma)$, which is $\{0\}$ plus the union over finitely many $t\in [-\pi,\pi]$ of the spectra of $N\times N$ matrices $A^M_{t,N}$, where each $A^M_{t,N}$ is obtained via Nystr\"om discretisation of a unitary transformation, $\tilde K_t$, of $K_t$. Our first main result, Theorem \ref{thm:Hausdorff_convergence}, is to show that $\sigma^N(D_\Gamma)\toH \sigma(D_\Gamma)$ as $N\to \infty$. Our other, more substantial result (Theorem \ref{thm:spectral_radius}) is to develop a fully discrete algorithm to test whether the spectral radius conjecture holds for $\Gamma$, i.e.\ to test whether $\sigma_{\ess}(D_\Gamma)=\sigma(D_\Gamma)<\half$. This algorithm, which derives from Theorem \ref{thm:GspecRMfinal}, requires the computation only of the spectral radii of finitely many finite matrices plus the norms of finitely many finite matrix resolvents.

Before we begin our analysis we note the following equivalences to \eqref{eq:standing}
that will play a key role in our calculations. Here, and throughout, the notations
\begin{equation} \label{eq:Sigmac}
\Sigma_c := \set{z \in \C : \Imag z \in (-c,c)}, \quad \Sigma_0 := \R,
\end{equation}
for $c>0$ will be convenient.

\begin{lem} \label{lem:f_anal} Given $f:\R_+\to \R$ and $\alpha \in (0,1)$, define $g:\R\to\R$ by
\begin{equation} \label{eq:gdef}
g(x) := \alpha^{-x}f(\alpha^x), \quad x\in \R, \quad \mbox{so that} \quad f(x) = xg(\log_\alpha x), \quad x\in \R_+.
\end{equation}
Then the following are equivalent:
\begin{enumerate}
\item[i)] $f$ is real analytic on $\R_+$ and $f(\alpha x) = \alpha f(x)$, $x>0$;
\item[ii)] $g:\R\to \R$ is real analytic and $g(x+1)=g(x)$, $x\in \R$;
\item[iii)] for some $c>0$, $g$ has an analytic extension to $\Sigma_c$ that satisfies  $g(z+1)=g(z)$, for $z\in \Sigma_c$, and $g$ and its derivatives $g^\prime$ and $g^{\prime\prime}$ are bounded in $\Sigma_c$.
\end{enumerate}
\end{lem}

Note that if $f$ satisfies our standing assumption \eqref{eq:standing}, then, defining $g$ by \eqref{eq:gdef},
\begin{equation} \label{eq:fpr}
f^\prime(x) = g(\log_\alpha x) + \frac{g^\prime(\log_\alpha x)}{\log \alpha}, \quad f^{\prime\prime}(x) = \frac{g^\prime(\log_\alpha x)}{x\log \alpha}+\frac{g^{\prime\prime}(\log_\alpha x)}{x\log^2 \alpha}, \quad x>0.
\end{equation}
It follows from \eqref{eq:gdef}, the first of \eqref{eq:fpr}, and the equivalence of i) and iii), that $f$ is Lipschitz continuous on $[0,\infty)$ if we set $f(0):=0$.

To make use of the results from \S\ref{sec:Ny} it is convenient to make a change of variables so that we work with integral operators on $[0,1]$ rather than $\Gamma_0$. Introducing the unitary transformation $U \from L^2(\Gamma_0) \to L^2(0,1)$ given by
\begin{equation} \label{eq:Udef}
U\phi(s) := \phi(\alpha^s,f(\alpha^s))(1+f'(\alpha^s)^2)^{1/4}\alpha^{s/2}\abs{\log\alpha}^{1/2}, \quad s \in [0,1], \;\; \phi\in L^2(\Gamma_0),
\end{equation}
define
\begin{equation} \label{eq:tilK}
\tilde{K}_t := UK_tU^{-1} = \sum_{j=-\infty}^\infty e^{ijt} UD_jU^{-1},  \quad t\in \R.
\end{equation}
Straightforward computations, starting from \eqref{eq:tilK}, \eqref{eq:DjDef}, and \eqref{eq:DD'2}, give that
\begin{align} \label{eq:kernel_computation}
	\tilde K_t \phi(x) &= \int_0^1 \tilde K_t(x,y) \phi(y)\, \rd y, \quad x\in [0,1], \;\; t\in \R,
\end{align}
where, for $t\in \R$, and $x,y\in \R$ with $x-y\not\in \Z$,
\begin{equation} \label{eq:ktall}
\tilde{K}_t(x,y) = \frac{1}{2\pi}\sum\limits_{j = -\infty}^{\infty} e^{ijt}\,\frac{p_j(x,y)}{1+q_j(x,y)^2}\left(\frac{1+f'(\alpha^{x})^2}{1+f'(\alpha^y)^2}\right)^{1/4}\alpha^{\frac{x+y+j}{2}}\abs{\log\alpha},
\end{equation}
with
\begin{equation} \label{eq:pjdef}
p_j(x,y) := \frac{(\alpha^y-\alpha^{x+j})f^\prime(\alpha^y) + f(\alpha^{x+j}) - f(\alpha^y)}{(\alpha^{x+j}-\alpha^y)^2}, \quad q_j(x,y) := \frac{f(\alpha^{x+j})-f(\alpha^y)}{\alpha^{x+j}-\alpha^y}.
\end{equation}
By Taylor's theorem applied to $F(t):= f((1-t)\alpha^y+t\alpha^{x+j})$, we see that, for the same range of $x$ and $y$,
\begin{align} \label{Tay1}
q_j(x,y) &= \int_0^1 f^\prime((1-t)\alpha^y+t\alpha^{x+j})\, \rd t,\\ \label{Tay2}
p_j(x,y) &= \int_0^1 f^{\prime\prime}((1-t)\alpha^y+t\alpha^{x+j})(1-t)\, \rd t.
\end{align}
Using \eqref{Tay1} and \eqref{Tay2} to extend the definitions of $q_j(x,y)$ and $p_j(x,y)$ to $\{(x,y)\in \R^2:x-y\in \Z\}$,\footnote{Precisely, for $j\in \Z$, $p_j(x,y)$ and $q_j(x,y)$ are given explicitly by \eqref{eq:pjdef} for $y-x\neq j$, while \eqref{Tay1} and \eqref{Tay2} imply that $p_j(x,y) =\half f^{\prime\prime}(\alpha^y)$, $q_j(x,y) = f^\prime(\alpha^y)$, for $y-x=j$.}
we see that, for each $t\in \R$, each term in the sum \eqref{eq:ktall} is continuous on $\R^2$. Further, using \eqref{eq:fpr} and the equivalence of i) and iii) in Lemma \ref{lem:f_anal}, it is easy to see that $p_j(x,y) = O(1)$ as $j\to \infty$, $=O(\alpha^{-j})$ as $j\to -\infty$, uniformly for $x$ and $y$ in compact subsets of $\R$, so that the series \eqref{eq:ktall} converges absolutely and uniformly on compact subsets, so that $\tilde K_t(\cdot,\cdot)\in C(\R^2)$. Further,
$f(\alpha x) = \alpha f(x)$ implies that $\tilde{K}_t(x+1,y) = e^{-it}\tilde{K}_t(x,y)$ and $\tilde{K}_t(x,y+1) = e^{it}\tilde{K}_t(x,y)$,  $x,y \in \R$. Note that $\tilde K_t(\cdot,\cdot)\in C(\R^2)$ implies that $\tilde K_t:L^2(0,1)\to C[0,1]$, that $\tilde K_t$ is compact as an operator on both $C[0,1]$ and $L^2(0,1)$, and that the spectrum of $\tilde K_t$ is the same on $C[0,1]$ as on $L^2(0,1)$, so that
\begin{equation}\label{eq:SpecSame}
\sigma(K_t;L^2(\Gamma_0)) = \sigma(\tilde K_t;L^2(0,1))=\sigma(\tilde K_t;C[0,1]), \quad t\in \R.
\end{equation}

We will use the above equivalence, and the Nystr\"om method results from \S\ref{sec:Ny}, to compute the spectrum and spectral radius of $K_t$, in particular by applying Theorem \ref{thm:GspecRMfinal}. As a step towards estimating the left hand side of \eqref{eq:keyboundfinal} in the case when $K=\tilde K_t$ and $K_N$ is its Nystr\"om method approximation,  we first show that, under our standing assumption \eqref{eq:standing}, $\tilde K_t(x,\cdot)$ and $\tilde K_t(\cdot,y)$ can be extended to bounded analytic functions on $\Sigma_c$ for some $c>0$. We will use the notation
 $\norm{\cdot}_c$ for the sup norm on $\Sigma_c$ and, for functions $h \from \Sigma_c \times \Sigma_d \to \C$, the notation
\[\norm{h}_{c,d} := \sup\limits_{\substack{x \in \Sigma_c \\ y \in \Sigma_d}} \abs{h(x,y)}.\]
In an extension of these notations, for  $h \from \Sigma_c \times \Sigma_c \to \C$ we define
\[\norm{h}_{c,0} := \sup\limits_{\substack{x \in \Sigma_c \\ y \in \R}} \abs{h(x,y)} \quad \text{and} \quad \norm{h}_{0,c} := \sup\limits_{\substack{x \in \R \\ y \in \Sigma_c}} \abs{h(x,y)},\]
and, for functions $h:\Sigma_c\to \C$, we set $\|h\|_0:=\|h\|_\infty:=\sup_{x\in \R}|h(x)|$, the ordinary sup norm on $\R$.

Let us now estimate the norms $\|\tilde{K}_t\|_{c,0}$ and $\|\tilde{K}_t\|_{0,c}$ under our standing assumption \eqref{eq:standing}. Since $g$ is real-valued on $\R$, it follows, using Lemma \ref{lem:equiv},   that \eqref{c:bound2} holds for all sufficiently small $c>0$.

\begin{prop} \label{prop:k_t_estimate general}
Given \eqref{eq:standing}, define $g$ by \eqref{eq:gdef} so that, by Lemma \ref{lem:f_anal},  $g$ has an analytic continuation to $\Sigma_c$, for some $c>0$, such that $g$, $g^\prime$, and $g^{\prime\prime}$ are bounded on $\Sigma_c$ and $g(z+1)=g(z)$, $z\in \Sigma_c$. Let $\fI_c := \|\Imag g\|_c + \|\Imag g^\prime\|_c/|\log \alpha|$, and set
\begin{align} \label{fFdef}
\fF_d := \|g\|_d + \|g^\prime\|_d/|\log \alpha| \quad \mbox{and} \quad \fG_d := \|g^\prime\|_d + \|g^{\prime\prime}\|_d/|\log \alpha|,   \quad \mbox{for } d=0,c.
\end{align}
If
\begin{equation} \label{c:bound2}
c \leq \arccos(\alpha)/\abs{\log\alpha} \quad \mbox{and} \quad \fI_c < 1,
\end{equation}
then $\tilde{K}_t(x,\cdot)$ and $\tilde{K}_t(\cdot,y)$ extend to bounded analytic functions on $\Sigma_c$ for all $x,y,t\in \R$, and
\begin{align}\label{ktb1}
\|\tilde{K}_t\|_{c,0} &\leq \frac{\left(1+\fF_c^2\right)^{1/4}}{\pi\left(1- \fI_c^2\right)}\Bigg[ \fG_c\, \frac{1+\alpha^{1/2}+\alpha^{-1/2}}{4\alpha^2} +   |\log\alpha|(\fF_c+\fF_0) \, \sum_{j=2}^\infty\frac{\alpha^{j/2}}{\alpha-\alpha^j}\Bigg],\\ \label{ktb2}
\|\tilde{K}_t\|_{0,c} &\leq \frac{\left(1+\fF_0^2\right)^{1/4}}{\pi\left(1- \fI_c^2\right)^{5/4}}\Bigg[ \fG_c\, \frac{1+\alpha^{1/2}+\alpha^{-1/2}}{4\alpha^2} +   2|\log\alpha|\,\fF_c \, \sum_{j=2}^\infty\frac{\alpha^{j/2}}{\alpha-\alpha^j}\Bigg].
\end{align}
Moreover, $\R \to C([0,1]\times [0,1])$, $t \mapsto \tilde{K}_t(\cdot,\cdot)$ is continuous.
\end{prop}

\begin{rem}[{\bf \em Bound on the sum in \eqref{ktb1} and \eqref{ktb2}}] \label{rem:series}
For fixed $\alpha\in (0,1)$, let $F(x):= \alpha^{x/2}/(\alpha-\alpha^x)$, for $x>0$, so that the $j$th term in the sum in \eqref{ktb1} and \eqref{ktb2} is $F(j)$.
Since $F$ is decreasing on $(1,\infty)$ we have that, for $n=2,3,\ldots$ and $0<\alpha<1$,
\begin{eqnarray} \label{eq:special_fun_error_est}
\sum_{j=n+1}^\infty\frac{\alpha^{j/2}}{\alpha-\alpha^j} &\leq &\mathcal{B}_n(\alpha)
\quad \mbox{so that} \\ \label{eq:twosided}
\sum_{j=2}^\infty\frac{\alpha^{j/2}}{\alpha-\alpha^j} &\leq &\mathcal{B}^*_n(\alpha) := \sum_{j=2}^n F(j) + \mathcal{B}_n(\alpha),
\end{eqnarray}
where
$$
\mathcal{B}_{n}(\alpha) := \int_{n}^\infty F(x) \,\rd x = \frac{1}{2\alpha^{1/2}}\int_{n}^\infty \frac{\rd x}{\sinh\left(|\log \alpha|(x-1)/2\right)}= \frac{\log(\tanh((n-1)|\log\alpha|/4))}{\alpha^{1/2}\log\alpha}.
$$
We note, for later reference, that
\begin{equation} \label{eq:BnAsy}
\mathcal{B}_{n}(\alpha) \sim 2\alpha^{n/2-1}/|\log\alpha|, \qquad \mbox{as} \qquad n\to\infty.
\end{equation}
\end{rem}

\begin{proof}[Proof of Proposition \ref{prop:k_t_estimate general}]
Using \eqref{eq:fpr} it follows that
\begin{equation} \label{eq:falpha}
f(\alpha^x) = \alpha^x g(x), \quad f^\prime(\alpha^x) = g(x) + \frac{g^\prime(x)}{\log \alpha}, \quad f^{\prime\prime}(\alpha^x) = \frac{g^\prime(x)}{\alpha^x\log \alpha}+\frac{g^{\prime\prime}(x)}{\alpha^x\log^2 \alpha}, \quad x\in \R.
\end{equation}
Note that the first of the bounds \eqref{c:bound2} implies that $c|\log\alpha| < \pi/2$, so that the assumptions we have made on $g$ mean that \eqref{eq:gdef} provides an analytic extension of $f$ from $\R_+$ to the sector of the complex plane
$$
G_c:=\{re^{i\theta}: r>0, |\theta| <c|\log\alpha|\} = \{\alpha^z:z\in \Sigma_c\},
$$
and the equations \eqref{eq:fpr} hold for all $x$ in this sector. Further, with this extension, \eqref{eq:falpha} holds for all $x\in \Sigma_c$.
Since \eqref{eq:fpr} holds for all $x\in G_c$, and noting that
$G_c$ is convex, we see that the integrals \eqref{Tay1} and \eqref{Tay2} are well-defined for all $x,y\in \Sigma_c$, $j\in \Z$, and provide analytic continuations of $p_j$ and $q_j$ to $\Sigma_c\times \Sigma_c$. Further, by the uniqueness of analytic continuation, the equations \eqref{eq:pjdef} hold for all $x,y\in \Sigma_c$, with $x+j\neq y$.

To complete the proof we will demonstrate that \eqref{eq:ktall} provides, for each $x,y\in \R$, analytic continuations of $\tilde{K}_t(\cdot,y)$ and $\tilde{K}_t(x,\cdot)$ from $\R$ to $\Sigma_c$ which satisfy the bounds \eqref{ktb1} and \eqref{ktb2}, by showing that each term in \eqref{eq:ktall} is well-defined (i.e., that $f^\prime(\alpha^y)^2\neq -1$, for $y\in \Sigma_c$, and $1+(q_j(x,y))^2\neq 0$ for $x,y\in \Sigma_c$), and that the series \eqref{eq:ktall} converges absolutely and uniformly for $(x,y) \in \R \times \Sigma_c$ and $(x,y)\in \Sigma_c \times \R$.

Using \eqref{eq:falpha}  we see that
$$
\sup_{z\in G_c}|\Imag f^\prime(z)| = \sup_{y\in \Sigma_c}|\Imag f^\prime(\alpha^y)| \leq \fI_c,
$$
so that
\begin{align} \label{eq:Icbound}
\inf\limits_{y \in \Sigma_c} \abs{1+f'(\alpha^y)^2} &\geq 1 - \sup\limits_{y \in \Sigma_c} \abs{\Imag f'(\alpha^y)}^2 \geq 1 - \fI_c^2,
\end{align}
and, using \eqref{Tay1},
\begin{equation} \label{eq:qbound}
\inf_{x,y\in \Sigma_c} \left|1+q_j(x,y)^2\right| \geq 1 - \sup_{z\in G_c}|\Imag f^\prime(z)|^2 \geq 1 - \fI_c^2.
\end{equation}
Thus, where $T_j(x,y)$ denotes the $j$-th term in the sum \eqref{eq:ktall}, we see that $T_j(x,y)$ is well-defined and analytic for $x,y\in \Sigma_c$, $j\in \Z$. Moreover, we have $T_j(x+1,y)=e^{-it}T_{j+1}(x,y)$ and it follows from $f(\alpha x) = \alpha f(x)$ that $T_j(x,y+1)=e^{it}T_{j-1}(x,y)$, for $x,y\in \Sigma_c$, $j \in \Z$. Thus, to prove absolute and uniform convergence of \eqref{eq:ktall} and the bounds \eqref{ktb1} and \eqref{ktb2}, it suffices to restrict consideration to $x,y \in \Xi_c := \{z \in \C :$ $\Real z \in [0,1],$ $\Imag z \in (-c,c)\}$.

We see that $\sup_{x\in \Xi_c}|\alpha^{x/2}|\leq 1$ and, using \eqref{eq:falpha}, that
\begin{align*}
\sup\limits_{x \in [0,1]} \abs{1+f'(\alpha^x)^2} \leq 1 + \fF_0^2,
\quad
\sup\limits_{x \in \Xi_c} \abs{1+f'(\alpha^x)^2} \leq 1 + \fF_c^2.
\end{align*}
To obtain a bound on $|T_j(x,y)|$ for $(x,y) \in [0,1] \times \Xi_c$, and for $(x,y) \in \Xi_c \times [0,1]$, it remains to bound $p_j(x,y)$. Let
$$
G_c^* := \{re^{i\theta}: \alpha^{2}\leq r\leq \alpha^{-1}, |\theta| <c|\log\alpha|\} = \{\alpha^z:\Real z\in [-1,2], \, \Imag z\in (-c,c)\}.
$$
It is clear that $(1-t) \alpha^y+t\alpha^{x+j}\in G_c^*$ for every $t\in [0,1]$, $|j|\leq 1$, $x\in \Xi_c$, and $y\in [0,1]$ if and only if $G_c^*$ is star-shaped with respect to every point in $[\alpha,1]$, which holds if and only if $c$ satisfies the first of the bounds \eqref{c:bound2}. Likewise, $(1-t) \alpha^y+t\alpha^{x+j}\in G_c^*$ for every  $t\in [0,1]$, $|j|\leq 1$, $x\in [0,1]$, and $y\in \Xi_c$ if and only if the first of \eqref{c:bound2} holds. Thus, if \eqref{c:bound2} holds, $|j|\leq 1$, and either  $(x,y) \in [0,1] \times \Xi_c$ or $(x,y) \in \Xi_c \times [0,1]$, it follows from \eqref{eq:fpr} and \eqref{Tay2} that
$$
|p_j(x,y)| \leq \frac{1}{2} \sup_{z\in G_c^*}|f^{\prime\prime}(z)| \leq \frac{\fG_c}{2\alpha^2|\log\alpha|}.
$$
On the other hand, if $|j|\geq 2$ and $x,y\in \Xi_c$, then it follows from \eqref{eq:pjdef} (which we have observed above holds for all $x,y\in \Sigma_c$ with $x+j\neq y$) that
$$
|p_j(x,y)| \leq \frac{|f^\prime(\alpha^y)| + |q_j(x,y)|}{|\alpha^{x+j}-\alpha^y|}.
$$
Further, for $j\in \Z$,
$$
|q_j(x,y)| \leq \sup_{z\in G_c} |f^{\prime}(z)| \leq \fF_c,
\quad x,y\in \Xi_c,
$$
while $|f^\prime(\alpha^y)| \leq \fF_0$, for $y\in [0,1]$, $\leq \fF_c$, for $y\in \Xi_c$.
Moreover, for $|j|\geq 2$, $x,y\in \Xi_c$,
$$
|\alpha^{x-j}-\alpha^y| \geq \left| \alpha^{\Real x - j} - \alpha^{\Real y} \right|\geq \left\{\begin{array}{cc}
                                   \alpha-\alpha^j, & j\geq 2, \\
                                   \alpha^{j+1}-1, & j\leq -2.
                                 \end{array}\right.
$$
Putting these bounds together we see that, for $|j|\leq 1$,
\begin{align*}
\sup\limits_{x\in\Xi_c, y\in [0,1]}|T_j(x,y)| \leq \frac{\fG_c \left(1+\fF_c^2\right)^{1/4}\alpha^{j/2}}{2\alpha^2\left(1- \fI_c^2\right)},\qquad
\sup\limits_{x\in[0,1], y\in \Xi_c}|T_j(x,y)| \leq \frac{\fG_c \left(1+\fF_0^2\right)^{1/4}\alpha^{j/2}}{2\alpha^2\left(1- \fI_c^2 \right)^{5/4}},
\end{align*}
while, for $|j|\geq 2$,
\begin{align*}
&\sup\limits_{x\in\Xi_c, y\in [0,1]}|T_j(x,y)| \leq \frac{|\log\alpha|(\fF_0+\fF_c)\left(1+\fF_c^2 \right)^{1/4}\alpha^{|j|/2}}{\left(1- \fI_c^2 \right)(\alpha-\alpha^{|j|})},\\
&\sup\limits_{x\in[0,1], y\in \Xi_c}|T_j(x,y)| \leq \frac{2|\log\alpha|\,\fF_c\left(1+\fF_c^2 \right)^{1/4}\alpha^{|j|/2}}{\left(1- \fI_c^2 \right)^{5/4}(\alpha-\alpha^{|j|})}.
\end{align*}
From these bounds on $|T_j(x,y)|$ it is clear that the series \eqref{eq:ktall}, with $p_j$ and $q_j$ given by \eqref{Tay1} and \eqref{Tay2}, converges absolutely and uniformly for $(x,y) \in [0,1] \times \Xi_c$ and $(x,y) \in \Xi_c \times [0,1]$, so that $\tilde{K}_t(x,\cdot)$ and $\tilde{K}_t(\cdot,y)$ are analytic in $\Sigma_c$, for $x,y\in \R$, as required. The convergence of the series is also uniform in $t$, which implies the continuity of $t \mapsto \tilde{K}_t(\cdot,\cdot)$. Further, the above bounds on $|T_j(x,y)|$ imply the bounds \eqref{ktb1} and \eqref{ktb2}.
\end{proof}

Recalling \eqref{eq:SpecSame} and that the kernel of $\tilde K_t$ is continuous, we will approximate the spectrum and spectral radius of $K_t$, for $t\in \R$, by approximating the spectrum and spectral radius of $\tilde K_t$, considered as an operator on $C[0,1]$, using the Nystr\"om method results of \S\ref{sec:Ny}. Define $\tilde{K}_{t,N} \from C[0,1]\to C[0,1]$, for $N\in \N$ and $t\in \R$, by (cf.~\eqref{eq:Nyst})
\begin{equation} \label{eq:KtnDef}
\tilde{K}_{t,N}\phi(x) := J_N(\tilde{K}_t(x,\cdot)\phi) = \frac{1}{N}\sum_{q=1}^N \tilde{K}_t(x,x_{q,N}) \phi(x_{q,N}), \quad x \in [0,1],\;\; \phi\in C[0,1],
\end{equation}
where
$$
J_N\psi(x) := \sum_{q=1}^N \omega_{q,N}\phi(x_{q,N}), \quad \psi\in C[0,1], \;\; N\in \N,
$$
with $x_{q,N} := \frac{1}{N}(q-\frac{1}{2})$, $\omega_{q,N} := N^{-1}$, for $q=1,\ldots,N$, $N\in \N$. Note that $J_N\psi$ is just the $N$-point midpoint rule approximation to $J\psi := \int_0^1\psi(x)\, \rd x$, and that this numerical quadrature rule sequence satisfies \eqref{eq:JNtoJ} and \eqref{eq:Xsum}, so that, for each $t\in \R$, $\tilde K_{t,N}\to \tilde K_t$ as $N\to\infty$ and $(\tilde K_{t,N})_{N\in \N}$ is collectively compact. This implies, from the general result \eqref{eq:pointwise}, that $\|(\tilde{K}_t - \tilde{K}_{t,N})\tilde{K}_t\|_{\infty}\to 0$ as $N\to \infty$. In our case this convergence is exponential (cf.~\cite[Ex.~12.11]{Kr:14}, \cite[\S19]{TrWe14}).

\begin{prop} \label{prop:norm_estimate}
Let $f$ and $c$ be as in Proposition \ref{prop:k_t_estimate general}. Then, for every $t \in \R$ and $N\in \N$,
\[
\norm{(\tilde{K}_t - \tilde{K}_{t,N})\tilde{K}_t}_{\infty}
\leq \frac{2 e^{2\pi c}\|\tilde{K}_t\|_{0,c}\|\tilde{K}_t\|_{c,0}}{e^{2\pi Nc}-1}.\]
\end{prop}

\noindent To prove this proposition we will use the following classical result.

\begin{thm}[Theorem 9.28 in \cite{Kr:98}] \label{thm:Kress}
Let $\psi \from \R \to \C$ be a $1$-periodic function that can be extended to a bounded analytic function on $\Sigma_c$ for some $c > 0$. Then
\[\abs{J\psi-J_N\psi}=\abs{\int_0^1 \psi(x) \, \mathrm{d}x - \frac{1}{N}\sum\limits_{q = 1}^N \psi(x_{q,N})} \leq \frac{2\norm{\psi}_c}{e^{2\pi N c}-1}, \qquad N\in \N.\]
\end{thm}

\begin{proof}[Proof of Proposition \ref{prop:norm_estimate}]
It is clear from the definitions of $\tilde K_t$ and $\tilde K_{t,N}$ that it is enough to consider the case $t\in [-\pi,\pi]$. Assuming $t\in [-\pi,\pi]$, let $M_{g_t} \from C[0,1] \to C[0,1]$ be multiplication by $g_t(x) := e^{ixt}$. Then $M_{g_t}$ is a surjective isometry with $M_{g_t}^{-1} = M_{g_{-t}}$. Let $\tilde{L}_t := M_{g_t}\tilde{K}_tM_{g_t}^{-1}$ and $\tilde{L}_{t,N} := M_{g_t}\tilde{K}_{t,N}M_{g_t}^{-1}$.
 Then $\tilde{L}_t$ is an integral operator with kernel $\tilde{L}_t(x,y) = e^{i(x-y)t}\tilde{K}_t(x,y)$. As $\tilde{K}_t(x+1,y) = e^{-it}\tilde{K}_t(x,y)$ and $\tilde{K}_t(x,y+1) = e^{it}\tilde{K}_t(x,y)$, we get $\tilde{L}_t(x+1,y) = \tilde{L}_t(x,y)= \tilde{L}_t(x,y+1)$ for all $x,y \in \R$. Moreover, $\|\tilde{L}_t\|_{0,c} \leq e^{\pi c}\|\tilde{K}_t\|_{0,c}$ and $\|\tilde{L}_t\|_{c,0} \leq e^{\pi c}\|\tilde{K}_t\|_{c,0}$. For $\phi \in C[0,1]$,
\[\psi(x):=  \int_0^1 \tilde{L}_t(x,y)\phi(y) \, \mathrm{d}y, \qquad x\in \Sigma_c,\]
is an analytic and $1$-periodic extension of $\tilde L_t\phi$ from $[0,1]$ to $\Sigma_c$. Thus also $y \mapsto \tilde{L}_t(x,y)\psi(y)$ is analytic and $1$-periodic for every $x \in \R$. By Theorem \ref{thm:Kress} we obtain that, for $x\in \R$,
\[|(\tilde L_t-\tilde L_{t,N})\tilde L_t\phi(x)| =|\tilde{L}_t\psi(x) - \tilde{L}_{t,N}\psi(x)| = |J(\tilde L_t(x,\cdot) \psi))- J_N(\tilde L_t(x,\cdot) \psi))|\leq \frac{2 C_t(x)}{e^{2\pi Nc}-1},\]
where $C_t(x) := \sup_{y \in \Sigma_c} |\tilde{L}_t(x,y)\psi(y)|$. As $\abs{\psi(y)} \leq \sup_{z \in [0,1]} |\tilde{L}_t(y,z)|\norm{\phi}_{\infty}$, for $y\in \Sigma_c$, we get
\[C_t(x) \leq \|\tilde{L}_t\|_{0,c}\|\tilde{L}_t\|_{c,0}\norm{\phi}_{\infty} \leq e^{2\pi c}\|\tilde{K}_t\|_{0,c}\|\tilde{K}_t\|_{c,0}\norm{\phi}_{\infty},\]
for all $x \in [0,1]$ and $\phi \in C[0,1]$. Therefore,
\[\norm{(\tilde{K}_t - \tilde{K}_{t,N})\tilde{K}_t}_{\infty} = \norm{(\Tilde{L}_t - \tilde{L}_{t,N})\tilde{L}_t}_{\infty} \leq \frac{2 e^{2\pi c}\|\tilde{K}_t\|_{0,c}\|\tilde{K}_t\|_{c,0}}{e^{2\pi Nc}-1}. \qedhere\]
\end{proof}

Let $A_{t,N}$ be the $N \times N$-matrix defined by (cf.\,\eqref{eq:A_{t,N}})
\begin{equation} \label{eq:AtNdef}
A_{t,N}(p,q) := \omega_{q,N} \tilde{K}_t(x_{p,N},x_{q,N}) =  \frac{1}{N}\tilde{K}_t(x_{p,N},x_{q,N}), \quad p,q = 1,\ldots,N.
\end{equation}
$\tilde{K}_t(\cdot,\cdot)$ is defined by the series \eqref{eq:ktall} which we truncate to evaluate numerically. For $M \in \N$ let
\begin{equation} \label{eq:B_t^M}
\tilde K_t^M(x,y) := \frac{1}{2\pi}\sum\limits_{j = -M}^{M} e^{ijt}\,\frac{p_j(x,y)}{1+q_j(x,y)^2}\left(\frac{1+f'(\alpha^{x})^2}{1+f'(\alpha^y)^2}\right)^{1/4}\alpha^{\frac{x+y+j}{2}}\abs{\log\alpha}, \qquad x,y\in [0,1],
\end{equation}
where $p_j$ and $q_j$ are defined by \eqref{eq:pjdef},
and consider the matrices $A_{t,N}^M$ given by $A_{t,N}^M(p,q) := \frac{1}{N}\tilde K_t^M(x_{p,N},x_{q,N})$, $p,q = 1,\ldots,N$, so that $A_{t,N}^M$ is an approximation to $A_{t,N}$ obtained by using finitely many terms in the series defining $\tilde K_t(\cdot,\cdot)$.  Similarly, define $\tilde K^M_{t,N}$ by \eqref{eq:KtnDef} with $\tilde K_{t,N}$ and $\tilde K_t(\cdot,\cdot)$ replaced by $\tilde K^M_{t,N}$ and $\tilde K^M_t(\cdot,\cdot)$, respectively.
For $M,N\in \N$, with $M\geq 2$, and $t\in \R$, by \eqref{eq:NormEst},
\begin{equation} \label{eq:eMbound}
\norm{A_{t,N} - A_{t,N}^M}_{\infty} \leq \|\tilde K_{t,N}-\tilde K_{t,N}^M\|_\infty\leq \sup_{x,y\in [0,1]}|\tilde K_t(x,y)-\tilde K^M_t(x,y)| \leq C_1(M),
\end{equation}
where, using the notations of \eqref{fFdef} and Remark \ref{rem:series},
\begin{equation} \label{eq:C1def}
C_1(M):=\frac{2|\log \alpha|\fF_0}{\pi}\left(1+\fF_0^2\right)^{1/4} \, \mathcal{B}_M(\alpha);
\end{equation}
this bound \eqref{eq:eMbound}-\eqref{eq:C1def} (cf.~\eqref{ktb1}) is obtained as in Proposition \ref{prop:k_t_estimate general} (set $c = 0$, only take the terms with $|j| \geq M+1$, and note that $\mathfrak{I}_0=0$),  and by using \eqref{eq:special_fun_error_est}. Notice that, by \eqref{eq:BnAsy}, $C_1(M)=O(\alpha^{M/2})$ as $M\to \infty$.

Our aim now is to estimate $\rho(D_\Gamma;L^2(\Gamma))$ and $\sigma(D_\Gamma;L^2(\Gamma))$ by computing $\rho(A^M_{t,N})$ and $\sigma(A^M_{t,N})$ for only finitely many $t$. For this purpose bounds on the Lipschitz constants of $A^M_{t,N}$ and $\tilde K_{t,N}$ as  functions of $t$ will be helpful.
It follows from \eqref{eq:B_t^M} that, for $M,N\in \N$, $t\in \R$, and $p,q=1,\ldots,N$,
\begin{eqnarray} \label{eq:BNMdef}
\abs{\frac{\partial}{\partial t} A_{t,N}^M(p,q)} &\leq &B_N^M(p,q) \\ \nonumber
&:=& \frac{1}{2\pi N}\sum\limits_{j = -M}^{M} \frac{|j|\,\left|p_j(x_{p,N},x_{q,N})\right|}{1+q_j(x_{p,N},x_{q,N})^2}\left(\frac{1+f'(\alpha^{x_{p,N}})^2}{1+f'(\alpha^{x_{q,N}})^2}\right)^{1/4}\alpha^{\frac{x_{p,N}+x_{q,N}+j}{2}}\abs{\log\alpha},
\end{eqnarray}
so that, for $s,t\in \R$, $M,N\in \N$,
\begin{equation} \label{eq:C2def}
\norm{A^M_{t,N}-A^M_{s,N}}_\infty \leq |s-t| \norm{B^M_N}_\infty.
\end{equation}
Note that, for all $M$ and $N$,
\begin{equation} \label{eq:BNM}
\|B_N^M\|_\infty \leq C_2 :=  \frac{1}{2\pi}\sup_{x,y\in [0,1]}\sum\limits_{j = -\infty}^{\infty} |j|\,\left|p_j(x,y)\right| \left(1+\|f'\|_\infty^2\right)^{1/4}\alpha^{\frac{x+y+j}{2}}\abs{\log\alpha},
\end{equation}
which is finite by the bounds on $p_j$ in the proof of Proposition \ref{prop:k_t_estimate general}.
Similarly, from \eqref{eq:ktall}, for $t\in \R$, $N\in \N$, $x,y\in [0,1]$,
\begin{eqnarray} \label{eq:C2b2}
\abs{\frac{\partial}{\partial t} \tilde K_{t,N}(x,y)} \leq C_2 \quad \mbox{so that} \quad
\|\tilde K_{t,N}-\tilde K_{s,N}\|_\infty \leq  |s-t| C_2,
\end{eqnarray}
for $s,t\in \R$ and $N\in \N$, by \eqref{eq:NormEst}.
Further, by \eqref{eq:KNbound}, \eqref{ktb2} with $c=0$, and \eqref{eq:twosided} with\footnote{Computations indicate that $\mathcal{B}_{10}^*(\alpha)$ exceeds the left hand side of \eqref{eq:twosided} by not more than $1.4\%$ for $\alpha\in (0,1)$.} $n=10$,
\begin{equation} \label{eq:NormBound}
\|\tilde K_{t,N}\|_\infty
\leq C_3 := \frac{\left(1+\fF_0^2\right)^{1/4}}{\pi}\Bigg[ \fG_0\, \frac{1+\alpha^{1/2}+\alpha^{-1/2}}{4\alpha^2} +   2|\log\alpha|\,\fF_0 \, \mathcal{B}_{10}^*(\alpha)\Bigg],
\end{equation}
for $t\in \R$ and  $N\in \N$. Provided $c$ is such that the conditions of Lemma 4.7 iii) and \eqref{c:bound2} hold, we have also, by Propositions \ref{prop:k_t_estimate general} and \ref{prop:norm_estimate} and \eqref{eq:twosided}, that,
\begin{equation} \label{eq:boundcombined}
\norm{(\tilde{K}_t - \tilde{K}_{t,N})\tilde{K}_t}_{\infty}
\leq \frac{C_4}{e^{2\pi Nc}-1}, \qquad t\in \R, \;\; N\in \N,
\end{equation}
where $C_4 := 2e^{2\pi c}C_5C_6$, and $C_5$ and $C_6$ denote the right hand sides of \eqref{ktb1} and \eqref{ktb2}, respectively, with the sum replaced by its upper bound $\mathcal{B}^*_{10}(\alpha)$.

As noted above Proposition \ref{prop:norm_estimate},  $\{\tilde K_{t,N}:N\in \N\}$ is collectively compact for each $t\in \R$. We will need, in the proof of Theorem \ref{thm:Hausdorff_convergence}, the following stronger statement.
\begin{lem} \label{lem:ColCom}
The sets $\{\tilde K_{t,N}:t\in \R,\,N\in \N\}$ and $\{\tilde K^M_{t,N}:t\in \R,\,M,N\in \N\}$ are collectively compact.
\end{lem}
\begin{proof}
Where $c$ is such that the conditions of Lemma 4.7 iii) and \eqref{c:bound2} hold, Proposition \ref{prop:k_t_estimate general} implies that,  for $t\in \R$, $N\in \N$, and $\phi\in C[0,1]$ with $\|\phi\|_\infty \leq 1$, $\tilde K_{t,N}\phi$ has an analytic continuation from $[0,1]$ to $\Sigma_c$. Where $\mathcal{F}$ denotes this set of analytic continuations, it follows from \eqref{ktb1} that $\mathcal{F}$ is uniformly bounded on $\Sigma_c$, so that (e.g., \cite[Thm.~14.6]{Rud2}) $\mathcal{F}$ is normal, so that $\{\tilde K_{t,N}\phi:t\in \R, N\in \N, \|\phi\|_\infty \leq 1\}$ is relatively compact, i.e., $\{\tilde K_{t,N}:t\in \R,\,N\in \N\}$ is collectively compact. The same argument applies to the family $\tilde K^M_{t,N}$, on noting, by inspection of the proof of Proposition \ref{prop:k_t_estimate general}, that, under the same conditions on $c$, $\tilde K_t^M(\cdot,y)$ extends to an analytic function on $\Sigma_c$ bounded by the right hand side of \eqref{ktb1}, for all $t,y\in \R$, $M\in \N$.
\end{proof}

In the following result, which holds for every $\Gamma$ that satisfies our standing assumption \eqref{eq:standing},  we propose an approximation for the spectrum of $D_\Gamma$ as the union of the spectra of finitely many finite matrices, and show that this approximation converges in the Hausdorff metric. In this theorem $\sigma(\tilde{K}^M_{t,N})$ denotes the spectrum of $\tilde K^M_{t,N}$ either on $L^2(0,1)$ or on $C[0,1]$ (cf.~\eqref{eq:SpecSame}), which coincides with $\{0\}\cup \sigma(A^M_{t,N})$ by Lemma \ref{lem:reduction_to_matrices}(ii), and $\sigma(D_\Gamma)$ denotes the spectrum of $D_\Gamma$ on $L^2(\Gamma)$, which coincides with the essential spectrum by Corollary \ref{cor:ess_ne}.
\begin{thm} \label{thm:Hausdorff_convergence}
Choose sequences $(m_N)_{N\in \N}, (M_N)_{N\in \N}\subset \N$ such that $m_N, M_N\to\infty$ as $N\to\infty$, and for each $N$, let $T_N:=\{\pm (k-1/2)\pi/m_N:k=1,\ldots,m_N\}$. Then
\begin{equation} \label{eq:sigmaNDef}
\sigma^N(D_\Gamma):= \bigcup\limits_{t \in T_N} \sigma(\tilde{K}^{M_N}_{t,N})  = \{0\} \cup \bigcup\limits_{t \in T_N} \sigma(A^{M_N}_{t,N}) \;\toH \; \sigma(D_{\Gamma})=\sigma_{\ess}(D_\Gamma)
\end{equation}
as $N\to\infty$.
\end{thm}
\begin{proof}

Set $\Lambda := \spec(D_{\Gamma})$, so that $\Lambda = \bigcup_{t \in [-\pi,\pi]} \spec(\tilde{K}_t)$ by Theorem \ref{cor:norm_and_spec}  and \eqref{eq:SpecSame}, and set $\Lambda_N := \sigma^N(D_\Gamma)$. We first observe that $(\Lambda_N)_{N \in \N}$ is uniformly bounded as $\{\|\tilde{K}^{M_N}_{t,N}\|_{\infty}:N\in \N, \,t\in T_N\}$ is bounded, by \eqref{eq:eMbound} and \eqref{eq:NormBound}. Thus, as  noted above Theorem \ref{thm:convergence_estimate}, to prove that $\Lambda_N\toH \Lambda$ it is enough to show that $\lim \inf \Lambda_N = \lim \sup \Lambda_N = \Lambda$.

Next, we note that if the sequences $(\mathfrak{M}_k)_{k\in \N}\subset \N$, $(N_k)_{k\in \N}\subset \N$, and $(t_k)_{k\in \N}\subset [-\pi,\pi]$ satisfy $N_k\to\infty$, $\mathfrak{M}_k\to\infty$, $t_k\to t\in [-\pi,\pi]$, as $k\to \infty$, then the sequence $(\tilde{K}^{\mathfrak{M}_k}_{t_k,N_k})_{k \in \N}$ is collectively compact by Lemma \ref{lem:ColCom}, and converges strongly to $\tilde{K}_t$ by \eqref{eq:C2b2} and \eqref{eq:eMbound}, and since $\tilde K_{t,N_k}\to \tilde K_t$, as noted above Proposition \ref{prop:norm_estimate}. Thus also $\sigma(\tilde{K}^{\mathfrak{M}_k}_{t_k,N_k})\toH\sigma(\tilde K_t)$, by Theorem \ref{thm:convergence_estimate}.

For every $t\in [-\pi,\pi]$ we can choose a sequence $(t_N)_{N\in \N}$ such that $t_N\in T_N$ for each $N$ and $t_N\to t$. Then, by the observation just made, $\Lambda_N\supset \sigma(\tilde{K}^{M_N}_{t_N,N})\toH \sigma(\tilde K_t)$ as $N\to\infty$, so that $\lim\inf \Lambda_N \supset \lim \inf \sigma(\tilde{K}^{M_N}_{t_N,N}) = \sigma(\tilde K_t)$. Thus $\lim \inf \Lambda_N \supset \Lambda$.

If $\lambda \in \lim\sup \Lambda_N$ then there exists a sequence $(N_k)_{k\in \N}\subset \N$ and a sequence $(\lambda_k)_{k\in \N}$ such that $\lambda_k\to\lambda$ and, for each $k$, $\lambda_k\in \Lambda_{N_k}$, so that $\lambda_k\in \sigma(\tilde{K}^{\mathfrak{M}_k}_{t_k,N_k})$,  for some $t_k\in T_{N_k}$, where $\mathfrak{M}_k:= M_{N_k}$. By passing to a subsequence if necessary, we may assume that $t_k\to t\in [-\pi,\pi]$, so that $\sigma(\tilde{K}^{\mathfrak{M}_k}_{t_k,N_k})\toH \sigma(\tilde K_t)$, in particular $\lim \inf \sigma(\tilde{K}^{\mathfrak{M}_k}_{t_k,N_k}) = \sigma(\tilde K_t)$, so that $\lambda \in \sigma(\tilde K_t)\subset \Lambda$. Thus $\lim \sup \Lambda_N\subset \Lambda \subset \lim \inf \Lambda_N$. But also $\lim \inf \Lambda_N \subset \lim \sup \Lambda_N$, so the proof is complete.
\end{proof}

\begin{rem}[\em \bf Reduced computation expression for $\sigma^N(D_\Gamma)$] \label{rem:symm2} Note that, by Remark \ref{rem:sym} and Lemma \ref{lem:reduction_to_matrices}(ii),  $\sigma^N(D_\Gamma)$, given by \eqref{eq:sigmaNDef}, can be written as
$$
\sigma^N(D_\Gamma) = \{0\} \cup  \{\lambda, \overline{\lambda}: \lambda \in \sigma(A^{M_N}_{t,N}),\, t\in T_N, \, t>0\}.
$$
\end{rem}

Our second main result, obtained by applying Theorem \ref{thm:GspecRMfinal}\footnote{To obtain \eqref{eq:estimate_to_check}, motivated by Remark \ref{rem:comparison} our starting point is \eqref{eq:keyboundfinalAlt} rather than \eqref{eq:keyboundfinal}, since our interest will be to apply Theorem \ref{thm:spectral_radius} in cases where $\rho_0<\|D_\Gamma\|$.}, provides a criterion, given $\rho_0>0$, for $\rho(D_\Gamma;L^2(\Gamma))<\rho_0$. Note that this criterion requires computation of spectral quantities only for the $N\times N$ matrices $A_{t,N}^M$ for finitely many $t\in [0,\pi]$.

\begin{thm} \label{thm:spectral_radius}
Under our standing assumption \eqref{eq:standing}, define $g$ by \eqref{eq:gdef} so that, by Lemma \ref{lem:f_anal},  $g$ has an analytic continuation to $\Sigma_c$, for some $c>0$, such that $g$, $g^\prime$, and $g^{\prime\prime}$ are bounded on $\Sigma_c$, and, without loss of generality, assume that \eqref{c:bound2} holds. Suppose that $\rho_0>0$, $m,M,N\in \N$, with $M\geq 2$. For $k=1,\ldots,m$, set $t_k:= (k-1/2)\pi/m$, suppose that $\rho(A^M_{t_k,N})<\rho_0$ for $k=1,\ldots,m$, and recursively define $\mu_{k,\ell}$, for $\ell=1,2,\ldots$, by $\mu_{k,1}:= \rho_0$, and by
\begin{equation} \label{eq:lambda_{k,l}}
\nu_{k,\ell} := \norm{(A_{t_k,N}^M - \mu_{k,\ell} I)^{-1}}_{\infty}^{-1} \quad \text{and} \quad \mu_{k,\ell+1} := \mu_{k,\ell}e^{i\frac{\nu_{k,\ell}}{2\rho_0}}, \quad \mbox{for } \ell\in \N.
\end{equation}
Further, for $k=1,\ldots,m$, let $n_k$ denote the smallest integer such that $\sum_{\ell=1}^{n_k} \nu_{k,\ell} \geq 4\pi \rho_0$, and let
$$
R_{m,M,N} := \min_{\stackrel{k=1,\ldots,m}{\ell=1,\ldots,n_k}}\left(\frac{\nu_{k,\ell}}{4}+\frac{\nu_{k,\ell+1}}{2}\right).
$$
If
\begin{equation} \label{eq:estimate_to_check}
\mathcal{L}_c(f,N) := \frac{C_4}{e^{2\pi Nc}-1} < \mathcal{R}_c(f,m,M,N) := r^2_0\left(1 + C_3\left(R_{m,M,N} - C_1(M)- \frac{\pi}{2m}\|B_N^M\|_\infty\right)^{-1} \right)^{-1},
\end{equation}
then $\rho(D_\Gamma;L^2(\Gamma)) < \rho_0$. Conversely, if $\rho(D_\Gamma;L^2(\Gamma)) < \rho_0$, then, provided $m$ and $M$ are sufficiently large, there exists $N_0\in \N$ such that \eqref{eq:estimate_to_check} holds and $\rho(A_{t,N}^M)<\rho_0$ for $t\in [0,\pi]$ and all $N\geq N_0$.
\end{thm}
\begin{proof} By Theorem \ref{cor:norm_and_spec}, Remark \ref{rem:sym}, and \eqref{eq:SpecSame}, to show that $\rho(D_\Gamma;L^2(\Gamma)) < \rho_0$, it is enough to check that $\rho_{C[0,1]}(\tilde K_t)<\rho_0$ for every $t\in [0,\pi]$. So pick $t\in [0,\pi]$. Then $|t-t_k| \leq \pi/(2m)$, for some $k\in \{1,\ldots,m\}$. To conclude that $\rho_{C[0,1]}(\tilde K_t)<\rho_0$ we apply Theorem \ref{thm:GspecRMfinal} with $K=\tilde K_t$, $K_N=\tilde K_{t,N}$, $A_N=A_{t,N}$, and $A_N^{\dag}=A_{t_k,N}^M$, noting that, with these definitions, $\|K_N\|_\infty = \|\tilde K_{t,N}\|_\infty \leq C_3$,
$$
\|A_N-A_N^{\dag}\|_\infty \leq \|A_{t,N}-A_{t,N}^M\|_\infty + \| A^M_{t,N}-A_{t_k,N}^M\|_\infty \leq C_1(M) + \frac{\pi}{2m}\|B_N^M\|_\infty,
$$
by \eqref{eq:eMbound} and \eqref{eq:C2def}, and $\|(K-K_N)K\|_\infty \leq C_4/(e^{2\pi Nc}-1)$, by \eqref{eq:boundcombined}.

Conversely, if $\rho(D_\Gamma;L^2(\Gamma)) < \rho_0$ then, by Theorem \ref{cor:norm_and_spec} and \eqref{eq:SpecSame}, $\rho_{C[0,1]}(\tilde K_t)<\rho_0$ for $t\in [0,\pi]$. Arguing as in the proof of Theorem \ref{thm:GspecRMfinal}, this implies that for every $t\in [0,\pi]$ there exists $N_0(t)$ such that $\|(\tilde K_{t,N}-\lambda I)^{-1}\|$ is bounded uniformly in $\lambda$ and $N$ for $N\geq N_0(t)$ and $|\lambda|\geq \rho_0$. This, combined with the estimate \eqref{eq:C2b2} and a standard compactness argument, implies that, for some $N^*\in \N$, and $c^*>0$, $\|(\tilde K_{t,N}-\lambda I)^{-1}\|^{-1}\geq c^*$ for all $N\geq N^*$, $|\lambda|\geq \rho_0$, and $t\in [0,\pi]$. It follows, by \eqref{eq:eMbound} and \eqref{eq:PerBas}, and recalling Lemma \ref{lem:reduction_to_matrices}(ii), that, for all $t\in [0,\pi]$ and all sufficiently large $M$ and $N$, $\rho(A_{t,N}^M)=\rho_{C[0,1]}(\tilde K^M_{t,N})< \rho_0$ and  $\|(\tilde K^M_{t,N}-\lambda I)^{-1}\|^{-1}\geq c^*/2$, for $|\lambda|\geq \rho_0$. Thus, and by Lemma \ref{lem:reduction_to_matrices}(iii), $R_{m,N,M}\geq 3c^*/8$, for all $m\in \N$ and all sufficiently large $M$ and $N$. It follows that \eqref{eq:estimate_to_check} holds, for all sufficiently large $m$, $M$, and $N$, since \eqref{eq:C1def} and \eqref{eq:BnAsy}, and that $\|B_N^M\|_\infty$ is bounded independently of $M$ and $N$ by \eqref{eq:BNM}, imply that the right hand side of \eqref{eq:estimate_to_check} is positive and bounded away from zero.
\end{proof}

\subsection{The 2D case: two-sided infinite graphs} \label{sec:two-sided}

\begin{figure}[ht]
\centering
\includegraphics[scale=0.8, trim = 0.5cm 2.5cm 0cm 2.6cm, clip]{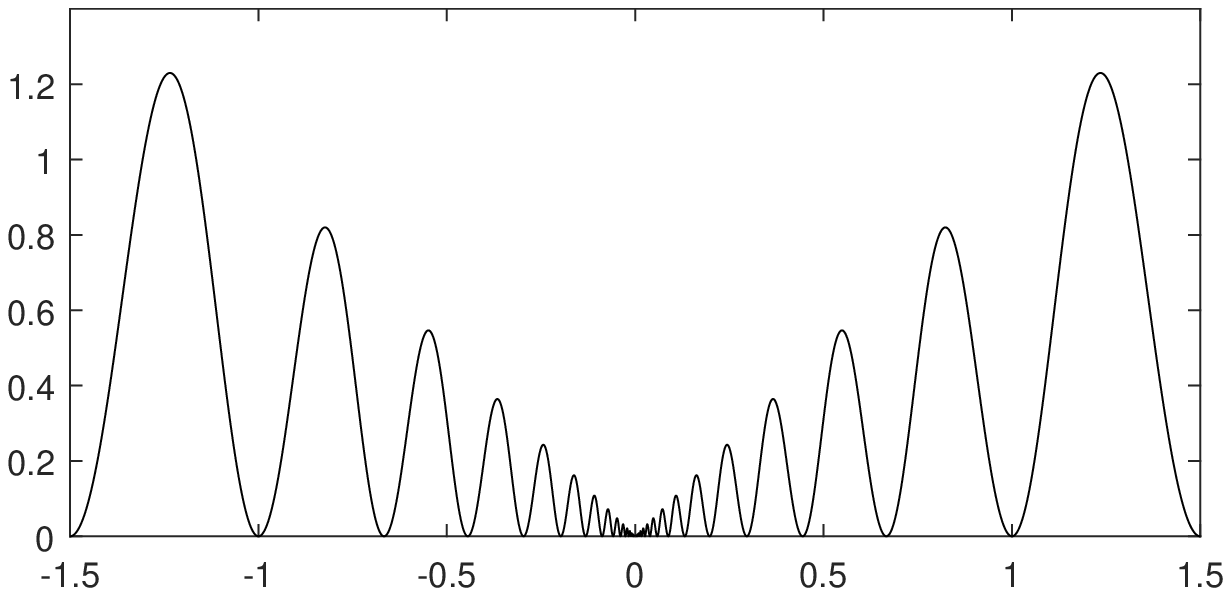}
\caption{Graph of $f \from \R \to \R$, $f(x) := \abs{x}\sin^2(\pi\log_{\alpha}\abs{x})$, in the case $\alpha = \frac{2}{3}$} \label{fig:ex2}
\end{figure}

We now extend the results of the previous subsection, for the case when $d=2$ and $A=\R_+$, to the more involved case\footnote{As discussed in the introduction and at the beginning of \S\ref{sec:dilation_invariant}, it is this case which is relevant to the spectral radius conjecture.} where $d=2$ and $A=\R$. Our goals and methods are those of \S\ref{sec:oneside}, but, where $\dot \R:= \R\setminus \{0\}$, we assume now that $f\in \cA(\dot \R)$, the space of functions $\R\to \R$ that are real analytic on $\dot \R$ with $f(0)=0$ (a prototypical example is Figure \ref{fig:ex2}). Thus our standing assumption through this subsection is that
\begin{equation} \label{eq:standing2}
\Gamma = \set{(x,f(x)): x \in \R} \;\mbox{ where }\; f\in \cA(\dot \R) \; \mbox{ and, for some $\alpha \in (0,1)$, } \; f(\alpha x) = \alpha f(x), \quad x\in \R.
\end{equation}

Define $f_\pm:\R_+\to \R$ by $f_\pm(x):= f(\pm x)$, $x\in \R_+$, and $g_\pm:\R\to \R$ (cf.~\eqref{eq:gdef}) by
\begin{equation} \label{eq:gdef2}
g_\pm(x) := \alpha^{-x}f_\pm(\alpha^x), \quad x\in \R, \quad \mbox{so that} \quad f_\pm (x) = xg_\pm(\log_\alpha x), \quad x\in \R_+.
\end{equation}
Note that the assumption $f\in \cA(\dot \R)$ is equivalent to the assumption that $f_\pm\in \cA(\R_+)$ and $f(0)=0$. Note also  that Lemma \ref{lem:f_anal} applies with $f$ and $g$ replaced by $f_\pm$ and $g_\pm$. One simple consequence of these observations (see the discussion below \eqref{eq:fpr}) is that \eqref{eq:standing2} implies $f\in C^{0,1}(\R)$.

As in the one-sided case, our starting point is the formula for $\sigma(D_\Gamma)$ in Corollary \ref{cor:norm_and_spec} in terms of $\sigma(K_t)$, for $t\in [-\pi,\pi]$. Again, $K_t$ is compact for every $t\in \R$ by Corollary \ref{cor:K_t_compact}, and, following the pattern of \S\ref{sec:oneside}, our goal is to compute $\sigma(K_t)$ and $\rho(K_t)$ by the results of \S\ref{sec:Ny}. As in \S\ref{sec:oneside}, to make use of these results it is convenient to work with integral operators on $[0,1]$ rather than $\Gamma_0$. Reflecting that $\Gamma_0$ has two components, $\Gamma_0^+ := \{(x,f(x))\in \Gamma_0:x>0\}$ and $\Gamma_0^-:= \{(x,f(x))\in \Gamma_0:x<0\}$, it is natural to work in this case with a unitary operator $U:L^2(\Gamma_0)\to (L^2(0,1))^2$, defined by $U\psi := (U_+P_+\psi, U_-P_-\psi)^T$, $\psi\in L^2(\Gamma_0)$, where $P_\pm$ is restriction to $\Gamma_0^\pm$, i.e.~$P_\pm \psi := \psi|_{\Gamma_0^\pm}$, $\psi\in L^2(\Gamma_0)$, and $U_\pm:L^2(\Gamma_0^\pm)\to L^2(0,1)$ is the unitary operator defined by (cf.~\eqref{eq:Udef})
$$
U_\pm \phi(x) := \phi(\pm \alpha^{x}, f_\pm(\alpha^x)) (1+f_\pm'(\alpha^x)^2)^{1/4} \alpha^{x/2}|\log \alpha|^{1/2}, \qquad x\in [0,1], \;\; \phi\in L^2(\Gamma_0^\pm).
$$
With $U$ as given above we define an operator $\tilde K_t$ on $(L^2(0,1))^2$, which is unitarily equivalent to $K_t$, by \eqref{eq:tilK}. It is easy to see (cf.~\eqref{eq:kernel_computation}) that
\begin{equation}\label{e:tilde K_t}
	\tilde K_t =
	\begin{pmatrix}
	\tilde K^-_t & \tilde L^-_t\\
	\tilde L^+_t & \tilde K^+_t
	\end{pmatrix}, \qquad t\in \R,
\end{equation}
where the entries of $\tilde K_t$ are integral operators on $L^2(0,1)$ with continuous kernels. The kernels of $\tilde K_t^\pm$ are $\tilde K^\pm_t(\cdot,\cdot)$, where (cf.~\eqref{eq:ktall})
\begin{equation} \label{eq:ktall2}
\tilde{K}^\pm_t(x,y) := \frac{1}{2\pi}\sum\limits_{j = -\infty}^{\infty} e^{ijt}\,\frac{p^\pm_j(x,y)}{1+q^\pm_j(x,y)^2}\left(\frac{1+f_\pm'(\alpha^{x})^2}{1+f_\pm'(\alpha^y)^2}\right)^{1/4}\alpha^{\frac{x+y+j}{2}}\abs{\log\alpha}, \quad x,y\in \R,
\end{equation}
and $p_j^\pm$ and $q_j^\pm$ are defined by \eqref{eq:pjdef}, \eqref{Tay1}, and \eqref{Tay2}, but with $f$ replaced by $f_\pm$. The kernels of $\tilde L_t^\pm$ are $\tilde L_t^\pm(\cdot,\cdot)$, where, for $x,y\in \R$,
\begin{equation} \label{eq:L+}
\tilde{L}^\pm_t(x,y) := \frac{1}{2\pi}\sum\limits_{j = -\infty}^{\infty} e^{ijt}\,\frac{\tilde p^\pm_j(x,y)}{1+\tilde q^\pm_j(x,y)^2}\left(\frac{1+f_\pm'(\alpha^{x})^2}{1+f_\mp'(\alpha^y)^2}\right)^{1/4}\alpha^{\frac{x+y+j}{2}}\abs{\log\alpha}
\end{equation}
and
\begin{equation} \label{eq:pjdef2}
\tilde p^\pm_j(x,y) := \frac{(\alpha^y+\alpha^{x+j})f_\mp^\prime(\alpha^y) + f_\pm(\alpha^{x+j}) - f_\mp(\alpha^y)}{(\alpha^{x+j}+\alpha^y)^2}, \quad \tilde q^\pm_j(x,y) := \frac{f_\pm(\alpha^{x+j})-f_\mp(\alpha^y)}{\alpha^{x+j}+\alpha^y}.
\end{equation}

Analogously to \eqref{eq:SpecSame}, we have that
\begin{equation}\label{eq:SpecSame2}
\sigma(K_t;L^2(\Gamma_0)) = \sigma(\tilde K_t;(L^2(0,1))^2)=\sigma(\tilde K_t;(C[0,1])^2), \quad t\in \R.
\end{equation}
As in the one-sided case (see the discussion around \eqref{eq:KtnDef}), to estimate spectral properties of $\tilde K_t$ as an operator on $(C[0,1])^2$ we approximate $\tilde K^\pm_t$ by finite rank operators $\tilde K^\pm_{t,N}$ given by
\[\tilde{K}^\pm_{t,N}\phi(x) := \frac{1}{N}\sum_{n=1}^N \tilde{K}^\pm_t(x,x_{n,N}) \phi(x_{n,N}), \quad \phi \in C[0,1], \; x \in [0,1], \; N\in \N.\]
Similarly, we approximate $\tilde{L}^\pm_{t}$ by $\tilde{L}^\pm_{t,N}$, leading to finite rank approximations $\tilde K_{t,N}$, $N\in \N$, to $\tilde K_t$, given by \eqref{e:tilde K_t} with $\tilde K^\pm_t$ and $\tilde L^\pm_t$ replaced by $\tilde K^\pm_{t,N}$ and $\tilde L^\pm_{t,N}$. Arguing as below \eqref{eq:KtnDef}, we have that $\tilde K_{t,N}\to \tilde K_t$ and $(\tilde K_{t,N})_{N\in \N}$ is collectively compact, so that $\|(\tilde K_t-\tilde K_{t,N})\tilde K_t\|_\infty \to 0$ as $N\to\infty$, where $\|\cdot\|_\infty$ here denotes the operator norm of an operator on $(C[0,1])^2$ equipped with the norm $\|\cdot\|_\infty$ defined by $\|\phi\|_\infty:=\max(\|\phi_+\|_\infty,\|\phi_-\|_\infty)$, for $\phi=(\phi_+,\phi_-)^T\in (C[0,1])^2$
Indeed, we have the following analogue of Proposition \ref{prop:norm_estimate}, in which our other norm notations are as defined above Proposition \ref{prop:k_t_estimate general}.
\begin{prop} \label{prop:norm_estimate2}
Let $t\in \R$ and $c>0$ and assume that, for every $x,y\in \R$, $\tilde K^{\pm}_t(x,\cdot)$, $\tilde L^{\pm}_t(x,\cdot)$, $\tilde K^{\pm}_t(\cdot,y)$ and $\tilde L^{\pm}_t(\cdot,y)$ have analytic continuations from $\R$ to $\Sigma_c$ that are bounded in $\Sigma_c$. Then $\|(\tilde{K}_t - \tilde{K}_{t,N})\tilde{K}_t\|_{\infty} \leq 2e^{2\pi c}C^*/(e^{2\pi N c}-1)$, where
\begin{align} \nonumber
C^* & :=	\max\bigg\{ \|\tilde K^-_t\|_{0,c} \|\tilde K^-_t\|_{c,0} + \|\tilde L^-_t\|_{0,c} \|\tilde L^+_t\|_{c,0} + \|\tilde K^-_t\|_{0,c} \|\tilde L^-_t\|_{c,0} + \|\tilde L^-_t\|_{0,c} \|\tilde K^+_t\|_{c,0},\\ \label{eq:C*def}
	&\qquad\qquad\quad
	\|\tilde L^+_t\|_{0,c} \|\tilde K^-_t\|_{c,0} + \|\tilde K^+_t\|_{0,c} \|\tilde L^+_t\|_{c,0} + \|\tilde L^+_t\|_{0,c} \|\tilde L^-_t\|_{c,0} + \|\tilde K^+_t\|_{0,c} \|\tilde K^+_t\|_{c,0} \bigg\}.
\end{align}
\end{prop}

\begin{proof}
Notice first that
$$
	(\tilde K_t-K^N_t)\tilde K_t
	=\begin{pmatrix}
	(\tilde K^-_t - K^-_{t,N})\tilde K^-_t+(\tilde L^-_t-L^-_{t,N})\tilde L^+_t &
	(\tilde K^-_t - K^-_{t,N})\tilde L^-_t+(\tilde L^-_t-L^-_{t,N})\tilde K^+_t\\
	(\tilde L^+_t-L^+_{t,N})\tilde K^-_t + (\tilde K^+_t-K^+_{t,N})\tilde L^+_t &
	(\tilde L^+_t-L^+_{t,N})\tilde L^-_t + (\tilde K^+_t-K^+_{t,N})\tilde K^+_t
	\end{pmatrix},
$$
and denote the entries of this matrix by $A_{j,k}$, $j,k=1,2$. For each of these terms we obtain an estimate similar to that in Proposition \ref{prop:norm_estimate}, by arguing as in the proof of that proposition. Since also
\[
	\|(\tilde K_t - K^N_t)\tilde K_t\|_{\infty} = \max_{j=1,2} \sum_{k=1}^2 \|A_{j,k}\|_{\infty},
\]
the result follows.
\end{proof}

Let us now estimate, under our standing assumption \eqref{eq:standing2}, the norms of the kernels that appear in Proposition \ref{prop:norm_estimate2} (cf.~Proposition \ref{prop:k_t_estimate general}).

\begin{prop} \label{prop:k_t_estimate general2}
Given \eqref{eq:standing2}, define $g_\pm$ by \eqref{eq:gdef2} so that, by Lemma \ref{lem:f_anal}, $g_\pm$ have analytic continuations to $\Sigma_c$, for some $c>0$, such that $g_\pm$, $g_\pm'$, and $g_\pm^{''}$ are bounded on $\Sigma_c$ and  $g_\pm(z+1)=g_\pm(z)$, $z\in \Sigma_c$.
Let
\begin{align} \label{fJdef}
\fI_{c,\pm} := \|\Imag g_\pm\|_c + \|\Imag g_\pm^\prime\|_c/|\log \alpha| \quad \mbox{and} \quad \fK_{c,\pm} := \alpha^{-1}|\log\alpha|\max(\|g_\pm\|_c,\|g_\mp\|_0),
\end{align}
 and set
\begin{align} \label{fFdef2}
\fF_{d,\pm} := \|g_\pm\|_d + \|g_\pm^\prime\|_d/|\log \alpha| \quad \mbox{and} \quad \fG_{d,\pm} := \|g_\pm^\prime\|_d + \|g_\pm^{\prime\prime}\|_d/|\log \alpha|,   \quad \mbox{for } d=0,c.
\end{align}
If
\begin{equation} \label{c:bound2_2}
c \leq \arccos(\alpha)/\abs{\log\alpha}, \quad \fI_{c,\pm} < 1,\quad \mbox{and} \quad \|\Imag g_\pm\|_c + \alpha c\fK_{c,\pm}  < \alpha^2,
\end{equation}
then $\tilde K^{\pm}_t(x,\cdot)$, $\tilde L^{\pm}_t(x,\cdot)$, $\tilde K^{\pm}_t(\cdot,y)$, and $\tilde L^{\pm}_t(\cdot,y)$ extend to bounded analytic functions on $\Sigma_c$ for all $x,y,t\in \R$, and
\begin{align}\nonumber
\|\tilde{K}^\pm_t\|_{c,0} &\leq \frac{\left(1+\fF_{c,\pm}^2\right)^{1/4}}{\pi\left(1- \fI_{c,\pm}^2\right)}\Bigg[ \fG_{c,\pm}\, \frac{1+\alpha^{1/2}+\alpha^{-1/2}}{4\alpha^2} +   |\log\alpha|(\fF_{c,\pm}+\fF_{0,\pm}) \, \sum_{j=2}^\infty\frac{\alpha^{j/2}}{\alpha-\alpha^j}\Bigg],\\ \nonumber
\|\tilde{K}^\pm_t\|_{0,c} &\leq \frac{\left(1+\fF_{0,\pm}^2\right)^{1/4}}{\pi\left(1- \fI_{c,\pm}^2\right)^{5/4}}\Bigg[ \fG_{c,\pm}\, \frac{1+\alpha^{1/2}+\alpha^{-1/2}}{4\alpha^2} +   2|\log\alpha|\,\fF_{c,\pm} \, \sum_{j=2}^\infty\frac{\alpha^{j/2}}{\alpha-\alpha^j}\Bigg],\\
	\label{ltb1}
	\|\tilde L^{\pm}_t\|_{c,0} &\leq \frac{1}{2\pi}\frac{\abs{\log\alpha}\fF_{0,\mp} + \fK_{c,\pm}}{1 - \alpha^{-4}\left(\norm{\Imag g_\pm}_c + \alpha c\fK_{c,\pm}\right)^2} \,\left(1 + \fF_{c,\pm}^2\right)^{1/4} \sum\limits_{j = -\infty}^{\infty} \frac{\alpha^{j/2}}{\alpha^{j+2}+\alpha^2},\\
	\label{ltb2}
	\|\tilde L^{\pm}_t\|_{0,c} &\leq \frac{1}{2\pi} \frac{\abs{\log\alpha}\fF_{c,\mp} + \fK_{c,\mp}}{1 - \alpha^{-4}\left(\norm{\Imag g_\mp}_c + \alpha c\fK_{c,\mp}\right)^2}\,\left(\frac{1 + \fF_{0,\pm}^2}{1 - \fI_{c,\mp}^2}\right)^{1/4}\sum\limits_{j = -\infty}^{\infty} \frac{\alpha^{j/2}}{\alpha^{j+2}+\alpha^2}.
\end{align}
Moreover, the mappings $\R \to C([0,1]\times [0,1])$, $t \mapsto \tilde{K}^{\pm}_t(\cdot,\cdot)$ and $t \mapsto \tilde{L}^{\pm}_t(\cdot,\cdot)$, are continuous.
\end{prop}

\begin{rem}[\it \bf Bound on the sum in \eqref{ltb1} and \eqref{ltb2}] \label{rem:series2} For fixed $\alpha \in (0,1)$, let $F(x) := (\alpha^{x/2} + \alpha^{-x/2})^{-1}$, for $x\in \R$, so that the $j$th term in the sum in \eqref{ltb1} and \eqref{ltb2} is $\alpha^{-2}F(j)$, and note that $F$ is even. Arguing as in Remark \ref{rem:series}, since $F$ is decreasing on $[0,\infty)$,
\begin{equation} \label{eq:Cn}
\sum_{j=n+1}^\infty \frac{\alpha^{j/2}}{\alpha^{j+2}+\alpha^2} = \alpha^{-2}\sum_{j=n+1}^\infty F(j) \leq \cC_n(\alpha) := \alpha^{-2}\int_n^\infty F(x)\, \rd x = \frac{2\arctan(\alpha^{n/2})}{\alpha^2\abs{\log\alpha}},
\end{equation}
for $0<\alpha < 1$ and $n\in \N_0:=\N\cup \{0\}$. Thus, for $0<\alpha < 1$ and $n\in \N_0$,
\begin{align} \label{eq:sum2bound}
\sum\limits_{j = -\infty}^{\infty} \frac{\alpha^{j/2}}{\alpha^{j+2}+\alpha^2} &\leq  \cC_n^*(\alpha) := \frac{1}{2\alpha^2} + 2\alpha^{-2}\sum_{j=1}^n F(j) + 2\cC_n(\alpha) \quad \mbox{and} \\ \label{eq:rembound}
\abs{\sum\limits_{j = -\infty}^{\infty} \frac{\alpha^{j/2}}{\alpha^{j+2}+\alpha^2} - \sum\limits_{j = -n}^n \frac{\alpha^{j/2}}{\alpha^{j+2}+\alpha^2}} &\leq 2 \cC_n(\alpha).
\end{align}
\end{rem}

\begin{proof}[Proof of Proposition \ref{prop:k_t_estimate general2}]
The results for $\tilde{K}_t^\pm(\cdot,\cdot)$ follow immediately from Proposition \ref{prop:k_t_estimate general}, applied with $f$ replaced by $f_\pm$.
It remains to consider the off-diagonal entries $\tilde L^{\pm}_t$. We give the detail for $\tilde{L}_t^+$; the results for $\tilde{L}_t^-$ follow by the same argument with the roles of $f_+$ and $f_-$ reversed.

We extend the definition of $\tilde L^+_t$ via~\eqref{eq:L+} to all $x,y\in \Sigma_c$. The upcoming computations will show that this is well-defined and the estimates \eqref{ltb1} and \eqref{ltb2} hold. Denote by $T_j(x,y)$ the $j$th term in the sum~\eqref{eq:L+}, that is,
\begin{equation} \label{eq:T_j}
T_j(x,y) := e^{ijt} \frac{\tilde p^+_j(x,y)}{1+\tilde q^+_j(x,y)^2}\left(\frac{1+f_+'(\alpha^{x})^2}{1+f_-'(\alpha^y)^2}\right)^{1/4}\alpha^{\frac{x+y+j}{2}}\abs{\log\alpha}, \quad x,y\in \Sigma_c.
\end{equation}
We clearly have $T_j(x+1,y)=e^{-it}T_{j+1}(x,y)$ and the dilation invariance, $f(\alpha x) = \alpha f(x)$, implies that $T_j(x,y+1)=e^{it}T_{j-1}(x,y)$ for $x,y\in \Sigma_c$, $j \in \Z$. Thus, to prove the uniform convergence of the series \eqref{eq:L+} and the bounds \eqref{ltb1} and \eqref{ltb2}, it suffices to restrict consideration to $x,y \in \Xi_c = \{z \in \C :$ $\Real z \in [0,1]$, $\Imag z \in (-c,c)\}$.
First notice that, for $x,y \in \Xi_c$,
\begin{equation} \label{eq:denominator_estimate}
\abs{\alpha^{x+j}+\alpha^y} \geq \alpha^{\Real x + j}\cos(\log\alpha \Imag x) + \alpha^{\Real y}\cos(\log\alpha \Imag y) \geq \alpha^{\Real x + j+1} + \alpha^{\Real y + 1}\geq \alpha^{j+2} + \alpha^2,
\end{equation}
by the first of \eqref{c:bound2_2}, which implies that $\cos(t\log \alpha )\geq \alpha$, for $-c\leq t\leq c$. Hence,  for $x,y\in \Xi_c$ and $j\in \Z$, recalling \eqref{eq:gdef2},
\begin{align*}
\abs{\tilde q^+_j(x,y)} =
\abs{\frac{\alpha^{x+j}g_+(x) - \alpha^yg_-(y)}{\alpha^{x+j}+\alpha^y}} \leq \frac{\alpha^{\Real x+j}|g_+(x)| + \alpha^{\Real y}|g_-(y)|}{\alpha^{\Real x + j+1}+\alpha^{\Real y + 1}}.
\end{align*}
Using the estimate
\begin{equation} \label{eq:Moebius_estimate}
(a+b)/(c+d) \leq \max\set{a/c,b/d},
\end{equation}
which holds for all $a,b\geq 0$, $c,d>0$, it follows that
\[\sup\limits_{x \in \Xi_c,y \in [0,1]} \abs{\tilde q^+_j(x,y)} \leq \fK_{c,+}/|\log \alpha|\quad \mbox{and} \quad \sup\limits_{x \in [0,1],y \in \Xi_c} \abs{\tilde q^+_j(x,y)}
\leq \fK_{c,-}/|\log\alpha|, \quad j\in \Z.\]
Noting that \eqref{eq:falpha} holds with $f$ and $g$ replaced with $f_\pm$ and $g_\pm$,
we have also that $|f_\pm^\prime(\alpha^y)|\leq \fF_{0,\pm}$, for $y\in [0,1]$, while $|f_\pm^\prime(\alpha^y)|\leq \fF_{c,\pm}$, for $y\in \Xi_c$.
Thus
\begin{align*}
\sup\limits_{x \in \Xi_c,y \in [0,1]} \abs{\tilde p^+_j(x,y)} &\leq \sup\limits_{x \in \Xi_c,y \in [0,1]} \frac{|f_-'(\alpha^y)| + \abs{\tilde q^+_j(x,y)}}{\abs{\alpha^{x+j}+\alpha^y}} \leq \frac{\fF_{0,-} + \fK_{c,+}/|\log\alpha|}{\alpha^{j+2}+\alpha^2} \quad \mbox{and}\\
\sup\limits_{x \in [0,1],y \in \Xi_c} \abs{\tilde p^+_j(x,y)} &\leq \frac{\fF_{c,-} + \fK_{c,-}/|\log\alpha|}{\alpha^{j+2}+\alpha^2}.
\end{align*}
Moreover, for $x,y\in \Xi_c$,
\begin{align*}
\abs{\Imag \tilde q^+_j(x,y)} &= \abs{\Imag \left(\frac{\alpha^{x+j}g_+(x) - \alpha^y g_-(y)}{\alpha^{x+j}+\alpha^y}\right)}\\
&= \frac{1}{\abs{\alpha^{x+j}+\alpha^y}^2} \abs{\Imag\left(\alpha^{2(\Real x + j)}g_+(x) - \alpha^{\bar{x}+j+y}g_-(y) + \alpha^{x+j+\bar{y}}g_+(x) - \alpha^{2\Re y}g_-(y)\right)}.
\end{align*}
Thus, using \eqref{eq:denominator_estimate},  noting that
$|\Imag(\alpha^{it}z)| \leq |\Imag(z)| + |\sin(t\log\alpha)||z| \leq |\Imag(z)| + |t\log\alpha||z|$, for $t\in \R$ and $z\in \C$,
and using \eqref{eq:Moebius_estimate} again to obtain the last two inequalities, we see that for $x,y\in \Xi_c$, where $r = \alpha^{\Real x+j}$ and $s= \alpha^{\Real y}$,
\begin{align*}
\abs{\Imag \tilde q^+_j(x,y)}
& \leq \frac{1}{\alpha^2(r+s)^2}\Big(r^2|\Imag g_+(x)| +s^2|\Imag g_-(y)| \,+\\
&  \hspace{4ex} rs\big(|\Imag g_-(y)| + |\Imag g_+(x)| + c|\log\alpha|(|g_-(y)|+|g_+(x)|)\big)\Big)\\
& \leq \frac{r|\Imag g_+(x)| +s|\Imag g_-(y)| + c|\log\alpha|\max\{s|g_-(y)|,r|g_+(x)|\}}{\alpha^2(r+s)}\\
& \leq \alpha^{-2}\left(\max\{|\Imag g_+(x)|,|\Imag g_-(y)|\} + c|\log\alpha|\max\{|g_-(y)|,|g_+(x)|\}\right).
\end{align*}
Thus
\begin{align*}
\sup\limits_{x\in \Xi_c, y\in [0,1]}\abs{\Imag \tilde q^+_j(x,y)}&\leq \alpha^{-2}\left(\|\Imag g_+\|_c + \alpha c \fK_{c,+}\right) \quad \mbox{and}\\
\sup\limits_{x\in [0,1], y\in \Xi_c}\abs{\Imag \tilde q^+_j(x,y)}&\leq \alpha^{-2}\left(\|\Imag g_-\|_c + \alpha c \fK_{c,-}\right)
\end{align*}
so that
\begin{align*}
\inf\limits_{x \in \Xi_c,y \in [0,1]}\left| 1 + \tilde q^+_j(x,y)^2\right| &\geq 1 - \sup\limits_{x \in \Xi_c,y \in [0,1]}\abs{\Imag \tilde q^+_j(x,y)}^2\\
&\geq 1 - \alpha^{-4}\left(\|\Imag g_+\|_c + \alpha c \fK_{c,+}\right)^2 \quad \mbox{and}\\
\inf\limits_{x \in [0,1],y \in \Xi_c}\left| 1 + \tilde q^+_j(x,y)^2\right| &\geq 1 - \alpha^{-4}\left(\|\Imag g_-\|_c + \alpha c \fK_{c,-}\right)^2.
\end{align*}
Combining these estimates, and using \eqref{eq:Icbound} with $f$ replaced by $f_-$, we get
\begin{align*}
\sup\limits_{x \in \Xi_c,y \in [0,1]} \abs{T_j(x,y)} &\leq   \frac{|\log\alpha|\fF_{0,-}+\fK_{c,+}}{1 - \alpha^{-4}\left( \norm{\Imag g_+}_c + \alpha c \fK_{c,+}\right)^2} \left(1 + \fF_{c,+}^2\right)^{1/4}\frac{\alpha^{j/2}}{\alpha^{j+2}+\alpha^2} \quad \mbox{and}\\
\sup\limits_{x \in [0,1],y \in \Xi_c} \abs{T_j(x,y)} &\leq   \frac{|\log\alpha|\fF_{c,-}+\fK_{c,-}}{1 - \alpha^{-4}\left( \norm{\Imag g_-}_c + \alpha c \fK_{c,-}\right)^2} \left(\frac{1 + \fF_{0,+}^2}{1-\fI_{c,-}^2}\right)^{1/4}\frac{\alpha^{j/2}}{\alpha^{j+2}+\alpha^2}.
\end{align*}
The estimates \eqref{ltb1} and \eqref{ltb2} and the other results for $\tilde{L}_t^+$ follow.
\end{proof}

Let $\tilde B_{t,N}^\pm$ and $\tilde C^\pm_{t,N}$ be the $N\times N$ matrices defined by the right hand side of \eqref{eq:AtNdef} with $\tilde K_t(\cdot,\cdot)$ replaced by $\tilde K_t^\pm(\cdot,\cdot)$ and  $\tilde L_t^\pm(\cdot,\cdot)$, respectively, and define the $2N\times 2N$ matrix $A_{t,N}$ by
\[A_{t,N} := \begin{pmatrix} B_{t,N}^{-} & C_{t,N}^{-} \\ C_{t,N}^{+} & B_{t,N}^{+} \end{pmatrix}, \qquad t\in \R, \;\; N\in \N.\]
Recalling Remark \ref{rem:matrix}, the matrix $A_{t,N}$ and the operator $K_{t,N}$ are related by Lemma \ref{lem:reduction_to_matrices}(ii) and (iii). As in \S\ref{sec:oneside}, our goal is to estimate the spectrum and spectral radius of $K_{t,N}$ via computing the spectra of approximations to $A_{t,N}$ for some fixed $N$ and finitely many $t\in [-\pi,\pi]$.

To define these approximations, proceeding analogously to \S\ref{sec:oneside},
for $M \in \N$ let
$\tilde{K}_t^{\pm,M}(\cdot,\cdot)$ and $\tilde{L}_t^{\pm,M}(\cdot,\cdot)$ be approximations to $\tilde{K}_t^{\pm}(\cdot,\cdot)$ and $\tilde{L}_t^{\pm}(\cdot,\cdot)$, respectively, given by \eqref{eq:ktall2} and \eqref{eq:L+} but  with the infinite series replaced by  finite sums from $j = -M$ to $M$ (cf.~\eqref{eq:B_t^M}).
Let $B_{t,N}^{\pm,M}$ and $C_{t,N}^{\pm,M}$ be the corresponding approximations to the matrices $B_{t,N}^{\pm}$ and $C_{t,N}^{\pm}$, so that
\[B_{t,N}^{\pm,M}(p,q) = \frac{1}{N}\tilde{K}_t^{\pm,M}(x_{p,N},x_{q,N}), \quad C_{t,N}^{\pm,M}(p,q) = \frac{1}{N}\tilde{L}_t^{\pm,M}(x_{p,N},x_{q,N}), \quad p,q=1,\ldots,N,\]
and let
\begin{equation} \label{eq:AtNMdef2}
A_{t,N}^M := \begin{pmatrix} B_{t,N}^{-,M} & C_{t,N}^{-,M} \\ C_{t,N}^{+,M} & B_{t,N}^{+,M} \end{pmatrix}, \qquad t\in \R, \;\; M,N\in \N.
\end{equation}
Similarly, define the operators $\tilde K^{\pm,M}_{t,N}$ and $\tilde L^{\pm,M}_{t,N}$ as we defined $\tilde K^\pm_{t,N}$ and $\tilde L^\pm_{t,N}$ above Proposition \ref{prop:norm_estimate2}, but replacing the kernels $\tilde{K}_t^{\pm}(\cdot,\cdot)$ and $\tilde{L}_t^{\pm}(\cdot,\cdot)$ in their definitions by the respective approximations $\tilde{K}_t^{\pm,M}(\cdot,\cdot)$ and $\tilde{L}_t^{\pm,M}(\cdot,\cdot)$, and let
\begin{equation}\label{e:tilde K_tNM}
	\tilde K_{t,N}^M =
	\begin{pmatrix}
	\tilde K^{-,M}_{t,N} & \tilde L^{-,M}_{t,N}\\
	\tilde L^{+,M}_{t,N} & \tilde K^{+,M}_{t,N}
	\end{pmatrix}, \qquad t\in \R,\;\; M,N\in \N.
\end{equation}
Then, similarly to \eqref{eq:eMbound}, we have that, for $t\in \R$ and $M,N\in \N$ with
$M \geq 2$,
\begin{equation}
\|A_{t,N}-A_{t,N}^M\|_\infty \leq \|\tilde K_{t,N}-\tilde K_{t,N}^M\|_\infty \leq C_1(M),
\end{equation}
 where, using the notations of Remarks \ref{rem:series} and \ref{rem:series2} and Proposition \ref{prop:k_t_estimate general2}, and setting $\fK_0:= \fK_{0,+}=\fK_{0,-}$,
 \begin{align} \nonumber
 C_1(M) &:= \frac{1}{\pi}\max\Big\{\left(1+\fF_{0,+}^2\right)^{1/4}\big(2|\log\alpha|\fF_{0,+}\mathcal{B}_M(\alpha) + (|\log\alpha|\fF_{0,-}+\fK_0)\cC_M(\alpha)\big),\\ \label{eq:C1M2}
 & \hspace{12.5ex} \left(1+\fF_{0,-}^2\right)^{1/4}\big(2|\log\alpha|\fF_{0,-}\mathcal{B}_M(\alpha) + (|\log\alpha|\fF_{0,+}+\fK_0)\cC_M(\alpha)\big)\Big\}.
 \end{align}
Note that by \eqref{eq:BnAsy} and \eqref{eq:Cn}, $C_1(M) = O(\alpha^{M/2})$ as $M\to\infty$.

Arguing as in \eqref{eq:BNMdef}, we have also that, for $M,N\in \N$, $t\in \R$, $p,q=1,\ldots,N$,
\begin{equation} \label{eq:ElementBounds}
\left|\frac{\partial}{\partial t} B_{t,N}^{\pm,M}(p,q)\right| \leq B_N^{\pm,M}(p,q) \quad \mbox{and} \quad \left|\frac{\partial}{\partial t} C_{t,N}^{\pm,M}(p,q)\right| \leq C_N^{\pm,M}(p,q),
\end{equation}
where $B_N^{\pm,M}(p,q)$ is defined by \eqref{eq:BNMdef} but with $p_j$, $q_j$, and $f$ replaced by $p_j^\pm$, $q_j^\pm$ and $f_\pm$, respectively (compare \eqref{eq:ktall} and \eqref{eq:ktall2}). Similarly, $C_N^{\pm,M}(p,q)$ is defined by the right hand side of \eqref{eq:BNMdef} with $p_j$ and $q_j$ replaced by $\tilde p_j^\pm$ and $\tilde q_j^\pm$, respectively, and with $f$ replaced by $f_\pm$ in the numerator, by $f_\mp$ in the denominator (compare \eqref{eq:ktall} and \eqref{eq:L+}). Thus, where
\begin{equation} \label{eq:BNMdef2}
B_{N}^M := \begin{pmatrix} B_{N}^{-,M} & C_{N}^{-,M} \\ C_{N}^{+,M} & B_{N}^{+,M} \end{pmatrix}, \qquad M,N\in \N,
\end{equation}
\eqref{eq:C2def} holds (with the above definitions of $A_{t,N}$, $A_{t,N}^M$ and $B_N^M$) for all $s,t\in \R$, $M,N\in \N$. Note also, arguing as in \eqref{eq:BNM} and \eqref{eq:C2b2}, that $\|B_N^M\|_\infty$ is bounded uniformly for $M,N\in \N$, and that $\|\tilde K_{t,N}-\tilde K_{s,N}\|_\infty = O(|s-t|)$ as $|s-t|\to 0$, uniformly for $s,t\in \R$, $N\in \N$. Further, similarly to \eqref{eq:NormBound} and \eqref{eq:C1M2}, $\|\tilde K_{t,N}\|_\infty \leq C_3$ for $t\in \R$ and $N\in \N$, where\footnote{Computations indicate that $\mathcal{C}_{10}^*(\alpha)$ exceeds the left hand side of \eqref{eq:sum2bound} by not more than $2.1\%$ for $\alpha\in (0,1)$.}
\begin{align} \nonumber
C_3 &:= \frac{1}{2\pi}\max\Big\{\left(1+\fF_{0,+}^2\right)^{1/4}\big(\fG_{0,+} \mathfrak{R}(\alpha)+4|\log \alpha|\fF_{0,+}\mathcal{B}^*_{10}(\alpha) + (|\log \alpha|\fF_{0,-}+\fK_0)\cC^*_{10}(\alpha)\big),\\ \label{eq:normbound2}
 & \hspace{11.5ex} \left(1+\fF_{0,-}^2\right)^{1/4}\big(\fG_{0,-} \mathfrak{R}(\alpha)+4|\log \alpha|\fF_{0,-}\mathcal{B}^*_{10}(\alpha) + (|\log \alpha|\fF_{0,+}+\fK_0)\cC^*_{10}(\alpha)\big)\Big\},
 \end{align}
$\mathfrak{R}(\alpha):=(1 + \alpha^{1/2} + \alpha^{-1/2})/(2\alpha^2)$, $\mathcal{B}^*_{10}$ and $\cC^*_{10}$ are defined by \eqref{eq:twosided} and \eqref{eq:sum2bound}, and the other notations are defined in Proposition \ref{prop:k_t_estimate general2}. Provided $c$ is such that the conditions of Proposition \ref{prop:k_t_estimate general2} hold, we have also (cf.~\eqref{eq:boundcombined}), by Propositions \ref{prop:norm_estimate2} and \ref{prop:k_t_estimate general2}, that \eqref{eq:boundcombined} holds with
\begin{equation} \label{eq:C4newdef}
C_4 := 2e^{2\pi c}\widetilde C^*,
\end{equation}
where $\widetilde C^*$ is defined by the right hand side of \eqref{eq:C*def}, but with the norms on that right hand side replaced by the upper bounds in Proposition \ref{prop:k_t_estimate general2}. Additionally, we replace the infinite sums in these upper bounds by the bounds $\mathcal{B}_{10}^*(\alpha)$ and $\cC_{10}^*(\alpha)$.

The following result, which holds for every $\Gamma$ that satisfies our standing assumption \eqref{eq:standing2}, is identical to Theorem \ref{thm:Hausdorff_convergence}, except that $\tilde K_{t,N}^{M_N}$ and $A_{t,N}^{M_N}$ are defined here by \eqref{e:tilde K_tNM} and \eqref{eq:AtNMdef2}, respectively.  In this theorem $\sigma(\tilde{K}^M_{t,N})$ denotes the spectrum of $\tilde K^M_{t,N}$ either on $(L^2(0,1))^2$ or on $(C[0,1])^2$ (cf.~\eqref{eq:SpecSame}), which coincides with $\{0\}\cup \sigma(A^M_{t,N})$ by Lemma \ref{lem:reduction_to_matrices}(ii) and Remark \ref{rem:matrix}, and $\sigma(D_\Gamma)$ denotes the spectrum of $D_\Gamma$ on $L^2(\Gamma)$, which coincides with the essential spectrum by Corollary \ref{cor:ess_ne}. The proof of this result follows that of Theorem \ref{thm:Hausdorff_convergence}, noting that the argument of Lemma \ref{lem:ColCom} applies to each of the operator families $\tilde K_{t,N}^{\pm,M}$ and $\tilde L_{t,N}^{\pm,M}$ so that (where $\tilde K_{t,N}^{M}$ is defined by \eqref{e:tilde K_tNM}) $\{\tilde K_{t,N}^M:t\in [-\pi,\pi],M,N\in \N\}$ is collectively compact.
\begin{thm} \label{thm:Hausdorff_convergence_2}
Choose sequences $(m_N)_{N\in \N}, (M_N)_{N\in \N}\subset \N$ such that $m_N, M_N\to\infty$ as $N\to\infty$, and for each $N$, let $T_N:=\{\pm (k-1/2)\pi/m_N:k=1,\ldots,m_N\}$. Then, as $N\to\infty$,
\begin{equation} \label{eq:sigmaNDef2}
\sigma^N(D_\Gamma):= \bigcup\limits_{t \in T_N} \sigma(\tilde{K}^{M_N}_{t,N})  = \{0\} \cup \bigcup\limits_{t \in T_N} \sigma(A^{M_N}_{t,N}) \;\toH \; \sigma(D_{\Gamma})=\sigma_{\ess}(D_\Gamma).
\end{equation}
\end{thm}

 \begin{rem} \label{rem:reduce2} Using Remark \ref{rem:matrix}, we see that Remark \ref{rem:symm2} applies also to $\sigma^N(D_\Gamma)$ given by \eqref{eq:sigmaNDef2}.
 \end{rem}

Our second main result of this subsection, obtained by applying Theorem \ref{thm:GspecRMfinal}, noting Remark \ref{rem:matrix}, is proved in the same way as the analogous result, Theorem \ref{thm:spectral_radius}, in the one-sided case. In this theorem $A_{t,N}^M$, $B_N^M$, $C_1(M)$, $C_3$, and $C_4$, are defined by \eqref{eq:AtNMdef2}, \eqref{eq:BNMdef2}, \eqref{eq:C1M2}, \eqref{eq:normbound2}, and \eqref{eq:C4newdef}, respectively.

\begin{thm} \label{thm:spectral_radius2}
Under our standing assumption \eqref{eq:standing2}, define $g_\pm$ by \eqref{eq:gdef2} so that, by Lemma \ref{lem:f_anal},  $g_\pm$ have  analytic continuations to $\Sigma_c$, for some $c>0$, such that $g_\pm$, $g_\pm^\prime$, and $g_\pm^{\prime\prime}$ are bounded on $\Sigma_c$, and, without loss of generality, assume that \eqref{c:bound2_2} holds. Suppose that $\rho_0>0$, $m,M,N\in \N$, with $M\geq 2$. For $k=1,\ldots,m$, set $t_k:= (k-1/2)\pi/m$, suppose that $\rho(A^M_{t_k,N})<\rho_0$ for $k=1,\ldots,m$, and recursively define $\mu_{k,\ell}$, for $\ell=1,2,\ldots$, by $\mu_{k,1}:= \rho_0$, and by \eqref{eq:lambda_{k,l}}.
Further, for $k=1,\ldots,m$, let $n_k$ denote the smallest integer such that $\sum_{\ell=1}^{n_k} \nu_{k,\ell} \geq 4\pi \rho_0$, and let
$$
R_{m,M,N} := \min_{\stackrel{k=1,\ldots,m}{\ell=1,\ldots,n_k}}\left(\frac{\nu_{k,\ell}}{4}+\frac{\nu_{k,\ell+1}}{2}\right).
$$
If
\begin{equation} \label{eq:estimate_to_check2}
\mathcal{L}_c(f,N) := \frac{C_4}{e^{2\pi Nc}-1} < \mathcal{R}_c(f,m,M,N) := r^2_0\left(1 + C_3\left(R_{m,M,N} - C_1(M)- \frac{\pi}{2m}\|B_N^M\|_\infty\right)^{-1} \right)^{-1},
\end{equation}
then $\rho(D_\Gamma;L^2(\Gamma)) < \rho_0$. Conversely, if $\rho(D_\Gamma;L^2(\Gamma)) < \rho_0$, then, provided $m$ and $M$ are sufficiently large, there exists $N_0\in \N$ such that \eqref{eq:estimate_to_check2} holds and $\rho(A_{t,N}^M)<\rho_0$ for $t\in [0,\pi]$ and all $N\geq N_0$.
\end{thm}

\begin{rem} \label{rem:final}
The condition \eqref{eq:estimate_to_check2} can be written as
$$
S_c(\Gamma,m,M,N)<0,
$$
where $S_c(\Gamma,m,M,N):= \cL_c(f,N)-\cR_c(f,m,M,N)$.
For any $c>0$ such that the conditions of the first sentence of the above theorem are satisfied (and these conditions are necessarily satisfied for all sufficiently small $c>0$), we note that:

i) The proof of the above theorem, which follows that of Theorem \ref{thm:spectral_radius}, shows that, provided $\rho(D_\Gamma;L^2(\Gamma)) < \rho_0$, $R_{m,M,N}$ is positive and bounded away from zero for all sufficiently large $m,M,N\in \N$, so that the same holds for $\cR_c(f,m,M,N)$. Since also $\cL_c(f,N)\to 0$ as $N\to\infty$,  $S_c(\Gamma,m,M,N)<0$ for all sufficiently large $m$, $M$, and $N$, if $\rho(D_\Gamma;L^2(\Gamma)) < \rho_0$.

ii) For fixed $m,M,N\in \N$, evaluation of the functionals $\cL(f,N)$ and $\cR_c(f,m,M,N)$ requires inputs relating to the functions $f_\pm$ defined by \eqref{eq:gdef2}, namely: the constant $\alpha \in (0,1)$ and the bounds $\|g_\pm\|_0$, $\|g_\pm^\prime\|_0$, and $\|g_\pm^{\prime\prime}\|_0$ on $g$ defined by \eqref{eq:gdef2} (to evaluate $C_1(M)$ and $C_3$); the bounds $\|g_\pm\|_c$, $\|g_\pm^\prime\|_c$, and $\|g_\pm^{\prime\prime}\|_c$ (to evaluate $C_4$); the values $f_\pm(\alpha^{x_{p,N}+j})$, $f_\pm^\prime(\alpha^{x_{p,N}})$, and $f_\pm^{\prime\prime}(\alpha^{x_{p,N}})$, for $p=1,\ldots, N$ and $j=-M,\ldots,M$ (to compute the entries of the matrices $A_{t_k,N}^M$, for $k=1,\ldots,m$).

iii) If exact values of $\|g_\pm\|_0$, $\|g_\pm^\prime\|_0$, $\|g_\pm^{\prime\prime}\|_0$, $\|g_\pm\|_c$, $\|g_\pm^\prime\|_c$, and $\|g_\pm^{\prime\prime}\|_c$ are not available, it is enough to replace them with upper bounds which leads to upper bounds $\widehat C_1(M)$, $\widehat C_3$ and $\widehat C_4$ for $C_1(M)$, $C_3$ and $C_4$, respectively (with $\widehat C_1(M)\to 0$ as $M\to\infty$ at the same rate as $C_1(M)$). Let $\widehat \cL_c(f,N)$ and $\widehat \cR_c(f,m,M,N)$ be defined by \eqref{eq:estimate_to_check} with $C_1(M)$, $C_3$, $C_4$, replaced by their upper bounds, so that $\cL_c(f,N)\leq \widehat \cL_c(f,N)$ and $\widehat \cR_c(f,m,M,N)\leq \cR_c(f,m,M,N)$. Then the conclusions of Theorem \ref{thm:spectral_radius2} hold with  $\cL_c(f,N)$ and $\cR_c(f,m,M,N)$ replaced by $\widehat \cL_c(f,N)$ and $\widehat \cR_c(f,m,M,N)$, respectively; in particular $S_c(\Gamma,m,M,N)<0$ if $\widehat \cL_c(f,N) < \widehat \cR_c(f,m,M,N)$ and, provided $\rho(D_\Gamma;L^2(\Gamma)) < \rho_0$, $\widehat \cL_c(f,N) < \widehat \cR_c(f,m,M,N)$ for all sufficiently large $m$, $M$, and $N$.
\end{rem}

\subsection{The 2D case: lower bounds for the numerical range} \label{sec:num_range}

This paper is motivated by the question at the end of  \S\ref{sec:src_main} which notes that there exist (e.g., \cite[Fig.~3]{ChaSpe:21}) 2D examples of $\Gamma\in \mathcal{D}$ with $w_{\ess}(D_\Gamma)\gg \half$ and asks: is $\rho_{L^2(\Gamma),\ess}\geq \half$ for any of these examples? We will explore this
in \S\ref{sec:num_ex} where we will see that $\rho_{L^2(\Gamma),\ess)}< \half$ for each example we treat, supporting Conjecture \ref{con:kenigmod}. We will also see that $w_{\ess}(D_\Gamma)>\half$ for at least some of these examples, indeed that $W_{\ess}(D_\Gamma)\supset B_R:=\{z\in \C:|z|<R\}$ for $R$ substantially larger than $\half$.

As the route to obtain these estimates for $W_{\ess}(D_\Gamma)$, in this subsection we obtain lower bounds\footnote{By a {\em lower bound} for $W(D_\Gamma)$ we mean, simply, a set $S\subset \C$ such that $S\subset W(D_\Gamma)$.} for $W(D_\Gamma)=W_{\ess}(D_\Gamma)$ in the case when $\Gamma$ is a dilation invariant graph, precisely in the 2D case when either \eqref{eq:standing} or \eqref{eq:standing2} applies. It is enough to restrict attention to the one-sided case \eqref{eq:standing} as one easily sees (e.g., \cite[\S2.3]{ChaSpe:21}) that if $\Gamma$ satisfies \eqref{eq:standing2} and $\widetilde \Gamma := \{(x,f(x)):x>0\}$, then $\widetilde \Gamma$ satisfies \eqref{eq:standing} and
\begin{equation} \label{eq:inclusion}
W(D_{\widetilde \Gamma})\subset W(D_\Gamma).
\end{equation}

We obtain lower bounds for $W(D_\Gamma)$ in the case that \eqref{eq:standing} applies via the Nystr\"om method that we used in \S\ref{sec:oneside} to approximate $\sigma(D_\Gamma)$.
Firstly, by Corollary \ref{cor:norm_and_spec}, and since, for $t\in [-\pi,\pi]$, $\tilde L_t$, defined in the proof of Proposition \ref{prop:norm_estimate}, is unitarily equivalent to $\tilde K_{t}$ defined by \eqref{eq:tilK}, which is unitarily equivalent to $K_{t}$, we have that $W(\tilde L_t)\subset W(D_\Gamma)$.  For $p\in \N_{0}$, let $\Tc_p\subset L^2(0,1)$ denote the set of trigonometric polynomials of degree at most $p$. An orthonormal basis of $\Tc_p$ is given by $\set{e_{-p},\ldots,e_p}$, where $e_j(x) := e^{2\pi ijx}$. Consider the restricted numerical range
\[W_p(\tilde L_t) := \set{\langle\tilde L_t\phi,\phi\rangle : \phi \in \Tc_p,\norm{\phi}_{L^2(0,1)} = 1}.\]
It is clear that $W_p(\tilde L_t) \subset W(\tilde L_t)$ for $t\in [-\pi,\pi]$, and that (e.g., \cite[\S2.3]{ChaSpe:21}) $W_p(\tilde L_t)=W(T^{p,t})$, where $T^{p,t} = (T^{p,t}_{jk})_{j,k = -p}^p$ is the $(2p+1) \times (2p+1)$ matrix with $T^{p,t}_{jk} := \langle\tilde L_t e_k,e_j\rangle$, $j,k=-p,\dots,p$. Recalling from the proof of Proposition \ref{prop:norm_estimate} that $\tilde L_t$ has kernel $\tilde{L}_t(x,y) = e^{i(x-y)t}\tilde{K}_t(x,y)$, we will approximate $T^{p,t}$ by $T^{p,t,N}$, where
\begin{equation} \label{eq:TpNdef}
T^{p,t,N}_{jk} := \frac{1}{N^2}\sum\limits_{m,n = 1}^N e^{i(x_{m,N}-x_{n,N})t}\tilde K_t(x_{m,N},x_{n,N})e_k(x_{n,N})\overline{e_j(x_{m,N})}, \qquad j,k=-p,\ldots,p,
\end{equation}
recalling that $x_{m,N} := \tfrac{1}{N}\left(m-\tfrac{1}{2}\right)$, $m=1,\ldots,N$, and then approximate $T^{p,t,N}$ further by $T^{p,t,N,M}$, defined by \eqref{eq:TpNdef} but with $\tilde K_{t}$, given by the infinite sum \eqref{eq:ktall}, replaced by $\tilde K_{t}^M$, given by the finite sum \eqref{eq:B_t^M}. The following result, which uses the notation \eqref{eq:C1def}, will enable us to estimate the difference between $W_p(\tilde L_t)=W(T^{p,t})$ and $W(T^{p,t,N,M})$.

\begin{prop} \label{prop:NRerror}
Let $f$ and $c>0$ be as in Proposition \ref{prop:k_t_estimate general}, in particular satisfying \eqref{c:bound2}, so that $\tilde L_t(x,\cdot)$ and $\tilde L_t(\cdot,y)$ are analytic and bounded in $\Sigma_c$ for all $x,y,t\in \R$,  let $\tilde L^M_t(x,y) := e^{i(x-y)}\tilde K^M(x,y)$, for $x,y\in \R$ and $M\in \N$, and suppose that
$g$ is a bounded, analytic, $1$-periodic function on $\Sigma_c$.
Then, for $M,N\in \N$ and $t\in [-\pi,\pi]$,
\[\abs{\langle\tilde L_t g,g\rangle - \frac{1}{N^2}\sum\limits_{m,n = 1}^N \tilde{L}^{M}_{t}(x_{m,N},x_{n,N})g(x_{n,N})\overline{g(x_{m,N})}} \leq 2\norm{g}_0\norm{g}_c e^{\pi c}\frac{\|\tilde K_t\|_{0,c} + \|\tilde K_t\|_{c,0}}{e^{2\pi Nc}-1} + C_1(M).\]
\end{prop}

\begin{proof}
\begin{align*}
&\abs{\langle\tilde L_t g,g\rangle - \frac{1}{N^2}\sum\limits_{m,n = 1}^N \tilde L_t(x_{m,N},x_{n,N})g(x_{n,N})\overline{g(x_{m,N})}}\\
\leq &\int_0^1 \abs{g(x)}\abs{\int_0^1 \tilde L_t (x,y)g(y) \, \mathrm{d}y - \frac{1}{N}\sum\limits_{n = 1}^N \tilde L_t (x,x_{n,N})g(x_{n,N})} \, \mathrm{d}x\\
&+ \frac{1}{N} \sum\limits_{n = 1}^N \abs{\int_0^1 \overline{g(x)}\tilde L_t (x,x_{n,N}) \, \mathrm{d}x - \frac{1}{N}\sum\limits_{m = 1}^N \overline{g(x_{m,N})}\tilde L_t (x_{m,N},x_{n,N})}\abs{g(x_{n,N})}\\
\leq &\norm{g}_0\frac{2\|\tilde L_t \|_{0,c}\norm{g}_c}{e^{2\pi Nc}-1} + \frac{2\norm{g}_c\|\tilde L_t \|_{c,0}}{e^{2\pi Nc}-1}\norm{g}_0,
\end{align*}
by Theorem \ref{thm:Kress}, noting that $g$ and, from the proof of Proposition \ref{prop:norm_estimate}, $\tilde L_t (x,\cdot)$, for $x\in \R$, and $\tilde L_t (\cdot,y)$, for $y\in \R$, are all $1$-periodic. The claimed result now follows immediately from \eqref{eq:eMbound} and by  noting that, from the proof of Proposition \ref{prop:norm_estimate}, $\|\tilde L_t\|_{0,c}\leq e^{\pi c} \|\tilde K_t\|_{0,c}$ and $\|\tilde L_t\|_{c,0}\leq e^{\pi c} \|\tilde K_t\|_{c,0}$.
\end{proof}

Clearly, $g \in \Tc^p$ with $\norm{g}_{L^2(0,1)} = 1$ if and only if
\begin{equation} \label{eq:grep}
g = \sum\limits_{j = -p}^p c_je_j \quad \text{with} \quad \sum\limits_{j = -p}^p \abs{c_j}^2 = 1,
\end{equation}
in which case
\[\norm{g}_c \leq \sum\limits_{j = -p}^p \abs{c_j}\norm{e_j}_c \leq \sum\limits_{j = -p}^p \abs{c_j}\norm{e_j}_c \leq \sum\limits_{j = -p}^p \abs{c_j}e^{2\pi p c} \leq \sqrt{2p+1} \; e^{2\pi p c}.\]
The following corollary follows immediately from this observation and Proposition \ref{prop:NRerror}, noting that, if $g$ and $c_{-p},\ldots,c_p$ are related by \eqref{eq:grep}, then the summation in Proposition \ref{prop:NRerror} coincides with $\sum_{j,k=-p}^p T_{j,k}^{p,t,N,M}c_k\overline{c_j}\in W(T^{p,t,N,M})$.

\begin{cor} \label{cor:num_range}
For $p\in \N_0$, $M,N\in \N$, $t\in [-\pi,\pi]$, where $c>0$ is as in Proposition \ref{prop:NRerror},
\[d_H(W_p(\tilde L_t ),W(T^{p,t,N,M})) \leq C_7(p,N,M) := 2(2p+1)e^{\pi c(2p+1)}\frac{\|\tilde K_{t}\|_{0,c} + \|\tilde K_{t}\|_{c,0}}{e^{2\pi Nc}-1} + C_1(M).\]
\end{cor}

We can approximate $W(T^{p,t,N,M})$ by a standard method that dates back to \cite{Jo78}.
Choose $n\in \N$ and, for $\ell=0,\ldots,n$, let $\theta_\ell:= 2\pi \ell/n$, let $\lambda_\ell$ denote the largest eigenvalue of $\Real(e^{-i\theta_\ell}T^{p,t,N,M})$ and $x_\ell$ an associated unit eigenvector, and let $z_\ell:= \langle T^{p,t,N,M}x_\ell,x_\ell\rangle$, for $\ell=0,1,\ldots,n-1$, and $z_n:= z_0$. Then $W_n:= \conv(\{z_0,\ldots,z_n\})\subset W(T^{p,t,N,M})$
 and \cite{Jo78} $W_n\toH W(T^{p,t,N,M})$ as $n\to\infty$. The following simple corollary of the above results, in which we use the notations just introduced,
 will enable us to show that $B_R \subset W(\tilde L_t )\subset W(D_\Gamma)$ for concrete values of $R>0$ in the examples we treat in \S\ref{sec:num_ex}.

\begin{cor} \label{cor:num_range2}
Suppose that $p\in \N_0$, $M,N,n\in \N$, $t\in [-\pi,\pi]$,
 that $c>0$ is as in Proposition \ref{prop:NRerror}, and that $0$ is an interior point of $W_n$, in which case $z_\ell = \exp(i\gamma_\ell)|z_\ell|$, $\ell=0,\ldots,n$, with $\gamma_0\leq \gamma_1\leq \ldots \leq \gamma_n=\gamma_0+2\pi$.
Then, where $\theta_{\max} := \max_{\ell=1,\ldots,n} (\gamma_\ell-\gamma_{\ell-1})\in (0,2\pi]$ and $R_{\min}:= \min_{\ell=0,\ldots,n}|z_\ell|>0$, if $R^*:= R_{\min}\cos(\theta_{\max}/2)-C_7(p,t,N,M) > 0$, then $\overline{B_{R^*}}\subset W(D_\Gamma)$.
\end{cor}
\begin{proof} Under the above assumptions, $\overline{B_R}\subset W_n$, where $R:= R_{\min}\cos(\theta_{\max}/2)$. Since $W_n\subset W(T^{p,t,N,M})$ and $W_p(\tilde L_t )\subset W(\tilde L_t )\subset W(D_\Gamma)$, it follows that $\overline{B_{R^*}}\subset W(D_\Gamma)$ by Corollary \ref{cor:num_range}.
\end{proof}

\section{Localization and deformation} \label{sec:LD}
Let $\Omega_- \subset \mathbb{R}^d$, $d \geq 2$, be a Lipschitz domain with boundary $\Gamma = \partial \Omega_-$ and outward-pointing unit normal vector $n_y = n^\Gamma_y$ at almost every $y \in \Gamma$. If $\psi \colon \mathbb{R}^d \to \mathbb{R}^d$ is a $C^1$-diffeomorphism, then $\psi(\Omega_-)$ is a Lipschitz domain with boundary $\psi(\Gamma)$ by \cite[Theorem 4.1]{HMT07}. Note that this is no longer true if $\psi$ is only assumed to be a bi-Lipschitz map (see \cite[Lemma 1.2.1.4]{Gr:85}, and the discussion in \cite{HMT07}). As before, $D_\Gamma$ will be the double-layer operator on $\Gamma$. To simplify notation we will use $\sim$ to denote equality up to compact operators, and abbreviate $D_\Gamma-\lambda I$ as $D_\Gamma-\lambda$. For $x\in \R^d$ and $r>0$, $B(x,r):= \{y\in \R^d:|y-x|<r\}$.

\subsection{Localization without deformation} \label{sec:localwithout}
To start, we will consider domains with locally-dilation-invariant boundaries $\Gamma \in \mathscr{D}$; recall Definition \ref{defn:ldi}. Note that being locally dilation invariant at $x$ is equivalent to the existence of an isometry $\psi_x : \R^d \to \R^d$ with $\psi_x(x) = 0$ and $\psi_x\big(B(x,\delta(x)) \cap \Gamma\big) = B(0,\delta(x)) \cap \Gamma_x$, for some $\delta(x)>0$, where $\Gamma_x$ is the graph of a Lipschitz continuous function with $\alpha\Gamma_x = \Gamma_x$. We also assume that $\psi_x$ preserves the orientation of the outer normal vector field, that is, the outward-pointing normal vector field on $B(x,\delta(x)) \cap \Gamma$ is mapped to the upwards-pointing normal vector field on $B(0,\delta(x)) \cap \Gamma_x$. Similarly, $\Gamma$ is locally $C^1$ at $x$ if there is a $\delta(x) > 0$ and an Euclidean map $\psi_x : \R^d \to \R^d$ with $\psi_x(x) = 0$ and $\psi_x(B(x,\delta(x)) \cap \Gamma) = B(0,\delta(x)) \cap \Gamma_x$, where $\Gamma_x$ is the graph of a $C^1$-function with compact support.  Again, we assume that $\psi_x$ preserves the orientation of the outer normal vector field.  We will use extensively below the notation $\delta(x)$, for $x\in \Gamma$, which will have one of the above meanings, depending on whether $\Gamma$ is locally dilation invariant at $x$ or is $C^1$ at $x$.

Note that $D_{\Gamma}$ is compact for every $C^1$ domain \cite{FaJoRi:78}, as we recalled in \S\ref{sec:main_sig}; this implies that also $D_{\widetilde \Gamma}$ is compact if $\widetilde \Gamma$ is the graph of a $C^1$-function with compact support. Consequently, if $\Gamma$ consists of different parts that are separated by $C^1$ areas, we are able to localize using sharp cut-off functions. In the following we equip  $\Gamma$ with the topology induced from $\R^d$ and the  standard surface measure. In particular, $\chi_E$ denotes the characteristic function of a subset $E \subset \Gamma$ that is measurable with respect to the surface measure on $\Gamma$.

\begin{lem} \label{lem:sharp_cutoff}
	Let $E \subset  \Gamma$ be a measurable subset and assume that $\Gamma$ is locally $C^1$ at every $x \in \partial E$. Then $D_{\Gamma}$ essentially commutes with $\chi_E$, i.e., the commutator $[D_\Gamma,\chi_E]:= D_\Gamma\chi_E-\chi_ED_\Gamma$ is compact.
\end{lem}

\begin{proof}
	Since $\partial E$ is a compact subset of $\Gamma$, there is an $\epsilon > 0$ such that
	\[(\partial E + B(0,\epsilon)) \cap \Gamma \subset  \bigcup\limits_{x \in \partial E} (B(x,\delta(x)) \cap \Gamma).\]
	Set $E_o := (\partial E + B(0,\tfrac{\epsilon}{2})) \cap \Gamma$. We know that $\chi_AD_{\Gamma}\chi_B$ is compact, even Hilbert--Schmidt, whenever $A$ and $B$ have positive distance from each other. We have
	\begin{align*}
		E = (E \cap E_o) \cup (E \setminus E_o) \quad \mbox{and} \quad
		E^c = (E^c \cap E_o) \cup (E^c \setminus E_o),
	\end{align*}
where $E^c := \Gamma \setminus E$. By inspection, we see that the four sets on the right of these equalities, if non-empty, have pairwise positive distance from each other except for $E \cap E_o$ and $E^c \cap E_o$. We thus have
	\begin{align*}
		\chi_ED_{\Gamma}\chi_{E^c} &\sim \chi_{E \cap E_o}D_{\Gamma}\chi_{E^c \cap E_o} = \chi_E\chi_{E_o}D_{\Gamma}\chi_{E_o}\chi_{E^c},\\
		\chi_{E^c}D_{\Gamma}\chi_E &\sim \chi_{E^c \cap E_o}D_{\Gamma}\chi_{E \cap E_o} = \chi_{E^c}\chi_{E_o}D_{\Gamma}\chi_{E_o}\chi_E.
	\end{align*}
	The operator $\chi_{E_o}D_{\Gamma}\chi_{E_o}$ is compact because $\overline{E}_o$ is $C^1$ everywhere by construction. We conclude that $[D_{\Gamma},\chi_E] = \chi_{E^c}D_{\Gamma}\chi_E - \chi_ED_{\Gamma}\chi_{E^c}$ is compact as well.
\end{proof}

In particular, this shows that if $\Gamma$ is locally $C^1$ at some point $x \in \Gamma$, then $D_\Gamma$ is not Fredholm, since if $E = B(x,\tfrac{\delta(x)}{2})$, then $\chi_E D_\Gamma \chi_E$ is compact, and Lemma~\ref{lem:sharp_cutoff} implies that $D_\Gamma \sim \chi_E D_\Gamma \chi_E + \chi_{E^c} D_\Gamma \chi_{E^c}$. Hence
\begin{equation} \label{eq:notFredholm}
0\in \sigma_{\ess}(\chi_E D_\Gamma \chi_E) \subset  \sigma_{\ess}(D_{\Gamma}).
\end{equation}
More generally, if $E_1,\ldots,E_N$ are measurable subsets of $\Gamma$ such that $\Gamma = \cup_{j = 1}^N E_j$, $E_j \cap E_k = \emptyset$ for $j \neq k$ and $\Gamma$ is locally $C^1$ at every $x \in \cup_{j = 1}^N \partial E_j$, then
$D_\Gamma \sim \sum_{j=1}^N \chi_{E_j} D_\Gamma \chi_{E_j}$
and thus
\begin{equation} \label{eq:ess_decomposition}
\sigma_{\ess}(D_{\Gamma}) = \bigcup\limits_{j = 1}^N \sigma_{\ess}(\chi_{E_j}D_{\Gamma}\chi_{E_j}).
\end{equation}

Equation \eqref{eq:ess_decomposition} is a localisation result. We have also the following, more substantial localisation result (cf.~\cite{El:92,Mi:99,LeCoPer:22,ChaSpe:21}) for the case where $\Gamma\in \mathscr{D}$.

\begin{thm} \label{thm:localization}
Let $\Omega_- \subset \R^d$ be a bounded Lipschitz domain such that $\Gamma = \partial\Omega_- \in \mathscr{D}$, and pick $x_1,\ldots,x_N \in \Gamma$ for which $\Gamma \subset  \bigcup\limits_{j = 1}^N B(x_j, \delta(x_j))$. Then
	\[\sigma_{\ess}(D_\Gamma) = \bigcup\limits_{x \in \Gamma} \sigma_{\ess}(D_{\Gamma_x}) = \bigcup\limits_{j = 1}^N \sigma_{\ess}(D_{\Gamma_{x_j}}).\]
\end{thm}

\begin{proof}
Let $x \in \Gamma$. If $\Gamma$ is locally $C^1$ at $x$, then, by \eqref{eq:notFredholm} and since $D_{\Gamma_x}$ is compact, $\sigma_{\ess}(D_{\Gamma_x}) = \set{0} \subset  \sigma_{\ess}(D_{\Gamma})$. So assume that $\Gamma$ is locally dilation invariant at $x$. Then there is an $\alpha \in (0,1)$ such that $\alpha\Gamma_x = \Gamma_x$. Let $\eta \from \Gamma \to [0,1]$ be a Lipschitz continuous function with support in $B(x,\delta(x)) \cap \Gamma$ that is equal to $1$ in a neighbourhood of $x$. Let $\psi_x \from \R^d \to \R^d$ be an isometry with $\psi_x(x) = 0$ and $\psi_x(B(x,\delta(x)) \cap \Gamma) = B(0,\delta(x)) \cap \Gamma_x$. Define $\Psi_x \from L^2(\Gamma_x) \to L^2(\Gamma)$ by
\[\Psi_x\phi(y) = \begin{cases} \phi(\psi_x(y)) & \text{if } y \in B(x,\delta(x)) \cap \Gamma, \\ 0 & \text{otherwise,} \end{cases}\]
which has adjoint $\Psi_x^\prime:L^2(\Gamma)\to L^2(\Gamma_x)$ given by
\[\Psi_x^\prime\phi(y) = \begin{cases} \phi(\psi_x^{-1}(y)) & \text{if } y \in B(0,\delta(x)) \cap \Gamma_x, \\ 0 & \text{otherwise}. \end{cases}\]
$\Psi_x$ and $\Psi_x^\prime$ are partial isometries with initial spaces $L^2\big(B(0,\delta(x)) \cap \Gamma_x\big)$ and $L^2\big(B(x,\delta(x)) \cap \Gamma\big)$, respectively. For $m\in \N$ the  composition $\Psi_x^\prime\eta^{m}\Psi_x:L^2(\Gamma_x)\to L^2(\Gamma_x)$ is the operation of multiplication by $(\eta')^m$, where $\eta'(y) := \eta(\psi^{-1}_x(y))$, for $y\in B(x,\delta(x)) \cap \Gamma_x$, $\eta'(y) := 0$, $y\in \Gamma_x\setminus B(x,\delta(x))$.

Let $V_{\alpha}$ be the unitary dilation as defined in \eqref{eq:dilation} but with $\Gamma$ replaced by $\Gamma_x$. Then
\[V_{\alpha}^n\eta'V_{\alpha}^{-n}\phi(y) = \begin{cases} \eta(\psi_x^{-1}(\alpha^n y))\phi(y) & \text{if } \alpha^n y \in B(0,\delta(x)) \cap \Gamma_x, \\ 0 & \text{otherwise}, \end{cases}\]
for $n \in \N$, $y \in \Gamma_x$ and $\phi \in L^2(\Gamma_x)$. By dominated convergence, it follows that
$V_{\alpha}^n\eta'V_{\alpha}^{-n}\phi \to \phi$ as $n\to\infty$
for all $\phi \in L^2(\Gamma_x)$. (We can see, similarly, that, for every $m\in \N$, $V_{\alpha}^n(\eta')^mV_{\alpha}^{-n}\phi \to \phi$ as $n\to\infty$
for all $\phi\in L^2(\Gamma_x)$.) Using that $\psi_x$ is an isometry, we further observe that $\Psi_x^\prime\eta D_{\Gamma}\eta\Psi_x = \eta'D_{\Gamma_x}\eta'$. This implies, since also $D_{\Gamma_x}V_\alpha = V_\alpha D_{\Gamma_x}$, that
\begin{equation} \label{eq:thm5.2}
V_{\alpha}^n\Psi_x^\prime\eta D_{\Gamma}\eta \Psi_xV_{\alpha}^{-n} = V_{\alpha}^n \eta' D_{\Gamma_x} \eta' V_{\alpha}^{-n} = (V_{\alpha}^n \eta' V_{\alpha}^{-n})D_{\Gamma_x}(V_{\alpha}^n \eta' V_{\alpha}^{-n}) \to D_{\Gamma_x}
\end{equation}
in the strong operator topology.

Now assume that $D_{\Gamma} - \lambda$ is Fredholm and let $A \from L^2(\Gamma) \to L^2(\Gamma)$ be a Fredholm regularizer of $D_{\Gamma} - \lambda$, so that the products $A(D_{\Gamma} - \lambda)$ and $(D_{\Gamma} - \lambda)A$ are  compact perturbations of the identity. Then, for any $\phi \in L^2(\Gamma_x)$,
\begin{equation} \label{eq:thm5.2_2}
\norm{\phi} \leq \norm{V_{\alpha}^n\Psi_x^\prime \eta A \eta (D_{\Gamma} - \lambda) \eta \Psi_xV_{\alpha}^{-n}\phi} + \norm{V_{\alpha}^n(I - \Psi_x^\prime \eta A \eta (D_{\Gamma} - \lambda) \eta \Psi_x)V_{\alpha}^{-n}\phi}.
\end{equation}
The first term can be estimated as
\begin{align*}
\norm{V_{\alpha}^n\Psi_x^\prime \eta A \eta (D_{\Gamma} - \lambda) \eta \Psi_xV_{\alpha}^{-n}\phi} &\leq \norm{A}\norm{\eta (D_{\Gamma} - \lambda) \eta \Psi_xV_{\alpha}^{-n}\phi}\\
&= \norm{A}\norm{V_{\alpha}^n\Psi_x^\prime \eta (D_{\Gamma} - \lambda) \eta\Psi_xV_{\alpha}^{-n}\phi},
\end{align*}
where we used that $V_{\alpha}$ is an isometry, $\Psi_x^\prime$ is an isometry on the range of the multiplication operator $\eta$, and $\norm{\eta} \leq 1$. By \eqref{eq:thm5.2},
\[\norm{V_{\alpha}^n\Psi_x^\prime \eta (D_{\Gamma} - \lambda) \eta\Psi_xV_{\alpha}^{-n}\phi} \to \norm{(D_{\Gamma_x} - \lambda)\phi},\]
as $n\to\infty$. The second term in \eqref{eq:thm5.2_2} can be estimated as
\begin{align*}
&\norm{V_{\alpha}^n(I - \Psi_x^\prime \eta A \eta (D_{\Gamma} - \lambda) \eta \Psi_x)V_{\alpha}^{-n}\phi}\\
&\quad \leq \norm{V_{\alpha}^n(I - \Psi_x^\prime \eta A(D_{\Gamma} - \lambda) \eta^2\Psi_x)V_{\alpha}^{-n}\phi} + \norm{V_{\alpha}^n\Psi_x^\prime \eta A [\eta,D_{\Gamma}] \eta \Psi_xV_{\alpha}^{-n}\phi}\\
&\quad \leq \norm{V_{\alpha}^n(I - \eta'^3)V_{\alpha}^{-n}\phi} + \norm{V_{\alpha}^n\Psi_x^\prime \eta (A(D_{\Gamma} - \lambda) - I) \eta^2\Psi_xV_{\alpha}^{-n}\phi} + \norm{V_{\alpha}^n\Psi_x^\prime \eta A [\eta,D_{\Gamma}] \eta\Psi_xV_{\alpha}^{-n}\phi}.
\end{align*}
Note that both $A(D_{\Gamma} - \lambda) - I$ and $[\eta,D_{\Gamma}]$ are compact. We also observe that for any compact operator $K \from L^2(\Gamma_x) \to L^2(\Gamma_x)$ we have $V_{\alpha}^nKV_{\alpha}^{-n} \to 0$ in the strong operator topology,
since $V_{\alpha}^{-n} \to 0$ in the weak operator topology.
Therefore, and using that $V_{\alpha}^n\eta'^{3} V_{\alpha}^{-n} \to I$, we obtain that the second term in \eqref{eq:thm5.2_2} tends to zero as $n\to\infty$. We conclude that $\norm{\phi} \leq \norm{A}\norm{(D_{\Gamma_x} - \lambda)\phi}$ for every $\phi\in L^2(\Gamma_x)$. The same argument also shows that $\norm{\phi} \leq \norm{A}\norm{(D_{\Gamma_x} - \lambda)^\prime\phi}$ for every $\phi\in L^2(\Gamma_x)$. It follows that $D_{\Gamma_x} - \lambda$ is invertible and thus Fredholm.

Conversely,  choose $x_1,\ldots,x_N \in \Gamma$ such that
$\Gamma \subset  \cup_{i=1}^N B(x_i, \delta(x_i))$, and suppose that $D_{\Gamma_{x_j}} - \lambda$ is Fredholm for $j=1,\ldots,N$, and let $A_{j}$ be a Fredholm regularizer of $D_{\Gamma_{x_j}} - \lambda$.
For every $j$ let $\psi_{x_j} \from \R^d \to \R^d$ be an isometry with $\psi_{x_j}(x_j) = 0$ and
$\psi_{x_j}(B(x_j,\delta(x_j)) \cap \Gamma) = B(0,\delta(x_j)) \cap \Gamma_{x_j}$.
Choose Lipschitz continuous functions $\eta_1,\ldots,\eta_N$ such that $\set{\eta_j^2 : j = 1,\ldots,N}$ is a partition of unity of $\Gamma$ subordinate to the sets $B(x_j,\delta(x_j))$. Now note that, since $D_{\Gamma_{x_j}}$ essentially commutes with $\eta_j'$, $A_j$ essentially commutes with $\eta_j'$ as well. This implies that
\begin{align*}
\sum\limits_{j = 1}^N \Psi_{x_j}\eta_j'A_{j}\Psi_{x_j}^\prime\eta_j(D_{\Gamma} - \lambda) &\sim \sum\limits_{j = 1}^N \Psi_{x_j} A_{j}\eta_j'\Psi_{x_j}^\prime(D_{\Gamma} - \lambda)\eta_j = \sum\limits_{j = 1}^N \Psi_{x_j} A_{j}\Psi_{x_j}^\prime\eta_j(D_{\Gamma} - \lambda)\eta_j\\
&= \sum\limits_{j = 1}^N \Psi_{x_j} A_{j}\eta_j'(D_{\Gamma_{x_j}} - \lambda)\eta_j'\Psi_{x_j}^\prime \sim \sum\limits_{j = 1}^N \Psi_{x_j}\eta_j'A_{j}(D_{\Gamma_{x_j}} - \lambda)\eta_j'\Psi_{x_j}^\prime\\
&\sim \sum\limits_{j = 1}^N \Psi_{x_j}\eta_j'^2\Psi_{x_j}^\prime = \sum\limits_{j = 1}^N \eta_j^2 = I.
\end{align*}
Similarly, $(D_{\Gamma} - \lambda)\sum\limits_{j = 1}^N \eta_j\Psi_{x_j}A_{j}\eta_j'\Psi_{x_j}^\prime \sim I$. Thus $D_{\Gamma} - \lambda$ is Fredholm.
\end{proof}

\subsection{Deformations} \label{sec:deform}

Let $\beta \in (0,1)$. We will now consider the situation where a Lipschitz domain $\Omega_-$ with boundary $\Gamma$ is deformed by a $C^{1,\beta}$-diffeomorphism
$\psi \from \R^d \to \R^d$ that is conformal or anti-conformal at a point $x \in \Gamma$, meaning that the Jacobian $J\psi(x)$ of $\psi$ at $x$ lies in $\mathbb{R} O(d)$, where $\mathbb{R} O(d) := \set{\lambda A \in \R^{d \times d} : \lambda \in \R, A \text{ orthogonal}}$.
 Due to the invariance of the double-layer potential under transformations from $\mathbb{R}O(d)$, we will for our purposes be able to assume that $J\psi(x) = I$. Let $D_\Gamma(\cdot,\cdot)$ denote the kernel of $D_{\Gamma}$. Likewise, $V_j$ will be the integral operator with kernel $V_j(\cdot,\cdot)$, $|D_{\Gamma}|$ will be the integral operator with kernel $|D_{\Gamma}(\cdot,\cdot)|$, and so on. We write
\begin{align} \label{eq:V3}
 D_{\psi(\Gamma)}(\psi(z), \psi(y))
 &= D_\Gamma(z,y) \notag\\
  &+ \underbrace{D_{\psi(\Gamma)}(\psi(z), \psi(y)) - \frac{(z-y) \cdot n^\Gamma(y)}{c_d|\psi(z) - \psi(y)|^d}  n^{\psi(\Gamma)}(\psi(y)) \cdot J\psi(y) n^\Gamma(y)}_{V_1(z,y)} \notag\\
  &+ \underbrace{D_\Gamma(z,y) (n^{\psi(\Gamma)}(\psi(y)) \cdot J\psi(y) n^\Gamma(y) - 1)}_{V_2(z,y)} \notag\\
  &+ \underbrace{D_\Gamma(z,y) n^{\psi(\Gamma)}(\psi(y)) \cdot J\psi(y) n^\Gamma(y) \left[\frac{|z-y|^d}{|\psi(z) - \psi(y)|^d} - 1 \right]}_{V_3(z,y)}.
 \end{align}
The term $V_1$ is weakly singular, satisfying an estimate of the form $|V_1(z,y)| \leq C_{\beta} |z-y|^{1+\beta-d}$ for $\beta \in (0,1)$ and almost every $z,y \in \Gamma$, see \cite[Lemma~6]{Med92}. (In order to apply Medkov\'a's results from \cite{Med92}, note that the reduced boundary $\partial_\ast \Omega_- \subset  \Gamma$ has full surface measure for a Lipschitz domain $\Omega_-$, cf. \cite[p. 209]{EvansGariepy92}.) Therefore the operator $V_1 \colon L^2(\Gamma) \to L^2(\Gamma)$ is compact.

We view $V_2$ as the composition of $D_\Gamma$ with a multiplication operator, noting that by \cite[Lemma~3]{Med92} we have
\[\lim_{\rho \to 0^+} \esssup\limits_{y \in \Gamma \cap B(x, \rho)} |n^{\psi(\Gamma)}(\psi(y)) \cdot J\psi(y) n^\Gamma(y) - 1| = 0.\]
In particular, for every $\epsilon > 0$, there is $\rho > 0$ such that the operator norm of
\[V_2 \colon L^2(\Gamma \cap B(x,\rho)) \to L^2(\Gamma)\]
is smaller than $\epsilon$.

For a bounded Lipschitz domain with boundary $\Gamma$, let $\mathscr{A}_2 = \mathscr{A}_2(\Gamma)$ denote the algebra of integral operators $K$ such that $|K| \colon L^2(\Gamma) \to L^2(\Gamma)$ is bounded.
We will work within the class of domains such that $D_\Gamma \in \mathscr{A}_2$ in this subsection.  The preceding considerations show that $D_{\psi(\Gamma)} \in \mathscr{A}_2(\psi(\Gamma))$ if $D_\Gamma \in \mathscr{A}_2(\Gamma)$. Since the relationship between $\Gamma$ and $\psi(\Gamma)$ is symmetric, we see that
\begin{equation} \label{eq:A2}
D_{\psi(\Gamma)} \in \mathscr{A}_2(\psi(\Gamma)) \textrm{ if and only if } D_\Gamma \in \mathscr{A}_2(\Gamma).
\end{equation}

Note that $D_\Gamma \in \mathscr{A}_2$ for any $C^{1,\beta}$-boundary $\Gamma$, since $D_\Gamma$ has a weakly singular kernel in this case.
It is also known that $|D_{\widetilde{\Gamma}}| \colon L^2(\widetilde{\Gamma}) \to L^2(\widetilde{\Gamma})$ is bounded for any of the following \textit{graphs} $\widetilde{\Gamma}$:
\begin{itemize}	
	\item[(i)] a 2D wedge (modelling a polygonal corner) \cite{Mi:02};
	\item[(ii)] a 3D wedge (modelling a polyhedral edge) \cite{Pe:19};
	\item[(iii)] a 3D polyhedral cone \cite{El:92};
	\item[(iv)] a 3D smooth cone (such as a circular cone) \cite{QiNi:12}.
\end{itemize}
Furthermore, this statement may be extended to dilation-invariant domains which are built from such graphs.

\begin{lem}\label{lem:A2graphs}
	Suppose that $\widetilde{\Gamma}$ is a dilation-invariant graph, i.e., $\widetilde{\Gamma} = \alpha \widetilde{\Gamma}$, for some $\alpha \in (0,1)$,  and suppose that, at every $x \in \widetilde{\Gamma} \setminus \set{0}$, $\widetilde{\Gamma}$ is either locally $C^{1,\beta}$ for some $\beta \in (0,1)$, or coincides locally with a graph of type (i), (ii), (iii), or (iv). Then $|D_{\widetilde{\Gamma}}|$ is bounded on $L^2(\widetilde{\Gamma})$.
\end{lem}

\begin{proof}
By hypothesis, we know that $\abs{D_{j}} \colon L^2(\widetilde{\Gamma}_0) \to L^2(\widetilde{\Gamma}_0)$ is bounded for $j=-1,0,1$, where $D_j$ and $\widetilde{\Gamma}_0$ are as in \S\ref{sec:FB}. The lemma therefore follows from Proposition~\ref{prop:convergence} and \eqref{eq:conv}.
\end{proof}
A simple localization argument now demonstrates that the condition $D_\Gamma \in \mathscr{A}_2$ is quite general.
\begin{thm} \label{thm:locals}
	Suppose that at every $x \in \Gamma$, $\Gamma$ is locally $C^{1,\beta}$, for some $\beta\in (0,1)$, or coincides locally
with a graph satisfying the hypothesis of Lemma~\ref{lem:A2graphs}. Then $D_\Gamma \in \mathscr{A}_2(\Gamma)$.
\end{thm}

Returning to the decomposition of  $D_{\psi(\Gamma)}(\psi(z), \psi(y))$, we note, for the third term $V_3$, the following consequence of the fact that $J\psi(x) = I$.
\begin{lem} \label{lem:kernelmult}
Suppose that $D_\Gamma \in \mathscr{A}_2$. Then
$|V_3|(z,y) = |D_\Gamma|(z,y) Q(z,y)$,
where $Q \in L^\infty(\Gamma \times \Gamma)$ is non-negative and
$\lim_{\rho \to 0^+} \esssup_{y,z \in \Gamma \cap B(x, \rho) } Q(z,y) = 0$,
where the essential supremum is taken with respect to the product measure on $\Gamma \times \Gamma$.
\end{lem}

Let $\eta_\rho \from \Gamma \to [0,1]$ be a Lipschitz continuous function with $\supp \eta_{\rho} \subset B(x,\rho)$ such that $\eta_\rho$ is equal to $1$ in a neighbourhood of $x$. Furthermore, denote the operator of composition with $\psi$ by
\[C_{\psi} \colon L^2(\psi(\Gamma)) \to L^2(\Gamma), \quad C_\psi \phi(z) = \phi(\psi(z)),\]
and let $\eta_{\rho}' = \eta_{\rho} \circ \psi^{-1}$. Note also that we have the change of variables formula
\begin{equation} \label{eq:chofvar}
\int_{\psi(\Gamma)} \phi(y) \, \mathrm{d}s_{\psi(\Gamma)}(y) = \int_\Gamma (1+T(y))\phi(\psi(y)) \, \mathrm{d}s_{\Gamma}(y),
\end{equation}
where $T \in L^\infty(\Gamma)$ and $\lim\limits_{y \to x} T(y) = 0$. The operator
\[\eta_{\rho}' D_{\psi(\Gamma)} \eta_{\rho}' \colon L^2(\psi(\Gamma)) \to L^2(\psi(\Gamma))\]
is then similar to
\[C_{\psi} \eta_{\rho}' D_{\psi(\Gamma)} \eta_{\rho}' C_{\psi^{-1}} \colon L^2(\Gamma) \to L^2(\Gamma),\]
which, by \eqref{eq:chofvar}, is an integral operator with kernel
$(1+T(y)) \eta_{\rho}(z) D_{\psi(\Gamma)}(\psi(z), \psi(y)) \eta_{\rho}(y)$.
By the previous calculations, if $D_\Gamma \in \mathscr{A}_2$,  we therefore find that
\begin{equation} \label{eq:tildeV}
C_{\psi} \eta_{\rho}' D_{\psi(\Gamma)} \eta_{\rho}' C_{\psi^{-1}} \sim \eta_{\rho} D_\Gamma \eta_{\rho} + \eta_{\rho}\tilde{V}\eta_{\rho},
\end{equation}
where $\tilde{V}$ is an integral operator induced by a kernel of the form $\tilde{V}(z,y) = D_{\Gamma}(z,y)\tilde{Q}(z,y)$ with $\tilde{Q} \in L^\infty(\Gamma \times \Gamma)$ such that
\begin{equation} \label{eq:tildeQvan}
\lim_{\rho \to 0^+} \esssup\limits_{y,z \in \Gamma \cap B(x, \rho)} |\tilde{Q}(z,y)| = 0.
\end{equation}
For $\phi_1,\phi_2 \in L^2(\Gamma)$ with support in $B(x,\rho)$ we have
\[\abs{\sp{\tilde{V}\phi_1}{\phi_2}} \leq \int_{\Gamma} \int_{\Gamma} \abs{\phi_2(z)\tilde{V}(z,y)\phi_1(y)} \, \mathrm{d}s(y) \, \mathrm{d}s(z) \leq \esssup\limits_{y,z \in \Gamma \cap B(x, \rho)} |\tilde{Q}(z,y)|\norm{\abs{D_{\Gamma}}}\norm{\phi_1}\norm{\phi_2}.\]
Since $D_\Gamma \in \mathscr{A}_2$, \eqref{eq:tildeQvan} implies that
\begin{equation} \label{eq:locdiff_limit}
\lim_{\rho \to 0^+} \|\eta_{\rho} \tilde{V}\eta_{\rho}\| = 0.
\end{equation}
Taking powers, we conclude that
\begin{equation} \label{eq:locdiffn}
\lim_{\rho \to 0^+} \|(\eta_{\rho} D_\Gamma \eta_{\rho})^n - C_{\psi} (\eta_{\rho}' D_{\psi(\Gamma)} \eta_{\rho}')^nC_{\psi^{-1}}\|_{\ess} = 0
\end{equation}
for every $n \geq 1$.

We now prove the main theorem of this section. When $R = 1/2$, it says the following: if, for every point $x$, $\Gamma$ is locally obtained by deforming a domain that satisfies the spectral radius conjecture, then $\Gamma$ satisfies the spectral radius conjecture as well.

\begin{thm} \label{thm:essradlocal}
Let $R > 0$ and $\beta > 0$. Suppose that for every $x \in \Gamma$ there exists a $C^{1,\beta}$-diffeomorphism $\psi_x \colon \mathbb{R}^d \to \mathbb{R}^d$, conformal or anti-conformal at $x$, and a Lipschitz cut-off function $\eta_x \from \R^d \to [0,1]$ such that $\eta_x$ is equal to $1$ in a neighbourhood of $x$, $D_{\psi_x(\Gamma)} \in \mathscr{A}_2$, and
\[\sigma_{\ess}(\eta_x' D_{\psi_x(\Gamma)}\eta_x') \subset B(0,R),\]
where $\eta_x' = \eta_x \circ \psi_x^{-1}$. Then $D_\Gamma \in \mathscr{A}_2$ and
$\sigma_{\ess}(D_{\Gamma}) \subset B(0,R)$.
\end{thm}

\begin{proof}
As noted previously, the condition on $\Gamma$, even for just one $x$, implies that $D_\Gamma \in \mathscr{A}_2$.

Fix $x \in \Gamma$. Since the commutator $[D_{\psi_x(\Gamma)}, \eta_x']$ is compact, we see that $(\eta_x' D_{\psi_x(\Gamma)}\eta_x')^n - \eta_x'^n D_{\psi_x(\Gamma)}^n \eta_x'^n$ is compact for every $n \geq 1$. By the hypothesis, choose $n = n(x)$ sufficiently large so that
\[\|(\eta_x' D_{\psi_x(\Gamma)}\eta_x')^n\|_{\ess}^{\frac{1}{n}} = \|\eta_x'^n D_{\psi_x(\Gamma)}^n\eta_x'^n\|_{\ess}^{\frac{1}{n}} < R\|C_{\psi_x}\|^{-\frac{1}{n}}\|C_{\psi_x^{-1}}\|^{-\frac{1}{n}}.\]
If we choose a new cut-off function $\tilde{\eta}_x \from \R^d \to [0,1]$ such that $\supp \tilde{\eta}_x \subset  \set{y \in \R^d : \eta_x(y) = 1}$, note that, since $(\eta_x' D_{\psi_x(\Gamma)}\eta_x')^n - \eta_x'^n D_{\psi_x(\Gamma)}^n \eta_x'^n$ is compact, we have
\[\|(\tilde{\eta}_x' D_{\psi_x(\Gamma)}\tilde{\eta}_x')^n\|_{\ess}^{\frac{1}{n}} = \|\tilde{\eta}_x'^n D_{\psi_x(\Gamma)}^n\tilde{\eta}_x'^n\|_{\ess}^{\frac{1}{n}} \leq \|\eta_x'^n D_{\psi_x(\Gamma)}^n\eta_x'^n\|_{\ess}^{\frac{1}{n}} < R\|C_{\psi_x}\|^{-\frac{1}{n}}\|C_{\psi_x^{-1}}\|^{-\frac{1}{n}},\]
so that
\[\|C_{\psi_x} (\tilde{\eta}_x' D_{\psi_x(\Gamma)} \tilde{\eta}_x' )^nC_{\psi_x^{-1}}\|_{\ess}^{\frac{1}{n}} \leq \|\eta_x'^n D_{\psi_x(\Gamma)}^n\eta_x'^n\|_{\ess}^{\frac{1}{n}}\|C_{\psi_x}\|^{\frac{1}{n}}\|C_{\psi_x^{-1}}\|^{\frac{1}{n}} < R.\]
Hence, by making the support of $\eta_x$ sufficiently small and applying \eqref{eq:locdiffn}, we obtain that
\[\|(\eta_x D_\Gamma \eta_x)^n\|_{\ess}^{\frac{1}{n}} =  \|\eta_x^n D_\Gamma^n \eta_x^n\|_{\ess}^{\frac{1}{n}} < R.\]
Now choose $\delta(x)$ sufficiently small so that $\eta_x \equiv 1$ in a neighbourhood of $B(x, \delta(x)) \cap \Gamma$. By compactness, we may pick points $x_1, \ldots, x_N$ so that
$\Gamma \subset  \cup_{j=1}^N B(x_j, \delta(x_j))$.
Let
\[r := \max_{j = 1,\ldots, N} \|(\eta_{x_j} D_\Gamma \eta_{x_j})^{n(x_j)}\|_{\ess}^{\frac{1}{n(x_j)}} < R\]
and let $\set{\tau_j : j = 1,\ldots, N}$ be a Lipschitz partition of unity of $\Gamma$ subordinate to the sets $B(x_j, \delta(x_j))$. Choose $m$ to be a large integer such that $n(x_j)$ divides $m$ for all $j$. Then
\[D^m_\Gamma = \sum\limits_{j = 1}^N \tau_j D_\Gamma^m = \sum \tau_j \eta_{x_j}^{2m} D_\Gamma^m\]
and therefore, recalling that $[D_\Gamma, \eta_{x_j}]$ is compact,
\begin{align*}
\|D_\Gamma^m\|_{\ess} &= \Bigg\|\sum\limits_{j = 1}^N \tau_j (\eta_{x_j} D_\Gamma \eta_{x_j})^m\Bigg\|_{\ess} \leq \sum\limits_{j = 1}^N \norm{(\eta_{x_j} D_\Gamma \eta_{x_j})^m}_{\ess} \leq \sum\limits_{j = 1}^N \norm{(\eta_{x_j} D_\Gamma \eta_{x_j})^{n(x_j)}}^{\frac{m}{n(x_j)}}_{\ess}\\
&\leq Nr^m < N R^m.
\end{align*}
Therefore $\sigma_{\ess}(D_{\Gamma}) \subset B(0,R)$.
\end{proof}

Referring back to Lemma~\ref{lem:A2graphs}, all graphs $\widetilde{\Gamma}$ of type (i)-(iii) satisfy (the graph version of) the spectral radius conjecture, results which can also be found in the corresponding references.  Graphs of type (iv) that are convex also satisfy the spectral radius conjecture, by \cite{FaSaSe:92}.  If $\eta \colon \mathbb{R}^3 \to [0,1]$ is a compactly supported Lipschitz function, we thus have that
	$$\| (\eta D_{\widetilde{\Gamma}} \eta)^n \|_{\ess} = \| \eta^n D_{\widetilde{\Gamma}}^n \eta^n \|_{\ess} \leq \|D_{\widetilde{\Gamma}}^n\|_{\ess} < 2^{-n}$$
	for $n$ sufficiently large. Applying Theorem~\ref{thm:essradlocal} we obtain the following.
\begin{cor} \label{cor:CP}
	Let $d = 3$ and let $\Gamma$ be the boundary of a Lipschitz domain. Assume that for every $x \in \Gamma$ there exists a $C^{1,\beta}$-diffeomorphism $\psi_x \colon \mathbb{R}^3 \to \mathbb{R}^3$, conformal or anti-conformal at $x$, such that $\psi_x(\Gamma)$ is locally a subset of the boundary of a polyhedral cone or a convex smooth cone. Then $D_\Gamma$ satisfies the spectral radius conjecture.
\end{cor}
The class of domains considered in Corollary~\ref{cor:CP} encompasses all domains in 3D that may reasonably be referred to as a Lipschitz curvilinear polyhedra. From Theorem~\ref{thm:essradlocal} we of course also obtain the analogous corollary for 2D domains, a well known result;  the corresponding class of curves precisely describes the $C^{1,\beta}$-curvilinear polygons in 2D.

\subsection{Possible extensions and further questions}

Theorem~\ref{thm:essradlocal} in particular shows that the spectral radius conjecture only depends on the local behaviour of $\Gamma$, under the extraneous assumption that $D_\Gamma \in \mathscr{A}_2$. To achieve the same result for a general Lipschitz domain, we need a different tool to treat the term $V_3$ which does not rely on the introduction of absolute values. Since $V_3(z,y) = \widetilde{Q}(z,y) D_\Gamma(z,y)$, for a specific kernel $\widetilde{Q}(z,y)$, we would like an answer to the following question.
\begin{quest} \label{q:1}
Describe suitable classes of kernels $B$ such that the kernel $B(z,y) D_\Gamma(z,y)$ defines a bounded operator on $L^2(\Gamma)$, and estimate the norm of the corresponding operator.
\end{quest}
When $B$ is of the form $B(z,y) = a(z) - a(y)$, for some function $a$ on $\Gamma$, the resulting operator is the commutator $[a, D_\Gamma]$. Commutators of singular integral operators are very well studied, but we have been unable to identify or correctly apply existing results to the kernel multiplier of interest to us, namely,
\[B(z,y) = \eta_\rho(z) \left(\frac{|z-y|^d}{|\psi(z) - \psi(y)|^d} - 1\right) \eta_\rho(y).\]
Other versions of Question~\ref{q:1} also seem interesting. For example, one could ask for the stronger property that $B(z,y)$ be a kernel multiplier of each of the Riesz transforms of $\Gamma$. We also note the similarity between the term $V_3$ and the kernel of the Clifford--Cauchy integral operator, as presented in \cite[Consequence 3.6]{AxKeMc:06}.

Of course, one would like to know that not only the spectral radius conjecture is local, but that the entire essential spectrum is as well.
\begin{quest} \label{q:2}
For a domain $\Omega_-$ with boundary $\Gamma$, is the essential spectrum local and invariant under locally conformal deformations?
\end{quest}

We can give a positive answer to this question if we, in addition to the hypotheses of Theorem~\ref{thm:essradlocal}, assume that we do not have too many singular points, by which we mean points where the boundary is not $C^1$. In the statement, we let $\chi_{x,\rho} := \chi_{B(x,\rho)}$ and $\chi_{x,\rho}' := \chi_{x,\rho} \circ \psi_x^{-1}$ for $x \in \Gamma$ and $\rho > 0$.

\begin{thm}
Let $\beta > 0$ and assume that $\Gamma$ has at most countably many singular points. Further suppose that for every $x \in \Gamma$ there exists a $C^{1,\beta}$-diffeomorphism $\psi_x \colon \mathbb{R}^d \to \mathbb{R}^d$, conformal or anti-conformal at $x$, such that $D_{\psi_x(\Gamma)} \in \mathscr{A}_2$. Then $D_\Gamma \in \mathscr{A}_2$ and
\[\sigma_{\ess}(D_{\Gamma}) \subset  \bigcup\limits_{x \in \Gamma} \sigma_{\ess}(\chi_{x,\rho_x}'D_{\psi_x(\Gamma)}\chi_{x,\rho_x}')\]
for arbitrary $\rho_x > 0$. Moreover, there exist $\rho_x > 0$ such that
\[\sigma_{\ess}(D_{\Gamma}) = \bigcup\limits_{x \in \Gamma} \sigma_{\ess}(\chi_{x,\rho_x}'D_{\psi_x(\Gamma)}\chi_{x,\rho_x}').\]
\end{thm}

\begin{proof}
By \eqref{eq:A2}, $D_{\psi_x(\Gamma)} \in \mathscr{A}_2$ for just one $x \in \Gamma$ already implies $D_\Gamma \in \mathscr{A}_2$.

Let $x \in \Gamma$, $\rho_x > 0$ and $\lambda \in \C$. If $\lambda = 0$, then $0 \in \sigma_{\ess}(\chi_{x,\rho_x}'D_{\psi_x(\Gamma)}\chi_{x,\rho_x}')$ either by \eqref{eq:notFredholm} (if $\chi_{x,\rho_x}' \equiv 1$) or by construction (otherwise). So assume that $\lambda \neq 0$ and suppose that $\chi_{x,\rho_x}'D_{\psi_x(\Gamma)}\chi_{x,\rho_x}' - \lambda$ is Fredholm for every $x \in \Gamma$.
	
Fix $x$ for a moment. As $\Gamma$ has at most countably many singular points and $\psi_x$ is a diffeomorphism, we know that $\chi_{x,\tilde{\rho}_x}'D_{\psi_x(\Gamma)}\chi_{x,\tilde{\rho}_x}' - \lambda$ is Fredholm for all but at most countably many $\tilde{\rho}_x \in (0,\rho_x]$. In this case the essential norm of the Fredholm regularizer $A_{x,\tilde{\rho}_x}$ of $\chi_{x,\tilde{\rho}_x}'D_{\psi_x(\Gamma)}\chi_{x,\tilde{\rho}_x}' - \lambda $ is bounded by
\[\norm{A_{x,\tilde{\rho_x}}}_{\ess} \leq \max\set{\norm{A_{x,\rho_x}}_{\ess},\abs{\lambda}^{-1}}.\]
Therefore, by \eqref{eq:tildeV} and \eqref{eq:locdiff_limit}, for every $x \in \Gamma$ we can choose $\tilde{\rho}_x$ sufficiently small such that $\chi_{x,\tilde{\rho}_x}D_{\Gamma}\chi_{x,\tilde{\rho}_x} - \lambda$ is also Fredholm, and such that $\partial B(x, \tilde{\rho}_x)$ contains none of the singularities of $\Gamma$.

By compactness, we may choose finitely many points $x_j$ and corresponding radii $\tilde{\rho}_{x_j}$ such that the balls $B(x_j,\tilde{\rho}_{x_j})$ cover $\Gamma$. Define recursively
\[E_1 := B(x_1,\tilde{\rho}_{x_1}) \cap \Gamma, \quad E_{j+1} = \big(B(x_{j+1},\tilde{\rho}_{x_{j+1}}) \cap \Gamma\big) \setminus \bigcup\limits_{k = 1}^j E_k.\]
Now we can apply \eqref{eq:ess_decomposition} to obtain that $D_{\Gamma} - \lambda$ is Fredholm.

Conversely, assume that $D_{\Gamma} - \lambda$ is Fredholm and fix $x \in \Gamma$. By \eqref{eq:notFredholm}, we must have $\lambda \neq 0$. As we only have countably many singularities, $\chi_{x,\rho}D_{\Gamma}\chi_{x,\rho} - \lambda$ is also Fredholm for all but at most countably many $\rho > 0$ by \eqref{eq:ess_decomposition}. Now note that the situation between $D_{\Gamma}$ and $D_{\psi_x(\Gamma)}$ is symmetric. Hence, the same argument as above shows that $\chi_{x,\rho_x}'D_{\psi_x(\Gamma)}\chi_{x,\rho_x}' - \lambda$ is also Fredholm for a sufficiently small $\rho_x > 0$.
\end{proof}

In particular, we may extend Theorem \ref{thm:localization} to domains that are only \emph{approximately} locally dilation invariant in the following sense.

\begin{cor} \label{cor:final}
Let $\beta > 0$ and assume that $\Gamma$ has at most countably many singular points. Denote the set of singular points by $\mathcal{S}$. Suppose that for every $x \in \mathcal{S}$ there exists a $C^{1,\beta}$-diffeomorphism $\psi_x \colon \mathbb{R}^d \to \mathbb{R}^d$ such that
\begin{itemize}
	\item[(i)] $\psi_x$ is conformal or anti-conformal at $x$,
	\item[(ii)] $\psi_x(\Gamma)$ is locally $C^{1,\beta}$ or locally dilation invariant at $\psi_x(x)$,
	\item[(iii)] $D_{\psi_x(\Gamma)} \in \mathscr{A}_2$ (e.g., a domain with local behaviour as in Theorem \ref{thm:locals}).
\end{itemize}
Then there exist $\rho_x > 0$ such that
$\sigma_{\ess}(D_{\Gamma}) = \set{0} \cup \bigcup\limits_{x \in \mathcal{S}} \sigma_{\ess}(\chi_{x,\rho_x}'D_{\psi_x(\Gamma)}\chi_{x,\rho_x}')$.
\end{cor}

\section{Synthesis and numerical examples} \label{sec:num_ex}
In this final section we bring earlier results together to study the case where $\Gamma\in \cD_\cA$ (recall Definition \ref{defn:ldiA}), meaning that $\Gamma$ is the boundary of a bounded Lipschitz domain $\Omega_-\subset \R^2$ that is locally dilation invariant and is piecewise analytic. We then illustrate the various results of the paper by several examples. Regarding our first aim, the following result is immediate from Theorems \ref{thm:localization}, \ref{thm:Hausdorff_convergence_2}, and \ref{thm:spectral_radius2}.

\begin{thm} \label{thm:synth} Suppose that $\Gamma \in \cD_\cA$ and let $F\subset \Gamma$ be the finite set of points at which $\Gamma$ is not locally analytic but {\em is} locally dilation invariant, so that $\Gamma$ coincides locally near $x\in F$ with $\Gamma_x$, a dilation invariant graph (in some rotated coordinate system centred at $x$). Then, where $\sigma^N$ is as defined in Theorem \ref{thm:Hausdorff_convergence_2}, $\Sigma^N(D_\Gamma) := \cup_{x\in F} \sigma^N(D_{\Gamma_x}) \toH \sigma_{\ess}(D_\Gamma;L^2(\Gamma))$ as $N\to\infty$. Further, suppose that $c>0$ is small enough such that the conditions of the first sentence of Theorem \ref{thm:spectral_radius2} are satisfied for $\Gamma= \Gamma_x$, $x\in F$. Then, for every $\rho_0>0$, where $R_N(D_\Gamma)$ is defined by \eqref{eq:RN}, $\rho_{\ess}(D_\Gamma;L^2(\Gamma))< \rho_0$ if: i) $R_N(D_\Gamma)< \rho_0$; and ii) for  some $m,M,N\in \N$, with $M\geq 2$,
\begin{equation} \label{eq:cScGenDef}
\cS_c(\Gamma,m,M,N) := \max_{x\in F} S_c(\Gamma_x,m,M,N) < 0,
\end{equation}
where $S_c(\cdot,\cdot,\cdot,\cdot)$ is as defined in Remark \ref{rem:final}. Conversely, if  $\rho_{\ess}(D_\Gamma;L^2(\Gamma))< \rho_0$, then, for all sufficiently large $N$, $R_N(D_\Gamma)<\rho_0$, and, if $m, M\in \N$ are also sufficiently large,  then  $\cS_c(\Gamma,m,M,N)<0$.
\end{thm}
\begin{proof} Since $\Gamma\in \cD_\cA\subset \cD$, $\Gamma$ is locally $C^1$ at $x$ if $x\in \Gamma\setminus F$, while, if $x\in F$, $\Gamma_x$ satisfies \eqref{eq:standing2} (in some local coordinate system centred at $x$). Thus, and by Theorem \ref{thm:localization}, $\sigma_{\ess}(D_\Gamma)= \cup_{x\in \Gamma} \sigma_{\ess}(D_{\Gamma_x}) = \cup_{x\in F}\sigma_{\ess}(D_{\Gamma_x})$, since  $\sigma_{\ess}(D_{\Gamma_x})=\{0\}$ if $x\in \Gamma\setminus F$,  and $0\in \sigma_{\ess}(D_{\Gamma_x})$ if $x\in F$, by Theorem \ref{cor:norm_and_spec} and Corollary \ref{cor:K_t_compact}. The result thus follows from Theorems  \ref{thm:Hausdorff_convergence_2} and \ref{thm:spectral_radius2}.
\end{proof}

The following examples illustrate the above result and the results of \S\ref{sec:dilation_invariant} and \S\ref{sec:LD}. In each example, whether $\Gamma$ is a dilation invariant graph or $\Gamma\in \cD_\cA$, we demonstrate that $\rho_{\ess}(D_\Gamma;L^2(\Gamma))<\half$, providing new evidence in support of the spectral radius conjecture.\\

\noindent {\bf Example 1.} \label{ex1} We first consider an example  where $\Gamma$ satisfies \eqref{eq:standing}, with $f:\R_+\to\R$ given by
\begin{equation} \label{eq:fdef}
f(x) := x\sin^2(\pi\log_{\alpha}(x)), \quad x>0,
\end{equation}
for some $\alpha\in (0,1)$, so that (recall \eqref{eq:gdef}) $g(x) = \sin^2(\pi x) = (1-\cos(2\pi x))/2$, for $x\in \R$. As $\alpha$ increases in $(0,1)$ the graph of $f$ (see Figure \ref{fig:ex1}) becomes increasingly oscillatory and its Lipschitz character increases; elementary calculations give that the maximum and minimum of $f^\prime$ are $f^\prime_{\max}=\cos^2(\theta_\alpha/2) +\pi\sin(\theta_\alpha)/|\log\alpha|>1$ and $f^\prime_{\min}=1-f^\prime_{\max}<0$, where $\theta_\alpha := \arctan(2\pi/|\log\alpha|)$; note that $f^\prime_{\max} = \pi/|\log \alpha| + 1/2 + O(1-\alpha)$ as $\alpha\to 1^-$.
To apply Theorem \ref{thm:spectral_radius} we need to compute $C_1(M)$, $C_3$, and $C_4$; see \eqref{eq:C1def}, \eqref{eq:NormBound}, and the definition below \eqref{eq:boundcombined}. This requires computation of quantities that are defined in Proposition \ref{prop:k_t_estimate general} in terms of\footnote{In this example, and the other examples below, we are able to compute these norms exactly. We can, instead, just compute upper bounds; the theory and algorithm apply essentially unchanged -- see Remark \ref{rem:final} iii).}
\begin{gather*}
\|\Im g^{(j)}\|_c = \frac{(2\pi)^j}{2} \sinh(2\pi c), \;\;j = 0,1, \;\; c\geq 0,\\
 \|g\|_c = \cosh^2(\pi c), \;\; \|g^{(j)}\|_c = \frac{(2\pi)^j}{2} \cosh(2\pi c), \;\; j = 1,2,\;\; c\geq 0.
\end{gather*}
Theorem \ref{thm:spectral_radius} applies for all $c>0$ such that \eqref{c:bound2} applies, i.e.\ provided
\begin{equation} \label{eq:c_conds}
c \leq \frac{\arccos(\alpha)}{\abs{\log\alpha}} \quad \text{and} \quad c < \frac{1}{2\pi}\arsinh\left(\frac{2|\log\alpha|}{2\pi + |\log \alpha|}\right).
\end{equation}

 \begin{figure}[ht]
\centering
\begin{subfigure}[t]{.42\textwidth}
  \centering
  \includegraphics[scale = 0.48, trim = 2cm 0cm 0cm 0cm]{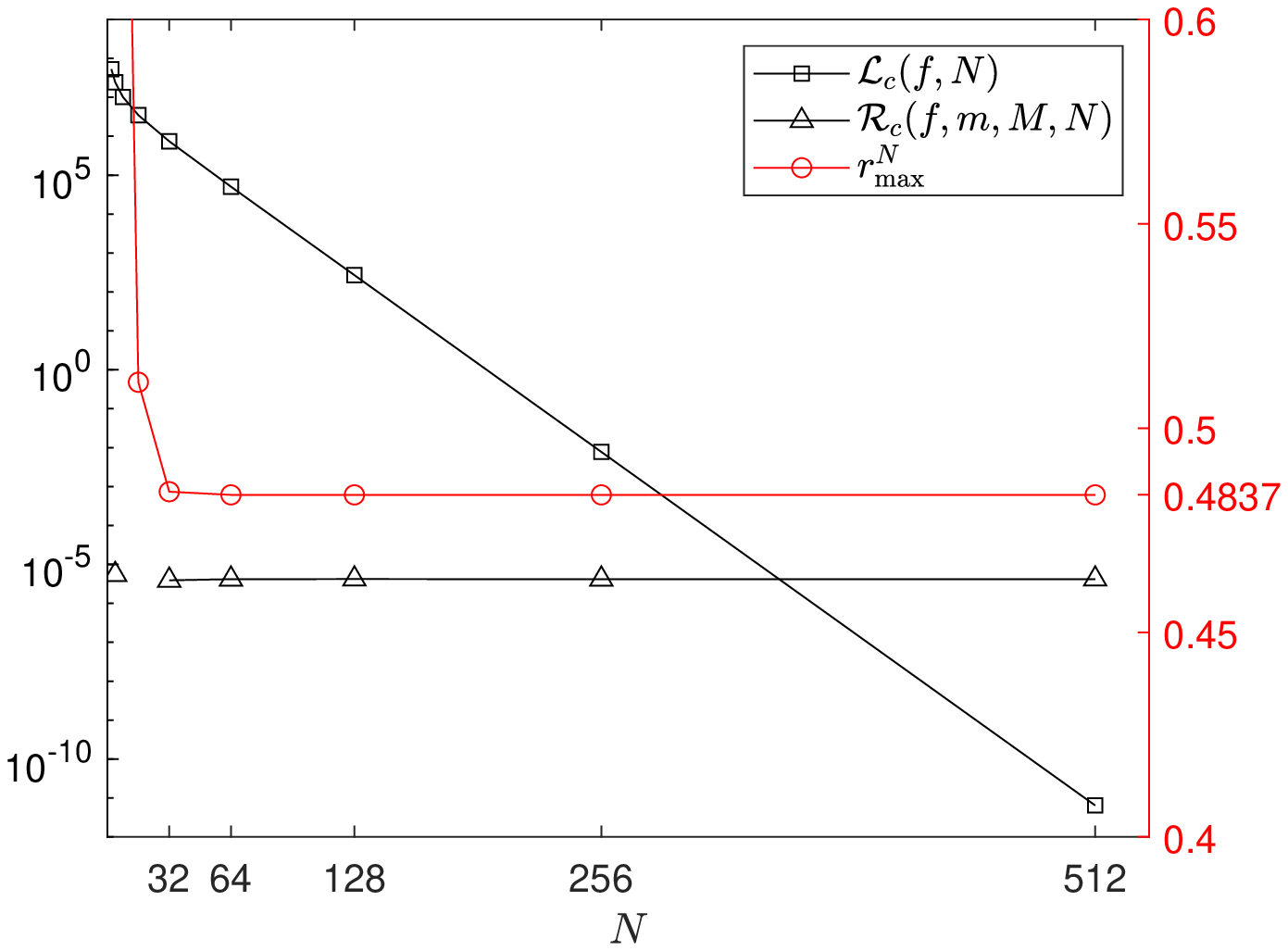}
  \caption{$\cL_c(f,N)$, $\cR_c(f,m,M,N)$, for $\rho_0=\half$, and the maximum of the spectral radii of $A_{t_k,N}^M$, $k=1,\ldots,m$, plotted against $N$ for $m=16,000$ and $M=100$. The right-hand scale (in red) is to be used for $r_{\max}^N$ (plotted in red). The $\cR_c(f,m,M,N)$ value is not shown for $N=16$ because it is negative.}
\end{subfigure}
\hspace{0.5cm}
\begin{subfigure}[t]{.53\textwidth}
  \centering
  \includegraphics[scale = 0.53, trim = 2cm 0cm 0cm 0cm]{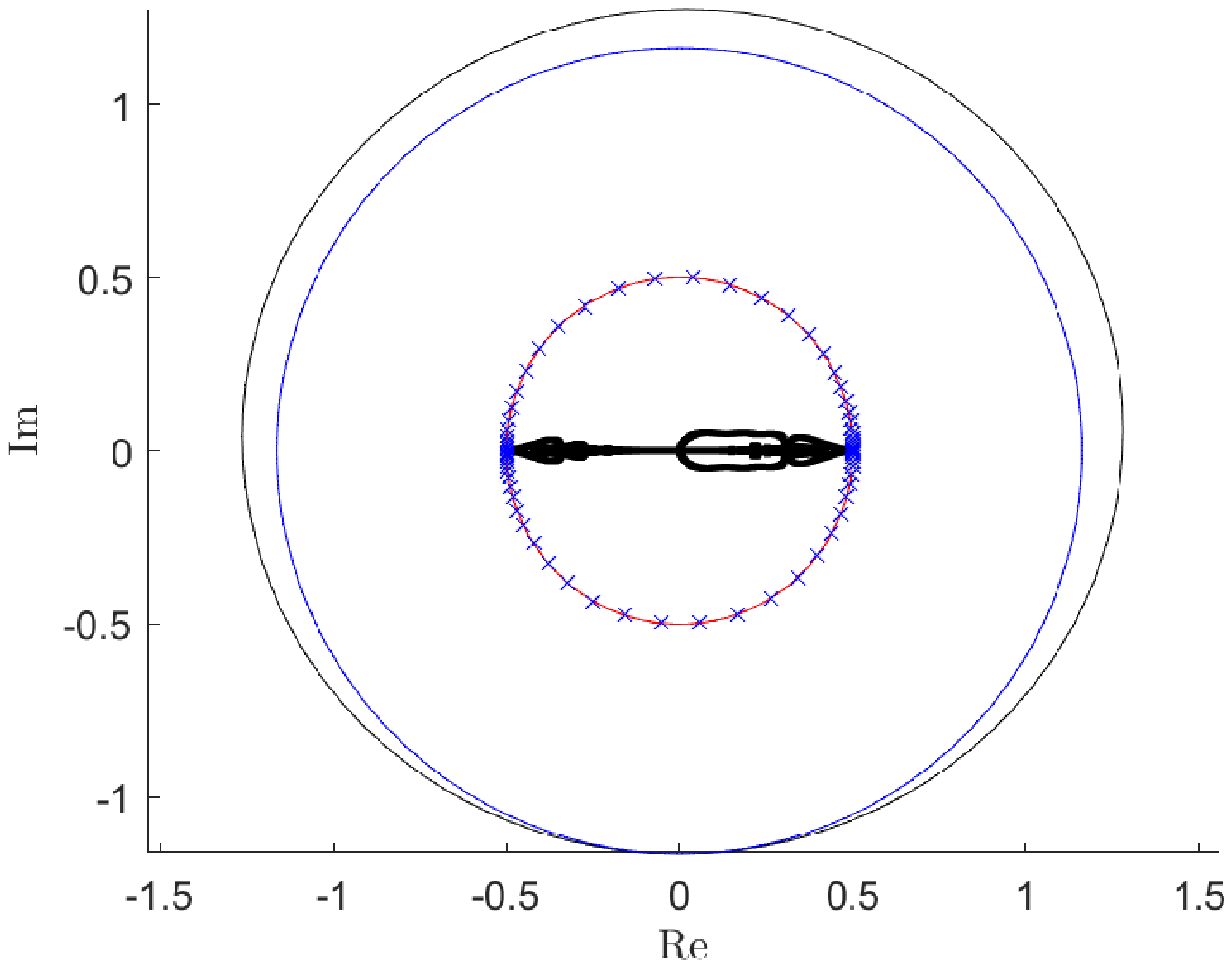}
  \caption{$\sigma^N(D_\Gamma)=\{0\}\cup \bigcup_{t\in t_N} \spec(A_{t,N}^{M_N})$, for $N = 512$, $m_N = 16,000$, $M_N=100$, an approximation to $\sigma(D_\Gamma)=\sigma_{\ess}(D_\Gamma)$ by Theorem \ref{thm:Hausdorff_convergence}. Also shown are the boundary of $W_{100}$ (for $p = 10$, $N = 512$, $M = 100$ and $t = \pi/18$) in black, $R^*\T \subset W_{100}\subset W(D_{\Gamma})$ in blue (see Corollary \ref{cor:num_range2} for definitions), $\half\T$ in red. The crosses on $\half\T$ are $\mu_{k,\ell}$ for $k = 16,000$ and $\ell = 1,\ldots,n_k = 67$ (see Theorem \ref{thm:spectral_radius})}
\end{subfigure}
\caption{Numerical results for Example 1; $f$ given by \eqref{eq:fdef}, $\alpha=\frac{3}{4}$} \label{fig:plot2}
\end{figure}

In Figure \ref{fig:plot2} we plot results for the case $\alpha=\frac{3}{4}$ (see Figure \ref{fig:ex1}), when $f^\prime_{\max} \approx 11.43$ and the above conditions reduce to $c \leq \arccos(3/4)/\log(4/3) \approx 2.51$ and  $c < \arsinh\left(2 \log(4/3)/(2\pi + \log(4/3))\right)/(2\pi)$ $\approx 0.0139$. Choosing $c=0.013$ we plot in Figure \ref{fig:plot2}(a),  for the case $\rho_0=\half$, $\cL_c(f,N)$ and $\cR_c(f,m,M,N)$, given by \eqref{eq:estimate_to_check}, against $N$, for $N=2^j$, $j=3,4,\ldots,9$, choosing $m=16,000$ and $M=100$. We see that with these choices $\cR_c(f,m,M,N)$ is positive and bounded away from zero for sufficiently large $N$, while, as is clear from its definition, $\cL_c(f,N)$ decreases exponentially with $N$; note that \eqref{eq:estimate_to_check} is satisfied for $N=512$. To apply Theorem \ref{thm:spectral_radius} to conclude that $\rho(D_\Gamma)=\rho_{\ess}(D_\Gamma)<\half$ we also need to check that $\rho(A^M_{t,N})<\half$, for $M=100$, N = 512, and $t = t_k=(k-1/2)/m$, $k=1,\ldots, m$, with $m=16,000$; equivalently, that $r^N_{\max} := \max_{z\in \sigma^N(D_\Gamma)}|z|<\half$, where $\sigma^N(D_\Gamma)$ is as defined by \eqref{eq:sigmaNDef} with $N=512$, $m_N=16,000$, $M_N=100$.  For these parameter values  $\sigma^N(D_\Gamma)$, an approximation to $\sigma(D_\Gamma)=\sigma_{\ess}(D_\Gamma)$  by Theorem \ref{thm:Hausdorff_convergence}, is plotted in
Figure \ref{fig:plot2}(b), and  $r^{N}_{\max} \approx 0.4837 < \half$. Our calculations are in standard double-precision floating-point arithmetic rather than exact arithmetic\footnote{In all these examples the spectral radii of the matrices $A_{t,N}^M$ as well as the norms of the $(A_{t,N}^M - \lambda_l I)^{-1}$ are computed using standard Matlab routines; our codes are available at \url{https://github.com/Raffael-Hagger/DLP}}, but this is convincing evidence, by application of Theorem \ref{thm:spectral_radius}, that $\rho(D_\Gamma)<\half$.

 In Figure \ref{fig:plot2}(b) we also plot, for the parameter values $p=10$, $N=512$, $M=100$, $n=100$, and $t=\pi/18$, the bounded domain $W_n$, which\footnote{In our computations of $W_n$ we neglect the factor $e^{i(x_{m,N}-x_{n,N})t}$ when using \eqref{eq:TpNdef} with $\tilde K_t$ replaced by $\tilde K_t^M$. The resulting matrix is unitarily equivalent to $T^{p,t,N,M}$ as defined in \S\ref{sec:num_range}, so has the same numerical range.} is contained in and is an approximation to $W(T^{p,t,N,M})$ (these notations defined in \S\ref{sec:num_range}).
 Where $R^*$ is as defined in Corollary \ref{cor:num_range2}, we also plot the circle $R^* \T$ of radius $R^*\approx 1.163$; by Corollary \ref{cor:num_range2} it is guaranteed that $\overline{B_{R^*}}\subset W(D_\Gamma)=W_{\ess}(D_\Gamma)$, so that $w(D_\Gamma)=w_{\ess}(D_\Gamma)\geq R^*$ is significantly larger than $\half$ for this example. We note that, for these parameter values, $C_7(p,N,M) \leq 3.964 \times 10^{-5}$ (see Corollaries \ref{cor:num_range} and \ref{cor:num_range2}).

 In Figure \ref{fig:plot2}(b) we additionally plot, to illustrate the application of Theorem \ref{thm:spectral_radius} and the adaptive definition of the parameters defined by \eqref{eq:lambda_{k,l}}, the points $\mu_{k,\ell}$ at which the resolvent of $A_{t_k,N}^M$ is calculated when $\rho_0=\half$, $N=512$, $M=100$, $m=16,000$, for the case $k=16,000$ (so that $t_k\approx \pi$) and $\ell=1,\ldots,n_k=67$. This value of $k$ is fairly typical; $n_k$ varies in the range  $[67,110]$ as a function of $k$. It is clear that the adaptive algorithm of Theorem \ref{thm:spectral_radius} (and see Corollary \ref{cor:GspecRMfinal}) is significantly more efficient than the uniform grid on $\rho_0 \T$ of Corollary \ref{cor:uni_space} (recall the discussion above Lemma \ref{lem:R*def}).   For $k = 16,000$  we have $\min_\ell \nu_{k,\ell} \approx 0.0101$ and $\max_\ell \nu_{k,\ell} \approx 0.2216$ (see \eqref{eq:lambda_{k,l}} for this notation). The ratio of these maximum and minimum values, which is approximately the ratio of the maximum to the minimum spacing of the points $\mu_{k,\ell}$, is about $20.2$.\\

\noindent{\bf Example 2.} We now turn to examples where we can apply the theory of \S\ref{sec:two-sided}. First we consider the case of a cone, that is $\Gamma=\{(x,f(x)):x\in \R\}$ where
\begin{equation} \label{eq:fcone}
f(x) := \mu\abs{x},  \quad x\in \R,
\end{equation}
for some $\mu \in \R$. This clearly satisfies \eqref{eq:standing2}, for any $\alpha\in (0,1)$, and, where $f_\pm$ and $g_\pm$ are defined by \eqref{eq:gdef2}, $f_\pm(x) = \mu x$, $g_\pm(x)=\mu$, for $x>0$.
For this example the spectrum and spectral radius are known, viz.
\begin{equation} \label{eq:SpecCone}
\sigma_{\ess}(\Gamma)=\sigma(D_\Gamma) = \{0\}\cup \left\{\pm \frac{\sin(\arctan(|\mu|)(1-iy))}{2\sin(\pi(1-iy)/2)}: y\in \R\right\} \;\; \mbox{so that} \;\; \rho(D_\Gamma) = \frac{|\mu|}{2\sqrt{1+\mu^2}}
\end{equation}
 (see, e.g., \cite{Mi:02})
 and, since $D_\Gamma$ is diagonalised by the Mellin transform when $\Gamma$ is a 2D cone, $D_\Gamma$ is normal so that $W_{\ess}(D_\Gamma)=\overline{W(D_\Gamma)}=\conv(\sigma(D_\Gamma))$ and $\|D_\Gamma\|=w(D_\Gamma)=\rho(D_\Gamma)$.

 To make comparison of these known results with the methods of \S\ref{sec:two-sided} we can choose, in principle, any $\alpha\in (0,1)$, but the choice of $\alpha$ affects the choice of $c$ via \eqref{c:bound2_2}, and thus the rate of decrease with $N$ of $\cL_c(f,N)$, defined by \eqref{eq:estimate_to_check2}.
The operators $\tilde{L}^{\pm}_t$, given by \eqref{eq:L+}, and thus the  matrices $A_{t_k,N}^M$, given by \eqref{eq:AtNMdef2},  also depend on $\alpha$, while the operators $\tilde{K}^{\pm}_t$, given by \eqref{eq:ktall2}, vanish in this case. It is worth noting that the blocks of $A_{t_k,N}^M$ corresponding to $\tilde L_t^\pm$ are Toeplitz matrices in this example.
The conditions \eqref{c:bound2_2} reduce to $c\leq \arccos(\alpha)/|\log \alpha|$ and $|\mu|c\leq \alpha^2/|\log \alpha|$. As the right hand sides of these inequalities are increasing on $(0,1)$, it is beneficial to choose $\alpha$ closer to $1$ in order to be able to choose a larger $c$. We select $\alpha = \frac{7}{8}$, so that
$\arccos(\alpha)/|\log\alpha| \approx 3.78$ and $\alpha^2/|\log\alpha| \approx 5.73$, and then take $c=0.57$,
so that \eqref{c:bound2_2} is satisfied for $|\mu|\leq 10$. For $\mu =10$ we see from Figure \ref{fig:plot3}(a) that, with $M = 200$, $m = 2,000$ and $N = 16$, the inequality \eqref{eq:estimate_to_check2} is satisfied and $\max_{k=1,\ldots,m} \rho(A_{t_k,N}^M) \approx 0.4975 <\half$, so that $\rho(D_\Gamma)<\half$ by  Theorem \ref{thm:spectral_radius2}, in agreement with \eqref{eq:SpecCone} which gives $\rho(D_\Gamma) =5/\sqrt{101}\approx 0.4975$. We see in Figure \ref{fig:plot3}(b) that the approximations $\sigma^N(D_\Gamma)$ to $\sigma(D_\Gamma)$, given by Theorem \ref{thm:Hausdorff_convergence_2}, agree closely,
for the given parameter values, with the expected lemniscates given by \eqref{eq:SpecCone} for different values of $\mu$.\\

\begin{figure}[ht]
\centering
\begin{subfigure}[t]{.46\textwidth}
  \centering
  \includegraphics[scale = 0.5, trim = 2cm 0cm 0cm 0cm]{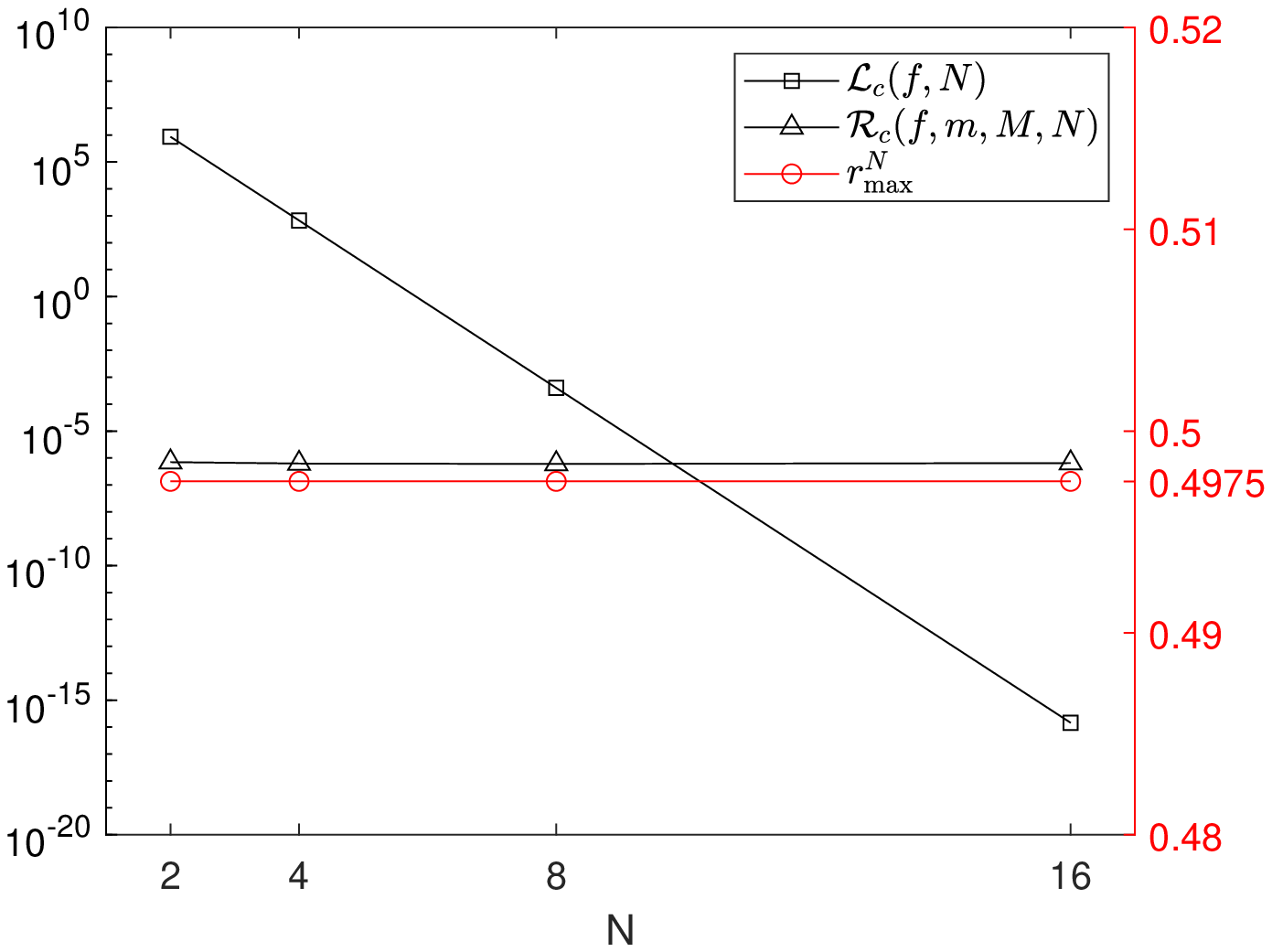}
  \subcaption{$\cL_c(f,N)$, $\cR_c(f,m,M,N)$, for $\rho_0=\half$, and the maximum of the spectral radii of $A_{t_k,N}^M$, $k=1,\ldots,m$, plotted against $N$ for $m=2,000$ and $M=200$ in the case $\mu = 10$.}
\end{subfigure}
\hspace{0.5cm}
\begin{subfigure}[t]{.49\textwidth}
  \centering
  \includegraphics[scale = 0.52, trim = 2cm 0cm 0cm 0cm]{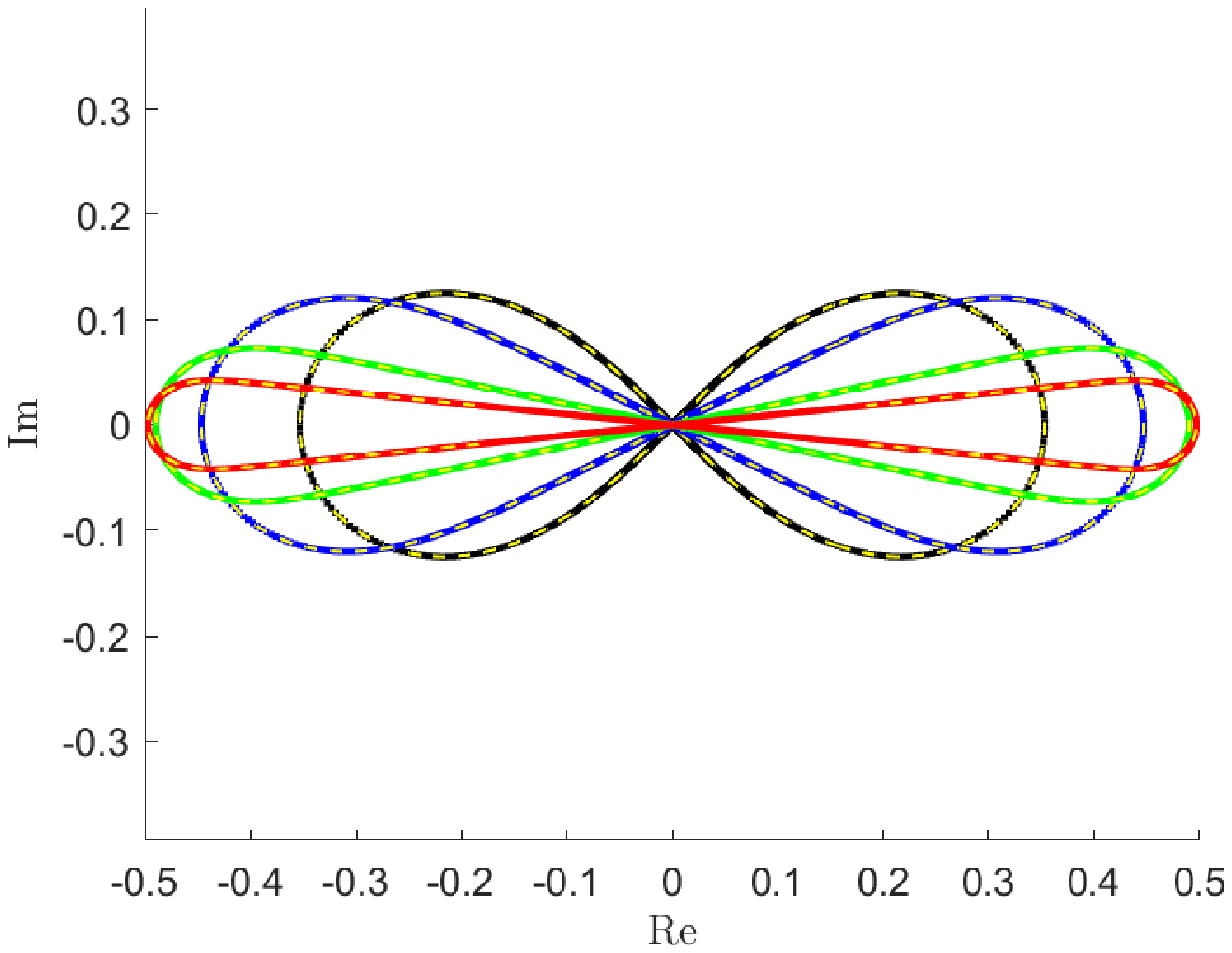}
  \subcaption{$\sigma^N(D_\Gamma) = \{0\}\cup \bigcup_{t\in T_N} \spec(A_{t,N}^M)$ for $N = 16$, $m_N = 2,000$, $M_N = 200$, and $\mu = 1$ (black), $\mu = 2$ (blue), $\mu = 5$ (green), $\mu = 10$ (red). The yellow dashed lines indicate the lemniscates \eqref{eq:SpecCone}, i.e. the (exact) spectrum of $D_{\Gamma}$ in each case.}
\end{subfigure}
\caption{Numerical results for Example 2; $f$ given by \eqref{eq:fcone}, $\alpha=\frac{7}{8}$} \label{fig:plot3}
\end{figure}

\noindent{\bf Example 3.} In this example we take
\begin{equation} \label{eq:ftwoside}
f(x) := \abs{x}\sin^2(\pi\log_{\alpha}\abs{x}), \quad x\in \R,
\end{equation}
and $\Gamma$ satisfies \eqref{eq:standing2}, but now just for one $\alpha\in (0,1)$. Note that $f_\pm$ and $g_\pm$, given by \eqref{eq:gdef2}, are the same as $f$ and $g$ in Example 1, so that $g_\pm$ satisfy the same bounds as $g$ in Example 1, and $\Gamma$ has Lipschitz character $f^\prime_{\max}$ as defined in that example. Thus the conditions \eqref{c:bound2_2} reduce in this example to \eqref{eq:c_conds} plus the condition that $\sinh(2\pi c)/2 + c|\log\alpha| \cosh^2(\pi c) < \alpha^2$. We choose $\alpha=\frac{2}{3}$ (see Figure \ref{fig:ex2}), for which $c=0.019$ satisfies this condition and \eqref{c:bound2_2}, and $f_{\max}^\prime \approx 8.26$. We see from Figure \ref{fig:plot5}(a) that, with $M = 60$, $m = 10,000$ and $N = 256$, the inequality \eqref{eq:estimate_to_check2} is satisfied and $\max_{k=1,\ldots,m} \rho(A_{t_k,N}^M)<\half$, so that $\rho(D_\Gamma)<\half$ by  Theorem \ref{thm:spectral_radius2}. Figure \ref{fig:plot5}(b) plots an approximation $\sigma^N(D_\Gamma)$ to $\sigma(D_\Gamma)$ given by Theorem \ref{thm:Hausdorff_convergence_2}, which is contained in the circle of radius $\half$ (in red). By contrast, by \eqref{eq:inclusion}, and arguing as in Example 1, the numerical range $W(D_\Gamma)$ contains at least the closed disc of radius $R^*\approx 0.8179$ (shown in blue in Figure \ref{fig:plot5}(b)), where $R^*$ is as given in Corollary \ref{cor:num_range2}.

\begin{figure}[h!]
\centering
\begin{subfigure}[t]{.44\textwidth}
  \centering
  \includegraphics[scale = 0.49, trim = 2cm 0cm 0cm 0cm]{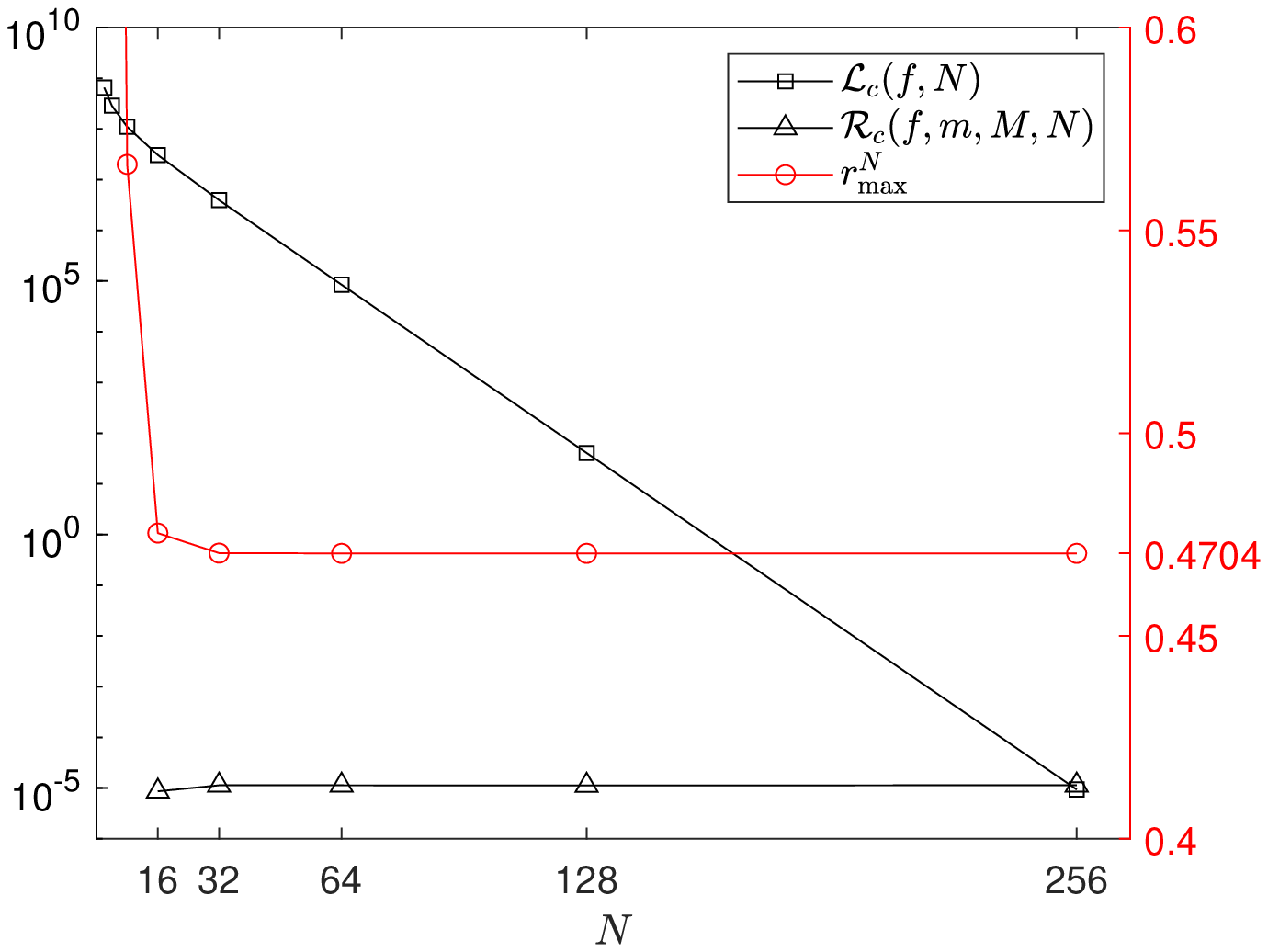}
  \subcaption{$\cL_c(f,N)$, $\cR_c(f,m,M,N)$, for $\rho_0=\half$, and the maximum of the spectral radii of $A_{t_k,N}^M$, $k=1,\ldots,m$, plotted against $N$ for $m=10,000$ and $M=60$. $\cR_c(f,m,M,N)$ values are not shown for $N=2,4,$ and $8$ because they are negative.}
\end{subfigure}
\hspace{0.5cm}
\begin{subfigure}[t]{.51\textwidth}
  \centering
  \includegraphics[scale = 0.53, trim = 2cm 0cm 0cm 0cm]{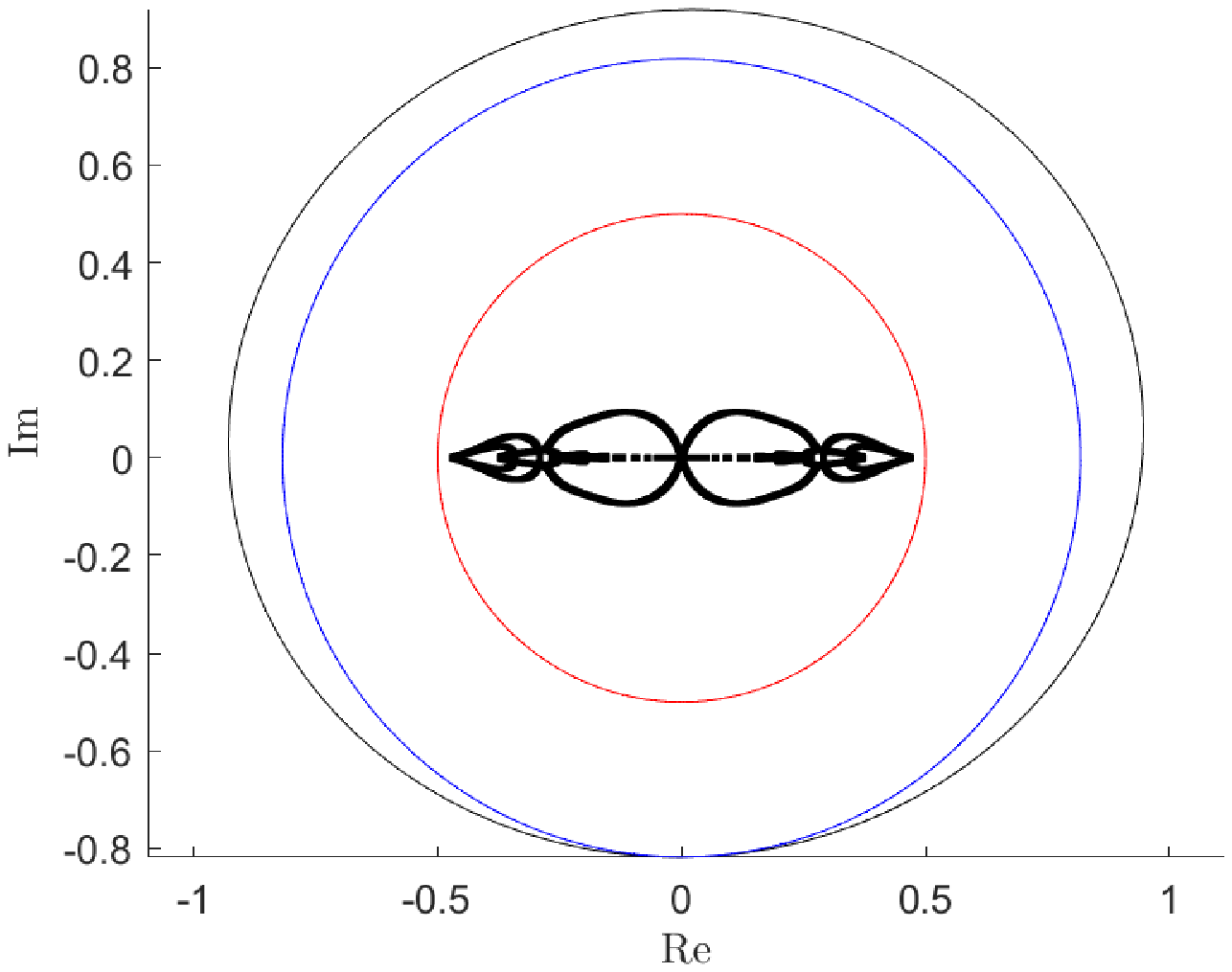}
  \subcaption{$\sigma^N(D_\Gamma)=\{0\}\cup \bigcup_{t\in t_N} \spec(A_{t,N}^{M_N})$, for $N = 256$, $m_N = 10,000$, $M_N=60$, an approximation to $\sigma(D_\Gamma)=\sigma_{\ess}(D_\Gamma)$ by Theorem \ref{thm:Hausdorff_convergence_2}. By Theorem \ref{thm:synth} this is also an approximation to $\sigma_{\ess}(D_\Gamma)$ for $\Gamma\in \cD_\cA$ as in Figure \ref{fig:exLDI}(b).  Also shown are the boundary of $W_{100}$ (for $p = 10$, $N = 256$, $M = 60$ and $t = \pi/13$) in black, $R^*\T \subset W_{100}\subset W(D_{\Gamma})$ in blue (see Corollary 4.25), $\half\T$ in red.
}
\end{subfigure}
\caption{Numerical results for Example 3; $f$ given by \eqref{eq:ftwoside}, $\alpha = \frac{2}{3}$. \label{fig:plot5}}
\end{figure}

Figure \ref{fig:plot5} is also, by Theorem \ref{thm:synth}, relevant to the bounded Lipschitz domain $\Omega_-$ shown in Figure \ref{fig:exLDI}(b). This has boundary $\Gamma\in \cD_\cA$ which is $C^1$ except at the single point $x=0$ where it coincides locally with the graph of $f$ given by \eqref{eq:ftwoside}. It follows by the above calculations and Theorem \ref{thm:synth} that $\rho_{\ess}(D_\Gamma)<\half$ for $\Gamma=\partial \Omega_-$ and that Figure  \ref{fig:plot5}(b) is also a plot of the approximation $\Sigma^N(D_\Gamma)$, defined in Theorem \ref{thm:synth}, to the essential spectrum of $D_\Gamma$ for $\Gamma=\partial \Omega_-$. While $\rho_{\ess}(D_\Gamma)<\half$,  $W_{\ess}(D_\Gamma)$, for $\Gamma=\partial \Omega_-$, contains the closed disc of radius $R^*\approx 0.8179$ shown in Figure  \ref{fig:plot5}(b), by the above result for the graph $\Gamma$ given by \eqref{eq:ftwoside}, and a localisation result \cite[Theorem 3.2]{ChaSpe:21} (and see \cite[Theorem 3.17]{ChaSpe:21}) for the essential numerical range, analogous to Theorem \ref{thm:localization}. Thus also, by \eqref{eq:NormBounds}, $\|D_\Gamma\|_{\ess}\geq w_{\ess}(D_\Gamma)\geq R^*$.\\

\noindent {\bf Example 4.} In this example we define $f$ by \eqref{eq:fdef} for $x>0$ and set $f(x):=0$ for $x\leq 0$, so that $\Gamma$ and $f$ satisfy \eqref{eq:standing2}. Recalling \eqref{eq:gdef2}, we see that $f_+$ and $g_+$ are the same as $f$ and $g$ in Example 1,  while $f_-=0$ and $g_-=0$. The graph $\Gamma$ has Lipschitz character $L= (f_{\max}^\prime-f_{\min}^\prime)/2 = f_{\max}^\prime - 1/2$, where $f_{\max}^\prime$ and $f_{\min}^\prime$ are as defined in Example 1. The conditions \eqref{c:bound2_2} reduce in this example to the same conditions as in Example 3, and again we choose  $\alpha=\frac{2}{3}$ and $c=0.019$, so that $\Gamma$ has Lipschitz character  $L=f_{\max}^\prime -1/2\approx 7.76$. Similarly to the previous example, we see from Figure \ref{fig:plot6}(a) that, with $M = 50$, $m = 5,000$, and $N = 256$, the inequality \eqref{eq:estimate_to_check2} is satisfied and $\max_{k=1,\ldots,m} \rho(A_{t_k,N}^M)<\half$, so that $\rho(D_\Gamma)<\half$ by  Theorem \ref{thm:spectral_radius2}. The approximation $\sigma^N(D_\Gamma)$ to $\sigma(D_\Gamma)$, given by Theorem \ref{thm:Hausdorff_convergence_2} and plotted in Figure \ref{fig:plot6}(b), is contained in the circle of radius $\half$, while the numerical range $W(D_\Gamma)$ contains at least the closed disc of radius $R^*\approx 0.8179$ shown in Figure \ref{fig:plot6}(b).

Figure \ref{fig:plot6} is also, by Theorem \ref{thm:synth}, relevant to the bounded Lipschitz domain $\Omega_-$ shown in Figure \ref{fig:exLDI}(a). This has boundary $\Gamma\in \cD_\cA$ which is $C^1$ except at the single point $x=0$ where it coincides locally with the graph of the function $f$ described above. It follows by the above calculations and Theorem \ref{thm:synth} that $\rho_{\ess}(D_\Gamma)<\half$ for $\Gamma=\partial \Omega_-$ and that Figure  \ref{fig:plot5}(b) is also a plot of the approximation $\Sigma^N(D_\Gamma)$, defined in Theorem \ref{thm:synth}, to the essential spectrum of $D_\Gamma$ for $\Gamma=\partial \Omega_-$. Arguing as in the previous example, while $\rho_{\ess}(D_\Gamma)<\half$,  $W_{\ess}(D_\Gamma)$, for $\Gamma=\partial \Omega_-$, contains the closed disc of radius $R^*\approx 0.8179$ plotted in Figure  \ref{fig:plot6}(b), and $\|D_\Gamma\|_{\ess}\geq w_{\ess}(D_\Gamma)\geq R^*$. \\

\begin{rem}[\bf Symmetry of the spectrum and essential spectrum] For Example 2 it is immediate from \eqref{eq:SpecCone} that $\sigma(D_\Gamma)$ is symmetric with respect to the origin; if $z\in \sigma(D_\Gamma)$ then $-z\in \sigma(D_\Gamma)$. We conjecture, based on the numerical results for Examples 3 and 4 and similar calculations, that this same symmetry holds in 2D whenever $\Gamma$ is a dilation invariant graph satisfying \eqref{eq:standing2}. If this conjecture is true, then, for all $\Gamma\in \cD_\cA$, $\sigma_{\ess}(D_\Gamma)$ is symmetric with respect to the origin by Theorem \ref{thm:synth}. We note that symmetry results in the 2D case of this sort are proved for $D_\Gamma:L^2(\Gamma)\to L^2(\Gamma)$ when $\Gamma$ is the boundary of a bounded $C^2$ domain $\Omega_-$ in \cite[Proposition 6]{KhPuSh07}, and for $D^\prime_\Gamma$ as an operator on the natural energy space for general Lipschitz $\Omega_-$ \cite[Theorem 2.1]{HeKaLi17} (in both these cases the relevant spectrum lies in $[-1,1]$). In 3D, even when $\Gamma$ is a polyhedron, the spectrum and essential spectrum of $D_\Gamma$ need not be symmetric with respect to the origin (see, e.g., \cite{LeCoPer:22}).
\end{rem}
\begin{figure}[h!]
\centering
\begin{subfigure}[t]{.44\textwidth}
  \centering
  \includegraphics[scale = 0.49, trim = 2cm 0cm 0cm 0cm]{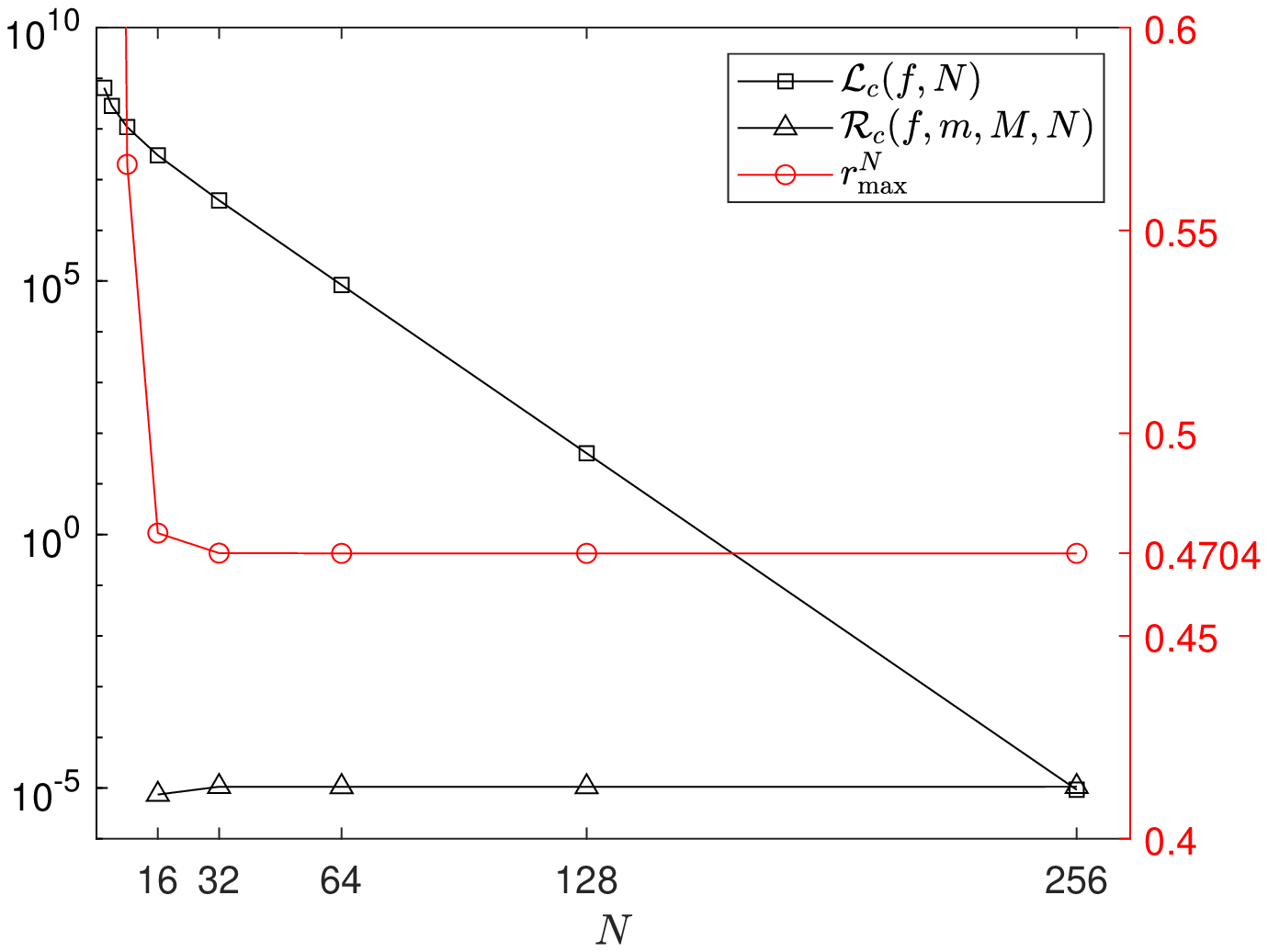}
  \subcaption{$\cL_c(f,N)$, $\cR_c(f,m,M,N)$, for $\rho_0=\half$, and the maximum of the spectral radii of $A_{t_k,N}^M$, $k=1,\ldots,m$, plotted against $N$ for $m=5000$ and $M=50$. $\cR_c(f,m,M,N)$ values are not shown for $N=2,4$ and $8$ because they are negative.}
\end{subfigure}
\hspace{0.5cm}
\begin{subfigure}[t]{.51\textwidth}
  \centering
  \includegraphics[scale = 0.53, trim = 2cm 0cm 0cm 0cm]{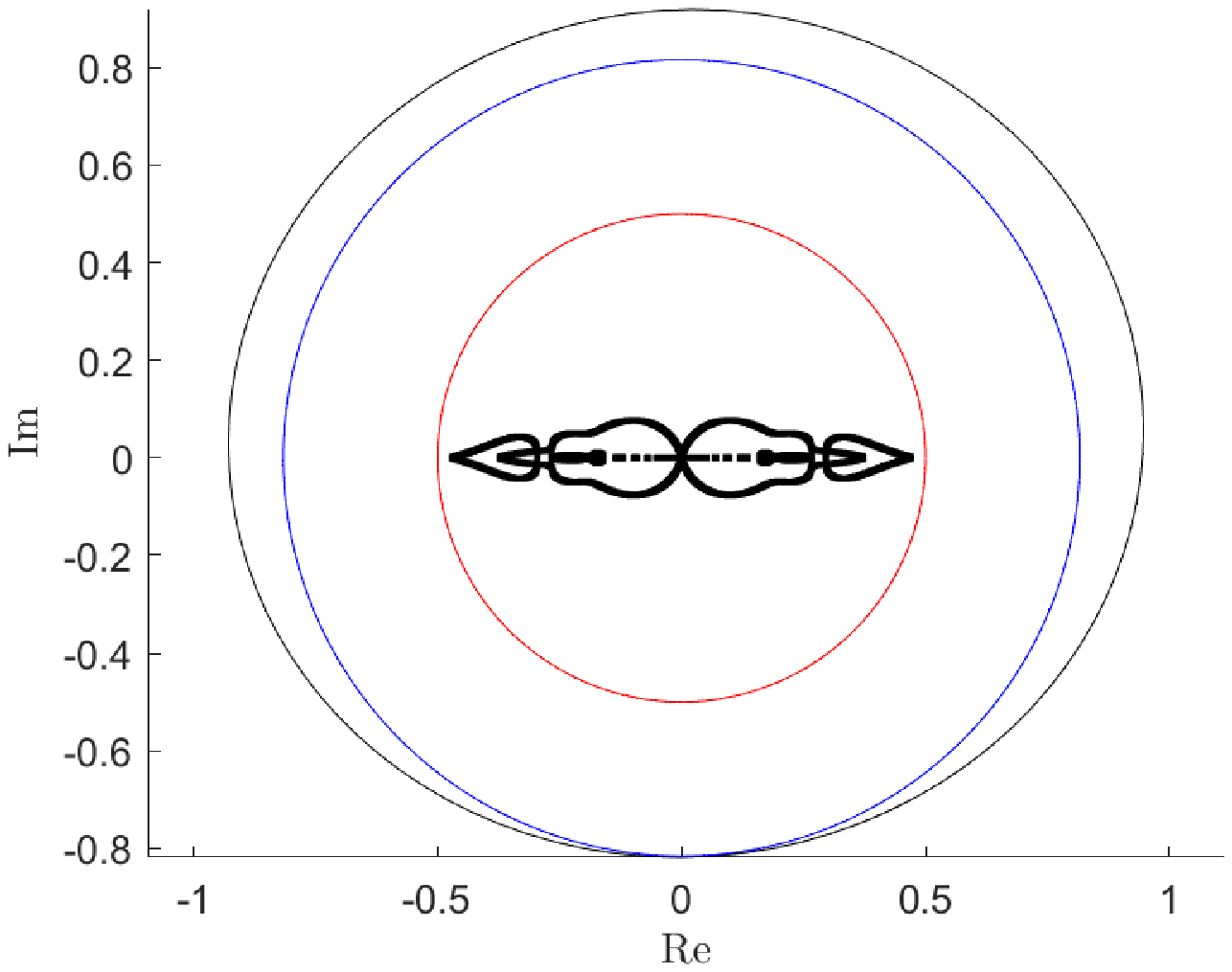}
  \subcaption{$\sigma^N(D_\Gamma)=\{0\}\cup \bigcup_{t\in t_N} \spec(A_{t,N}^{M_N})$, for $N = 256$, $m_N = 5000$, $M_N=50$, an approximation to $\sigma(D_\Gamma)=\sigma_{\ess}(D_\Gamma)$ by Theorem \ref{thm:Hausdorff_convergence_2}. By Theorem \ref{thm:synth} this is also an approximation to $\sigma_{\ess}(D_\Gamma)$ for $\Gamma\in \cD_\cA$ as in Figure \ref{fig:exLDI}(a).  Also shown are the boundary of $W_{100}$ (for $p = 10$, $N = 256$, $M = 50$, and $t = \pi/13$) in black, $R^*\T \subset W_{100}\subset W(D_{\Gamma})$ in blue (see Corollary 4.25), $\half\T$ in red.
}
\end{subfigure}
\caption{Numerical results for Example 4; $f_+=f$ given by \eqref{eq:fdef}, $f_- = 0$, $\alpha = \frac{2}{3}$ \label{fig:plot6}}
\end{figure}

\small

\noindent {\bf Acknowledgements.}
We are grateful to Marko Lindner (Hamburg University of Technology), Marco Marletta (Cardiff), and Euan Spence (Bath) for discussions regarding this project, and to the anonymous referee for their careful reading of the manuscript and associated feedback. The 2nd and 4th authors were supported by the European Union's Horizon 2020 research and innovation programme under the Marie Sklodowska-Curie grant agreement No 844451 and by Engineering and Physical Sciences Research Council (EPSRC) Grant EP/T008636/1.


\begin{thebibliography}{10}

\bibitem{AmDeMi:16}
{\sc H.~Ammari, Y.~Deng, and P.~Millien}, {\em Surface plasmon resonance of
  nanoparticles and applications in imaging}, Arch. Rational Mech. Anal., 220
  (2016), pp.~109--153.

\bibitem{An:71}
{\sc P.~M. Anselone}, {\em Collectively compact operator approximation theory
  and applications to integral equations}, Prentice-Hall Englewood Cliffs, NJ,
  1971.

\bibitem{At:75}
{\sc K.~Atkinson}, {\em Convergence rates for approximate eigenvalues of
  compact integral operators}, SIAM J. Numer. Anal., 12 (1975), pp.~213--222.

\bibitem{AxKeMc:06}
{\sc A.~Axelsson, S.~Keith, and A.~McIntosh}, {\em Quadratic estimates and
  functional calculi of perturbed {D}irac operators}, Invent. Math., 163
  (2006), pp.~455--497.

\bibitem{ArMaRo22}
{\sc J.~Ben-Artzi, M.~Marletta, and F.~R\"osler}, {\em Universal algorithms for
  computing spectra of periodic operators}, Numer. Math., 150 (2022),
  pp.~719--767.

\bibitem{bonsall1973numerical}
{\sc F.~F. Bonsall and J.~Duncan}, {\em Numerical Ranges II}, Cambridge
  University Press, 1973.

\bibitem{BGS66}
{\sc J.~D. Burago, V.~G. Maz'ja, and V.~D. Sapo\v{z}nikova}, {\em On the theory
  of potentials of a double and a simple layer for regions with irregular
  boundaries}, in Problems {M}ath. {A}nal. {B}oundary {V}alue {P}roblems
  {I}ntegr. {E}quations ({R}ussian), Izdat. Leningrad. Univ., Leningrad, 1966,
  pp.~3--34.

\bibitem{BG67}
{\sc J.~S. Burago and V.~G. Maz'ja}, {\em Certain questions of potential theory
  and function theory for regions with irregular boundaries}, Zap. Nau\v{c}n.
  Sem. Leningrad. Otdel. Mat. Inst. Steklov. (LOMI), 3 (1967), p.~152.

\bibitem{Cald77}
{\sc A.-P. Calder\'{o}n}, {\em Cauchy integrals on {L}ipschitz curves and
  related operators}, Proc. Nat. Acad. Sci. U.S.A., 74 (1977), pp.~1324--1327.

\bibitem{Ch:84}
{\sc G.~A. Chandler}, {\em Galerkin's method for boundary integral equations on
  polygonal domains}, The ANZIAM Journal, 26 (1984), pp.~1--13.

\bibitem{ChaSpe:21}
{\sc S.~N. Chandler-Wilde and E.~A. Spence}, {\em {Coercivity, essential norms,
  and the Galerkin method for second-kind integral equations on polyhedral and
  Lipschitz domains}}, Numer. Math., 150 (2022), pp.~299--371. See also the correction submitted to Numer. Math., at \href{https://www.personal.reading.ac.uk/~sms03snc/GCC.pdf}{https://www.personal.reading.ac.uk/$\sim$sms03snc/GCC.pdf}

\bibitem{ChLe08}
{\sc T.~Chang and K.~Lee}, {\em Spectral properties of the layer potentials on
  {L}ipschitz domains}, Illinois Journal of Mathematics, 52 (2008),
  pp.~463--472.

\bibitem{CMM82}
{\sc R.~R. Coifman, A.~McIntosh, and Y.~Meyer}, {\em L'int\'{e}grale de
  {C}auchy d\'{e}finit un op\'{e}rateur born\'{e} sur {$L^{2}$} pour les
  courbes lipschitziennes}, Ann. of Math. (2), 116 (1982), pp.~361--387.

\bibitem{BDav}
{\sc E.~B. Davies}, {\em Linear operators and their spectra}, Cambridge
  University Press, 2007.

\bibitem{LeCoPer:22}
{\sc M.~{d}e Le\'on-Contreras and K.-M. Perfekt}, {\em The quasi-static
  plasmonic problem for polyhedra}, Math. Ann.,  (2022).
\newblock https://doi.org/10.1007/s00208-022-02481-x.

\bibitem{El:92}
{\sc J.~Elschner}, {\em {The double layer potential operator over polyhedral
  domains I: Solvability in weighted Sobolev spaces}}, Applicable Analysis, 45
  (1992), pp.~117--134.

\bibitem{EvansGariepy92}
{\sc L.~C. Evans and R.~F. Gariepy}, {\em Measure theory and fine properties of
  functions}, Studies in Advanced Mathematics, CRC Press, Boca Raton, FL, 1992.

\bibitem{FaSaSe:92}
{\sc E.~Fabes, M.~Sand, and J.~K. Seo}, {\em The spectral radius of the
  classical layer potentials on convex domains}, in Partial differential
  equations with minimal smoothness and applications, Springer, 1992,
  pp.~129--137.

\bibitem{FaJoRi:78}
{\sc E.~B. Fabes, M.~Jodeit, and N.~M. Riviere}, {\em {Potential techniques for
  boundary value problems on $C^1$ domains}}, Acta Mathematica, 141 (1978),
  pp.~165--186.



\bibitem{GM13}
{\sc N.~V. Grachev and V.~G. Maz'ya}, {\em Solvability of a boundary integral
  equation on a polyhedron}, Journal of Mathematical Sciences, 191 (2013),
  pp.~193--213.

\bibitem{Gr:85}
{\sc P.~Grisvard}, {\em Elliptic problems in nonsmooth domains}, Pitman,
  Boston, 1985.

\bibitem{GuRa:97}
{\sc K.~E. Gustafson and D.~K.~M. Rao}, {\em Numerical range; The field of
  values of linear operators and matrices}, Universitext, Springer-Verlag, New
  York, 1997.

\bibitem{Ha:95}
{\sc W.~Hackbusch}, {\em Integral Equations: Theory and Numerical Treatment},
  Birkh\"auser
  Verlag, 1995.

\bibitem{HaRoSi01}
{\sc R.~Hagen, S.~Roch, and B.~Silbermann}, {\em $C^*$-Algebras and Numerical
  Analysis}, Marcel Dekker, 2001.

\bibitem{HeKaLi17}
{\sc J.~Helsing, H.~Kang, and M.~Lim}, {\em Classification of spectra of the
  {N}eumann-{P}oincar\'e operator on planar domains with corners by resonance},
  Ann. I. H. Poincar\'e C, 34 (2017), pp.~991--1011.

\bibitem{HP13}
{\sc J.~Helsing and K.-M. Perfekt}, {\em On the polarizability and capacitance
  of the cube}, Appl. Comput. Harmon. Anal., 34 (2013), pp.~445--468.

\bibitem{HePe:17}
\leavevmode\vrule height 2pt depth -1.6pt width 23pt, {\em The spectra of
  harmonic layer potential operators on domains with rotationally symmetric
  conical points}, J. Math. Pures Appl. (9), 118 (2018), pp.~235--287.

\bibitem{Hof94}
{\sc S.~Hofmann}, {\em On singular integrals of {C}alder\'{o}n-type in {${\bf
  R}^n$}, and {BMO}}, Rev. Mat. Iberoamericana, 10 (1994), pp.~467--505.

\bibitem{HMT07}
{\sc S.~Hofmann, M.~Mitrea, and M.~Taylor}, {\em Geometric and transformational
  properties of {L}ipschitz domains, {S}emmes-{K}enig-{T}oro domains, and other
  classes of finite perimeter domains}, J. Geom. Anal., 17 (2007),
  pp.~593--647.

\bibitem{HMT10}
\leavevmode\vrule height 2pt depth -1.6pt width 23pt, {\em Singular integrals
  and elliptic boundary problems on regular {S}emmes-{K}enig-{T}oro domains},
  Int. Math. Res. Not. IMRN,  (2010), pp.~2567--2865.

\bibitem{Jo78}
{\sc C.~R. Johnson}, {\em Numerical determination of the field of values of a
  general complex matrix}, SIAM J. Numer. Anal., 15 (1978), pp.~595--602.

\bibitem{Jo:82}
{\sc K.~J{\"o}rgens}, {\em Linear integral operators}, Pitman Advanced Pub.
  Program, 1982.

\bibitem{Kellogg29}
{\sc O.~D. Kellogg}, {\em Foundations of Potential Theory}, Springer, 1929.

\bibitem{Ken}
{\sc C.~E. Kenig}, {\em {Harmonic analysis techniques for second order elliptic
  boundary value problems}}, American Mathematical Society, 1994.

\bibitem{KhPuSh07}
{\sc D.~Khavinson, M.~Putinar, and H.~S. Shapiro}, {\em Poincar\'e's
  variational problem in potential theory}, Arch. Rational Mech. Anal., 185
  (2007), pp.~143--184.

\bibitem{Kral62}
{\sc J.~Kr\'{a}l}, {\em On the logarithmic potential}, Comment. Math. Univ.
  Carolinae, 3 (1962), pp.~3--10.

\bibitem{Kral64}
\leavevmode\vrule height 2pt depth -1.6pt width 23pt, {\em On the logarithmic
  potential of the double distribution}, Czechoslovak Math. J., 14(89) (1964),
  pp.~306--321.

\bibitem{Kral65}
\leavevmode\vrule height 2pt depth -1.6pt width 23pt, {\em The {F}redholm
  radius of an operator in potential theory}, Czechoslovak Math. J., 15(90)
  (1965), pp.~454--473; ibid. 15 (90), (1965), 565--588.

\bibitem{Kr:98}
{\sc R.~Kress}, {\em Numerical Analysis}, Springer, 1998.

\bibitem{Kr:14}
\leavevmode\vrule height 2pt depth -1.6pt width 23pt, {\em Linear Integral
  Equations}, Springer-Verlag, 3rd~ed., 2014.

\bibitem{Li:06}
{\sc M.~Lindner}, {\em Infinite matrices and their finite sections: an
  introduction to the limit operator method}, Birkh\"auser, 2006.

\bibitem{Med90}
{\sc D.~Medkov\'{a}}, {\em Invariance of the Fredholm radius of the Neumann
  operator}, \v{C}asopis pro p\v{e}stov\'{a}n\'{i} matematiky, 115 (1990),
  pp.~147--164.

\bibitem{Med92}
\leavevmode\vrule height 2pt depth -1.6pt width 23pt, {\em On essential norm of
  the {N}eumann operator}, Math. Bohem., 117 (1992), pp.~393--408.

\bibitem{Med97}
\leavevmode\vrule height 2pt depth -1.6pt width 23pt, {\em The third boundary
  value problem in potential theory for domains with a piecewise smooth
  boundary}, Czechoslovak Math. J., 47(122) (1997), pp.~651--679.

\bibitem{Me:18}
\leavevmode\vrule height 2pt depth -1.6pt width 23pt, {\em The {L}aplace
  Equation: Boundary Value Problems on Bounded and Unbounded Lipschitz
  Domains}, Springer, 2018.

\bibitem{MitreaD:97}
{\sc D.~Mitrea}, {\em The method of layer potentials for non-smooth domains
  with arbitrary topology}, Integr. Equ. Oper. Theory, 29 (1997), pp.~320--338.

\bibitem{Mi:99}
{\sc I.~Mitrea}, {\em Spectral radius properties for layer potentials
  associated with the elastostatics and hydrostatics equations in nonsmooth
  domains}, J. Fourier Anal. Appl., 5 (1999),
  pp.~385--408.

\bibitem{Mi:02}
\leavevmode\vrule height 2pt depth -1.6pt width 23pt, {\em On the spectra of
  elastostatic and hydrostatic layer potentials on curvilinear polygons},
  J. Fourier Anal. Appl., 8 (2002), pp.~443--488.

\bibitem{MiOtTu:17}
{\sc I.~Mitrea, K.~Ott, and W.~Tucker}, {\em Invertibility properties of
  singular integral operators associated with the {L}am\'e and {S}tokes systems
  on infinite sectors in two dimensions}, Integral Equ. Oper. Theory, 89
  (2017), pp.~151--207.

\bibitem{Net74}
{\sc I.~Netuka}, {\em Double layer potentials and the {D}irichlet problem},
  Czechoslovak Math. J., 24(99) (1974), pp.~59--73.

\bibitem{Pe:19}
{\sc K.-M. Perfekt}, {\em The transmission problem on a three-dimensional
  wedge}, Arch. Ration. Mech. Anal., 231 (2019), pp.~1745--1780.

\bibitem{QiNi:12}
{\sc Y.~Qiao and V.~Nistor}, {\em Single and double layer potentials on domains
  with conical points {I}: {S}traight cones}, Integral Equations Operator
  Theory, 72 (2012), pp.~419--448.

\bibitem{RaRoSi04}
{\sc V.~S. Rabinovich, S.~Roch, and B.~Silbermann}, {\em Limit operators and
  their applications in operator theory}, Birkh\"auser, 2004.

\bibitem{Rad19}
{\sc J.~Radon}, {\em {\"U}ber die randwertaufgaben beim logarithmischen
  potential}, Sitzber. Akad. Wiss. Wien, 128 (1919), pp.~1123--1167.

\bibitem{Ra:92}
{\sc A.~Rathsfeld}, {\em The invertibility of the double layer potential
  operator in the space of continuous functions defined on a polyhedron: The
  panel method}, Applicable Analysis, 45 (1992), pp.~135--177.

\bibitem{Ra:95}
\leavevmode\vrule height 2pt depth -1.6pt width 23pt, {\em The invertibility of
  the double layer potential operator in the space of continuous functions
  defined over a polyhedron. {T}he panel method. {E}rratum}, Applicable
  Analysis, 56 (1995), pp.~109--115.

\bibitem{Rud2}
{\sc W.~Rudin}, {\em Real and Complex Analysis}, McGraw-Hill New York, 3rd~ed.,
  1987.

\bibitem{Schnitzer:20}
{\sc O.~Schnitzer}, {\em Asymptotic approximations for the plasmon resonances
  of nearly touching spheres}, European Journal of Applied Mathematics, 31
  (2020), pp.~246--276.

\bibitem{Sh:69}
{\sc V.~Y. Shelepov}, {\em On the index of an integral operator of potential
  type in the space {$L_p$}}, Soviet Math. Dokl., 10 (1969), pp.~754--757.

\bibitem{Sh:91}
\leavevmode\vrule height 2pt depth -1.6pt width 23pt, {\em On the index and
  spectrum of integral operators of potential type along radon curves},
  Mathematics of the USSR-Sbornik, 70 (1991), pp.~175--203.

\bibitem{St:08}
{\sc O.~Steinbach}, {\em Numerical Approximation Methods for Elliptic Boundary
  Value Problems: Finite and Boundary Elements}, Springer, New York, 2008.

\bibitem{TrWe14}
{\sc L.~N. Trefethen and J.~A.~C. Weideman}, {\em The exponentially convergent
  trapezoidal rule}, SIAM Rev., 56 (2014), pp.~385--458.

\bibitem{Ver84}
{\sc G.~Verchota}, {\em {Layer potentials and regularity for the Dirichlet
  problem for Laplace's equation in Lipschitz domains}}, Journal of Functional
  Analysis, 59 (1984), pp.~572--611.

\bibitem{We:09}
{\sc W.~L. Wendland}, {\em {On the Double Layer Potential}}, in Analysis,
  Partial Differential Equations and Applications, A.~Cialdea, P.~E. Ricci, and
  F.~Lanzara, eds., Springer, 2009, pp.~319--334.

\end{thebibliography}

\end{document}